\documentclass{article}
\usepackage{amssymb,amsmath,amsthm}%ashley
\usepackage{graphicx}
\usepackage{subfigure}
\usepackage{makeidx}
\usepackage{multicol}

\usepackage{color}
\usepackage{enumerate,amssymb}
\usepackage[arrow,curve,matrix,tips,2cell]{xy}
  \SelectTips{eu}{10} \UseTips
  \UseAllTwocells

\theoremstyle{plain}
\newtheorem{theorem}{Theorem}[section]

\newtheorem{proposition}[theorem]{Proposition}
\newtheorem{corollary}[theorem]{Corollary}
\newtheorem{conjecture}[theorem]{Conjecture}
\newtheorem{problem}[theorem]{Problem}

\theoremstyle{definition}

\newtheorem{definition}[theorem]{Definition}
\newtheorem{example}[theorem]{Example}

\newtheorem{remark}[theorem]{Remark}

\DeclareMathAlphabet\EuR{U}{eur}{m}{n}
\SetMathAlphabet\EuR{bold}{U}{eur}{b}{n}
\newcommand{\curs}{\EuR}

\newcommand{\Sets}{\curs{Sets}}

\newcommand{\co}{\colon\thinspace}
\newcommand{\mb}[1]{\mathbb{#1}}
\newcommand{\ms}[1]{\mathbf{#1}}

\newcommand{\xx}{\overline{\xi}}

\newcommand{\op}{{op}}

\newcommand{\too}{\xrightarrow}

\newcommand{\Assoc}{\mathcal{A}ssoc}
\newcommand{\Comm}{\mathcal{C}omm}
\newcommand{\cC}{\curs{C}}
\newcommand{\cD}{\curs{D}}
\newcommand{\Sp}{\curs{Sp}}
\newcommand{\Ab}{\curs{Ab}}
\newcommand{\Alg}{\curs{Alg}}
\newcommand{\LMod}{\curs{LMod}}
\newcommand{\RMod}{\curs{RMod}}
\newcommand{\Mod}{\curs{Mod}}
\newcommand{\Cat}{\curs{Cat}}
\newcommand{\Spaces}{\curs{S}} 
\newcommand{\cM}{\curs{M}}
\newcommand{\cN}{\curs{N}}

\newcommand{\oO}{\mathcal{O}}

\newcommand{\hash}{\mathbin{\#}}

\newcommand{\pro}{{\rm pro}\mbox{-}}

\DeclareMathOperator{\Ext}{Ext}
\DeclareMathOperator{\Hom}{Hom}
\DeclareMathOperator{\Map}{Map}
\DeclareMathOperator{\Tor}{Tor}
\DeclareMathOperator{\End}{End}

\DeclareMathOperator{\Fun}{Fun}

\DeclareMathOperator{\Sq}{Sq}
\DeclareMathOperator{\Pic}{Pic}

\DeclareMathOperator{\Op}{Op}

\DeclareMathOperator{\Free}{Free}
\DeclareMathOperator{\Mul}{Mul}
\DeclareMathOperator{\Sym}{Sym}
\DeclareMathOperator{\ev}{ev}
\DeclareMathOperator{\Pow}{Pow}

\DeclareMathOperator*{\colim}{colim}
\DeclareMathOperator*{\hocolim}{hocolim}

\DeclareMathOperator{\THH}{THH}
\DeclareMathOperator{\TAQ}{TAQ}

\DeclareMathOperator*{\into}{\hookrightarrow}

\title{$E_n$ ring spectra and \index{Dyer--Lashof
    operations}Dyer--Lashof operations}
\author{Tyler Lawson}
\date{}

\index{Dyer--Lashof operations|(}
\begin{document}

\maketitle

\section{Introduction}
\label{sec:introduction}

\index{cohomology operations}Cohomology operations are absolutely essential in making cohomology an
effective tool for studying spaces. In particular, the mod-$p$
cohomology groups of a space $X$ are enhanced with a binary cup
product, a Bockstein derivation, and Steenrod's reduced \index{power operations}power
operations; these satisfy relations such as graded-commutativity, the
\index{Cartan formula}Cartan formula, the \index{Adem relations}Adem relations, and the instability relations
\cite{steenrod-epstein}. The combined structure of these cohomology
operations is very effective in homotopy theory because of three
critical properties.
\begin{description}
\item[These operations are natural.] We can exclude the possibility of
  certain maps between spaces because they would not respect these
  operations.
\item[These operations are constrained.] We can exclude the existence
  of certain spaces because the cup product and power operations would
  be incompatible with the relations that must hold.
\item[These operations are complete.] Because cohomology is
  \emph{\index{representability!cohomology operations}representable},
  we can determine all possible
  natural operations which take an $n$-tuple of cohomology elements
  and produce a new one. All operations are built, via composition,
  from these basic operations. All relations between these operations
  are similarly built from these basic relations.
\end{description}
In particular, this last property makes the theory reversible: there
are mechanisms which take cohomology as input and converge to
essentially complete information about homotopy theory in many useful
cases, with the principal examples being the stable and unstable Adams
spectral sequences. The stable Adams spectral sequence begins with the
Ext-groups $\Ext(H^*(Y), H^*(X))$ in the category of modules with
Steenrod operations and converges to the stable classes of maps from
$X$ to a $p$-completion of $Y$ \cite{adams-stablehomotopy}. The
unstable \index{spectral sequence!Adams}Adams spectral sequence is
similar, but it begins with \index{derived functors!nonabelian}nonabelian Ext-groups that are calculated
in the category of graded-commutative rings with Steenrod operations
\cite{bousfield-kan-unstableadams, bckqrs-unstableadams}.

Our goal is to discuss multiplicative homotopy theory: spaces,
categories, or spectra with extra multiplicative structure. In this
situation, we will see that the \emph{Dyer--Lashof operations} play
the role that the Steenrod operations did in ordinary homotopy
theory.

In ordinary algebra, commutativity is an extremely useful
\emph{property} possessed by certain monoids and algebras. This is no
longer the case in multiplicative homotopy theory or category
theory. In category theory, commutativity becomes \emph{structure}: to
give symmetry to a monoidal category $\cC$ we must make a choice of a
natural twist isomorphism $\tau\co A \otimes B \to B \otimes
A$. Moreover, there are more degrees of symmetry possible than in
algebra because we can ask for weaker or stronger identities on
$\tau$. By asking for basic identities to hold we obtain the
notion of a \index{braided monoidal category}braided monoidal category, and by asking for very strong
identities to hold we obtain the notion of a symmetric monoidal
category.  In homotopy theory and higher category theory we rarely
have the luxury of imposing identities, and these become replaced by
extra structure. One consequence is that there are many degrees of
commutativity, parametrized by \index{operad}operads.

The most classical such structures arose geometrically in the study of
\index{iterated loop space}iterated loop spaces. For a pointed space $X$, the $n$-fold loop space
$\Omega^n X$ has algebraic operations parametrized by certain
\index{configuration space}configuration spaces $\mathcal{E}_n(k)$, which assemble into an
\emph{$E_n$-operad}; moreover, there is a converse theorem due to
\index{Boardman}Boardman--\index{Vogt}Vogt and \index{May}May that provides a \index{recognition principle}recognition principle for what
structure on $Y$ is needed to express it as an iterated loop space. As
$n$ grows, these spaces possess more and more commutativity, reflected
algebraically in extra \index{Dyer--Lashof operations!loop spaces}Dyer--Lashof operations on the homology $H_* Y$
that are analogous to the Steenrod operations.

In recent years there is an expanding library of examples of
\index{ring spectrum}ring spectra that only admit, or only naturally
admit, these intermediate levels of structure between associativity
and commutativity. Our goal in this chapter is to give an outline of
the modern theory of highly structured ring spectra, particularly
$E_n$ ring spectra, and to give a toolkit for their study. One of the
things that we would like to emphasize is how to usefully work in this
setting, and so we will discuss useful tools that are imparted by
$E_n$ ring structures, such as operations on them that unify the study
of Steenrod and Dyer--Lashof operations. We will also introduce the next
stage of structure in the form of secondary operations. Throughout, we
will make use of these operations to show that structured ring spectra
are heavily constrained, and that many examples do not admit this
structure; we will in particular discuss our proof in \cite{secondary}
that the $2$-primary \index{Brown--Peterson spectrum}Brown--Peterson
spectrum does not admit the structure of an $E_\infty$ \index{ring
  spectrum}ring spectrum, answering an old question of May
\cite{may-problems}. At the close we will discuss some ongoing
directions of study.

\section{Acknowledgements}

The author would like to thank
Andrew Baker,
Tobias Barthel,
Clark Barwick,
Robert Bruner,
David Gepner,
Saul Glasman,
Gijs Heuts,
Nick Kuhn,
Michael Mandell,
Akhil Mathew,
Lennart Meier,
and
Steffen Sagave
for discussions related to this material, and Haynes Miller for a
careful reading of an earlier draft of this paper. The author also owes
a significant long-term debt to Charles Rezk for what understanding he
possesses.

The author was partially supported by NSF grant 1610408 and a grant
from the Simons Foundation. The author would like to thank the Isaac
Newton Institute for Mathematical Sciences for support and hospitality
during the programme HHH when work on this paper was undertaken. This
work was supported by: EPSRC grant numbers EP/K032208/1 and
EP/R014604/1.

\section{Operads and algebras}
\label{sec:algebras}
\index{operad}
\index{algebra!over an operad}

Throughout this section, we will let $\cC$ be a fixed \index{symmetric monoidal topological category}symmetric
monoidal topological category. For us, this means that $\cC$ is
enriched in the category $\Spaces$ of spaces, that there is a functor
$\otimes\co \cC \times \cC \to \cC$ of enriched categories, and that
the underlying functor of ordinary categories is extended to a
symmetric monoidal structure. We will write $\index{$\Map$}\Map_{\cC}(X,Y)$ for the
mapping space between two objects, and $\index{$\Hom$}\Hom_{\cC}(X,Y)$ for the
underlying set. Associated to $\cC$ there is the (ordinary) \index{homotopy
category}homotopy category $h\cC$, with morphisms $[X,Y] = \pi_0 \Map_{\cC}(X,Y)$.

\subsection{Operads}
\label{sec:operads}
\index{operad|(}

Associated to any object $X \in \cC$ there is an
\emph{\index{endomorphism operad}endomorphism operad} $\End_\cC(X)$. The $k$'th term is
\[
  \Map_{\cC}(X^{\otimes k}, X),
\]
with an operad structure given by composition of functors. For any
operad $\oO$, this allows us to discuss $\oO$-algebra structures on
the objects of $\cC$, maps of $\oO$-algebras, and further structure.

If $\oO$ is the \index{associative operad}associative operad $\index{$\Assoc$}\Assoc$, then $\oO$-algebras are
monoid objects in the symmetric monoidal structure on $\cC$. If $\oO$
is the \index{commutative operad}commutative operad
$\index{$\Comm$}\Comm$, then $\oO$-algebras are strictly
commutative monoids in $\cC$. However, these operads are highly rigid
and do not take any space-level structure into account. Mapping spaces
allow us to encode many different levels of structure.
\begin{example}
  There is a sequence of operads
  $\mathcal{A}_1 \to \mathcal{A}_2 \to \mathcal{A}_3 \to \dots$ built
  out of the Stasheff \index{associahedra}associahedra \cite{stasheff-associahedra}. An
  $\index{$\mathcal{A}_n$-algebra}\mathcal{A}_2$-algebra has a unital
  binary multiplication; an $\mathcal{A}_3$-algebra has a chosen
  homotopy expressing associativity, and has \index{Massey product}Massey products; an
  $\mathcal{A}_4$-algebra has a homotopy expressing a juggling formula
  for Massey products; and so on. Moreover, each operad is simply
  built from the previous: extension from an
  $\mathcal{A}_{n-1}$-structure to an $\mathcal{A}_n$-structure
  roughly asks to extend a certain map $S^{n-3} \times X^n \to X$ to a
  map $D^{n-2} \times X^n \to X$ expressing an $n$-fold coherence law
  for the multiplication \cite{angeltveit}. This gives $\mathcal{A}_n$
  a \emph{\index{perturbative}perturbative} property: if $X \to Y$ is a homotopy
  equivalence, then $\mathcal{A}_n$-algebra structures on one space
  can be transferred to the other.
\end{example}
\begin{example}
  The colimit of the $\mathcal{A}_n$-operads is called
  $\mathcal{A}_\infty$, and it is equivalent to the associative
  operad. It satisfies a \emph{\index{rectification}rectification}
  property. In a well-behaved category like the category $\Spaces$ of
  spaces or the category $\Sp$ of spectra, any
  $\mathcal{A}_\infty$-algebra is equivalent in the homotopy category
  of $\mathcal{A}_\infty$-algebras to an associative object.
\end{example}
\begin{example}
  There is a sequence of operads
  $\mathcal{E}_1 \to \mathcal{E}_2 \to \mathcal{E}_3 \to \dots$, where
  the space $\mathcal{E}_n(k)$ is homotopy equivalent to the
  \index{configuration space}configuration space of ordered $k$-tuples of points in $\mb
  R^n$. These have various models, such as the \emph{\index{operad!little cubes/disks}little cubes} or
  \emph{little discs} operads. The $\mathcal{E}_1$-operad is
  equivalent to the associative operad, and the
  $\mathcal{E}_\infty$-operad is equivalent to the commutative
  operad. We refer to an algebra over any operad equivalent to
  $\mathcal{E}_n$ as an
  \emph{\index{$E_n$-algebra}$E_n$-algebra}. These play an important
  role in the \emph{\index{recognition principle}recognition
    principle} \cite{may-loopspaces, boardman-vogt-recognition}: given
  an $E_n$-algebra $X$ we can construct an $n$-fold \index{classifying space}classifying space
  $B^n X$; and if the binary multiplication makes $\pi_0(X)$ into a
  group then $X \simeq \Omega^n B^n X$. The relation between
  $E_n$-algebra structures and iteration of the functor $\Omega$ is closely
  related to an additivity result of \index{Dunn additivity}Dunn
  \cite{dunn-additivity}, who showed that $E_{n+1}$-algebras are
  equivalent to $E_1$-algebras in the category of $E_n$-algebras.
\end{example}
\begin{example}
  Associated to a \index{topological monoid}topological monoid $M$, there is an operad $\oO_M$
  whose only nonempty space is $\oO_M(1) = M$. An algebra over this
  operad is precisely an object with $M$-action. This operad is
  usually not \index{perturbative}perturbative. However, $M$ can be resolved by a
  \emph{cellular} topological monoid $\widetilde M \to M$ such that
  $\oO_{\widetilde M}$-algebras are perturbative and can be rectified
  to $\oO_M$-algebras. This construction is a recasting of \index{Cooke}Cooke's
  \index{obstruction theory}obstruction theory for lifting homotopy actions of a group $G$ to
  honest actions \cite{cooke-homotopyactions}; stronger versions of
  this were developed by \index{Dwyer}Dwyer--\index{Kan}Kan and \index{Badzioch}Badzioch
  \cite{dwyer-kan-liftingdiagrams, badzioch-algebraictheories}.
\end{example}
\begin{example}
  There is a free-forgetful adjunction between operads and \index{symmetric
  sequence}symmetric
  sequences. Given any sequence of spaces $Z_n$ with
  $\Sigma_n$-actions, we can construct an operad \index{operad!free}$\Free(Z)$ such that
  a $\Free(Z)$-algebra structure is the same as a collection of
  $\Sigma_n$-equivariant maps $Z_n \to \Map_{\cC}(A^{\otimes n}, A)$.

  If, further, $Z_1$ is equipped with a chosen point $e$, we can
  construct an operad $\Free(Z,e)$ such that a $\Free(Z,e)$-algebra
  structure is the same as a $\Free(Z)$-algebra structure such that
  $e$ acts as the identity: $\Free(Z,e)$ is a pushout of a diagram
  $\Free(Z) \leftarrow \Free(\{e\}) \to \Free(\emptyset)$ of operads.
\end{example}
\begin{example}
  \label{ex:cuponeoperad}
  In the previous example, let $Z_2$ be $S^1$ with the antipodal
  action of $\Sigma_2$ and let all other $Z_n$ be empty, freely
  generating an operad $\mathcal{Q}_1$ that we call the \emph{\index{operad}cup-1
    operad}. A $\mathcal{Q}_1$-algebra is an object $A$ with a
  $\Sigma_2$-equivariant map $S^1 \to \Map_\cC(A^{\otimes 2}, A)$. The
  $\Sigma_2$-equivariant cell decomposition of $S^1$ allows us to
  describe $\mathcal{Q}_1$-algebras as objects with a binary
  multiplication $m$ and a chosen homotopy from the multiplication $m$
  to the multiplication in the opposite order $m \circ \sigma$. In
  particular, any homotopy-commutative multiplication lifts to a
  $\mathcal{Q}_1$-algebra structure.
\end{example}

In the category $\Sp$ of spectra, one of the main applications of
\index{$E_n$-algebra}$E_n$-algebras is that they have well-behaved
\index{module category}categories of modules,
whose homotopy categories are triangulated categories.

\begin{theorem}[\index{Mandell}Mandell {\cite{mandell-derivedsmash}}]
  \label{thm:mandell-e4}
  An $E_1$-algebra $R$ in $\Sp$ has a category of left modules
  $\LMod_R$. An $E_2$-algebra structure on $R$ makes the homotopy
  categories of left modules and right modules equivalent, and gives
  the homotopy category of left modules a monoidal structure
  $\otimes_R$. An $E_3$-algebra structure on $R$ extends this monoidal
  structure to a \index{braided monoidal category}braided monoidal structure. An $E_4$-algebra
  structure on $R$ makes this braided monoidal structure into a
  symmetric monoidal structure.
\end{theorem}

\begin{theorem}
  An $E_1$-algebra $R$ in $\Sp$ has a monoidal category of \index{bimodules}bimodules.
  An $E_\infty$-algebra $R$ in $\Sp$ has a symmetric monoidal category
  of left modules.
\end{theorem}

\index{operad|)}
\subsection{Monads}
\label{sec:monads}
\index{monad|(}

If $\cC$ is not just enriched, but is \emph{tensored} over spaces, an
$\oO$-algebra structure on $X$ is expressible in terms internal to
$\cC$. An $\oO$-algebra structure is equivalent to having \emph{\index{action
  map}action
  maps}
\[
  \gamma_k\co \oO(k) \otimes X^{\otimes k}
  \to X
\]
that are invariant under the action of $\Sigma_k$ and respect
composition in the operad $\oO$. If $\cC$ has colimits, we can define
\emph{\index{extended power construction}extended power constructions}
\[
  \Sym^k_\oO(X) = \left(\oO(k) \otimes_{\Sigma_k}
  X^{\otimes k}\right),
\]
and an associated free $\oO$-algebra functor
\[
  \Free_\oO(X) = \coprod_{k \geq 0} \Sym^k_{\oO}(X).
\]
An $\oO$-algebra structure on $X$ is then determined by a single map
$\Free_\oO(X) \to X$. To say more, we need $\cC$ to be compatible with
\index{enriched colimit}enriched colimits in the sense of \cite[\S 3]{kelly-enriched}.
\begin{definition}
  \label{def:enriched-colimits}
  A symmetric monoidal category $\cC$ is \emph{compatible with
    enriched colimits} if the monoidal structure on $\cC$ preserves
  enriched colimits in each variable separately.
\end{definition}
Compatibility with enriched colimits is necessary to give composite
action maps
\[
  \oO(k) \otimes \left(\bigotimes_{i=1}^k \oO(n_i) \otimes X^{\otimes
      n_i}\right) \to X^{\otimes \Sigma n_i}
\]
and make them assemble into a \index{free algebra}monad structure
$\Free_\oO \circ \Free_\oO \to \Free_\oO$ on this free functor. In
this case, $\oO$-algebras are equivalent to $\Free_\oO$-algebras, and
$\Sym^k_\oO$ and $\Free_\oO$ are enriched functors.

When these functors are enriched functors, they also give rise to a
monad on the homotopy category $h\cC$. We refer to algebras over it as
\emph{\index{homotopy algebra}homotopy $\oO$-algebras}. This is
strictly stronger than being an $\oO$-algebra in the homotopy
category; the latter asks for compatible maps
$\pi_0 \oO(n) \to [A^{\otimes n}, A]$, whereas the former asks for
compatible elements in $[\oO(n) \otimes_{\Sigma_n} A^{\otimes n}, A]$
that use $\oO$ before passing to homotopy. In the case of the
\index{$E_n$-operad}$E_n$-operads, such a structure in the homotopy category is what is
classically known as an \emph{\index{$\mathcal{H}_n$-algebra}
  $\mathcal{H}_n$-algebra} \cite{bmms-hinfty}.
  
This type of structure can be slightly rigidified using pushouts of
\index{free operad}free algebras. For any operad $\oO$ with identity $e \in \oO(1)$,
we can construct a homotopy coequalizer diagram
\[
  \Free(\Free(\oO,e),e) \rightrightarrows \Free(\oO,e) \to \oO^h
\]
in the category of operads. An object $A$ has an $\oO^h$-algebra
structure if and only if there are $\Sigma_k$-equivariant maps
$\oO(k) \to \Map_{\cC}(A^{\otimes k}, A)$ so that the associativity
diagram homotopy commutes and so that $e$ acts by the identity. In
particular, $A$ has an $\oO^h$-algebra structure if and only if it has
a homotopy $\oO$-algebra structure; the $\oO^h$-structure has a
\emph{chosen} homotopy for the associativity of composition. For
example, there is an operad parametrizing objects with a unital binary
multiplication, a chosen associativity homotopy, and a chosen
commutativity homotopy.

\index{monad|)}
\subsection{Connective algebras}
\label{sec:connectivity}

In the category of spectra, the \index{Eilenberg--Mac Lane spectrum}Eilenberg--Mac Lane spectra $HA$ are
characterized by a useful mapping property. We refer to a spectrum as
\emph{\index{connective}connective} if it is $(-1)$-connected. For any connective
spectrum $X$, the natural map
\[
  \Map_\Sp(X,HA) \to \Hom_{\Ab}(\pi_0 X, A)
\]
is a weak equivalence.

This has a number of strong consequences. For example, we get an
equivalence of \index{endomorphism operad}endomorphism operads $\End_\Sp(HA) \to \End_\Ab(A)$,
obtained by taking $\pi_0$:
\[
  \End_\Sp(HA)_k = \Map(HA^{\otimes k},HA) \simeq \End_{\Ab}(A)_k =
  \Hom(A^{\otimes k}, A).
\]
Thus, an action of an operad $\oO$ on $HA$ is equivalent to an action
of $\pi_0\oO$ on $A$, and this equivalence is natural. This technique
also generalizes, using the equivalences
\[
  \Hom(H\pi_0 R^{\otimes n}, H\pi_0 R) \too{\sim} \Map(R^{\otimes n}, H
  \pi_0 R).
\]

\begin{proposition}
  Suppose $R$ is a connective spectrum and $\oO$ is an operad acting
  on $R$. Then the map $R \to H\pi_0(R)$ can be given, in a functorial
  way, the structure of a map of $\oO$-algebras.
\end{proposition}

\begin{example}
  \label{ex:cupone}
  If $A$ is given the structure of a commutative ring,
  \index{Eilenberg--Mac Lane spectrum!ring structure}$HA$ inherits
  an essentially unique structure of an $E_\infty$-algebra. If
  $R$ is a connective and homotopy commutative \index{ring spectrum}ring spectrum, then it
  can be equipped with an action of the \index{cup-1 operad}cup-1 operad $\mathcal{Q}_1$
  from \ref{ex:cuponeoperad}. Any ring homomorphism $\pi_0 R \to A$
  lifts to a map of $\mathcal{Q}_1$-algebras $R \to HA$.
\end{example}

\subsection{Example algebras}

\begin{example}
  There exist models for the category of spectra so that the \index{function spectrum}function spectrum
  \[
    F(\Sigma^\infty_+ X, A) = A^X
  \]
  is a lax monoidal functor $\Spaces^\op \times \Sp \to \Sp$, with the
  homotopy groups of $A^X$ being the unreduced $A$-cohomology groups
  of $X$. The diagonal $\Delta$ makes any space $X$ into a commutative
  monoid in $\Spaces^\op$. If $A$ is an $\oO$-algebra in $\Sp$, then
  $A^X$ then becomes an $\oO$-algebra.
\end{example}

\begin{example}
  For any spectrum $E$, composition of functions naturally gives the
  \emph{\index{endomorphism algebra}endomorphism algebra} spectrum $\End(E) = F(E,E)$ the structure of an
  $\mathcal{A}_\infty$-algebra, and $E$ is a left module over
  $\End(E)$. The homotopy groups of $\End(E)$ are sometimes called the
  \emph{\index{Steenrod algebra}$E$-Steenrod algebra} and they
  parametrize operations on $E$-cohomology.
\end{example}

\begin{example}
  The \index{suspension spectrum}suspension spectrum functor
  \[
    X \mapsto \Sigma^\infty_+ X =\mb S[X]
  \]
  is strong symmetric monoidal. As a result, it takes $\oO$-algebras
  to $\oO$-algebras. For example, any topological group $G$ has an
  associated \emph{\index{spherical group algebra}spherical group
    algebra} $\mb S[G]$.
\end{example}

\begin{example}
  For any pointed space $X$, the $n$-fold \index{iterated loop space}
  loop space $\Omega^n X$ is an \index{$E_n$-algebra}$E_n$-algebra in spaces, and
  $\mb S[\Omega^n X]$ is an $E_n$-algebra. For any spectrum $Y$ the
  space $\Omega^\infty Y$ is an $E_\infty$-algebra in spaces, and
  $\mb S[\Omega^\infty Y]$ is an $E_\infty$-algebra.
\end{example}

\begin{example}
  The \index{Thom spectrum}Thom spectra $\index{$MO$}MO$ and $\index{$MU$}MU$ have $E_\infty$ ring structures
  \cite{may-quinn-ray-ringspectra}. At any prime $p$, $MU$ decomposes
  into a sum of shifts of the \index{Brown--Peterson spectrum}Brown--Peterson spectrum $BP$, which has
  the structure of an $E_4$-\index{ring spectrum}ring spectrum \cite{basterra-mandell-BP}.
\end{example}

\begin{example}
  The smash product being symmetric monoidal implies that it is also a
  strong symmetric monoidal functor $\Sp \times \Sp \to \Sp$. If $A$
  and $B$ are $\oO$-algebras then so is $A \otimes B$.
\end{example}

\begin{example}
  For a map $Q \to R$ of $E_\infty$ ring spectra, there is an
  \index{adjunction!free-forgetful}adjunction
  \[
    \Mod_Q \rightleftarrows \Mod_R
  \]
  between the extension of scalars functor  $M \mapsto R \otimes_Q M$
  and the forgetful functor. The left adjoint is strong symmetric
  monoidal and the right adjoint is lax symmetric monoidal, and hence
  both functors preserve $\oO$-algebras.

  This allows us to narrow our focus. For example, if $E$ has an
  $E_\infty$-algebra structure and we are interested in understanding
  operations on the $E$-homology of $\oO$-algebras, we can restrict
  our attention to those operations on the homotopy groups of
  $\oO$-algebras in $\Mod_E$ rather than considering all possible
  operations on the $E$-homology.
\end{example}

\subsection{Multicategories}
\label{sec:multicategories}
\index{multicategory|(}

A \emph{multicategory} (or \index{colored operad|see {multicategory}}colored operad) encodes the structure of a
category where functions have multiple input objects. They serve as a
useful way to encode many multilinear structures in stable homotopy
theory: multiplications, module structures, graded rings, and
coherent structures on categories. In this section we will give a quick
introduction to them, and will return in \S\ref{sec:coherence}.

\begin{definition}
  A multicategory $\cM$ consists of the following data:
  \begin{enumerate}
  \item a collection $Ob(\cM)$ of objects;
  \item a set $\index{$\Mul$}\Mul_\cM(\ms x_1,\dots,\ms x_d;\ms y)$ of
    \emph{\index{multimorphism}multimorphisms} for any objects
    $\ms x_i$ and $\ms y$ of $\cM$, or more generally a set
    $\Mul_\cM(\{\ms x_s\}_{s \in S}; \ms y)$ for any finite set $S$
    and objects $\ms x_s$, $\ms y$;
  \item \index{composition!of multimorphisms}composition operations
    \[
      \circ\co \Mul_\cM(\{\ms y_t\}_{t \in T};\ms z) \times \prod_{t \in T}
      \Mul_\cM(\{\ms x_s\}_{s \in f^{-1} (t)}; \ms y_t)\to
      \Mul_{\cM}(\{\ms x_s\}_{s \in S}; \ms z)
    \]
    for any map $f\co S \to T$ of finite sets and objects $\ms x_s$,
    $\ms y_t$, and $\ms z$ of $\cM$; and
  \item identity morphisms $\mathrm{id}_X \in \Mul_\cM(\ms x; \ms
    x)$ for any object $\ms x$.
  \end{enumerate}
  These are required to satisfy two conditions:
  \begin{enumerate}
  \item unitality: $\mathrm{id}_{\ms y} \circ g = g \circ
    (\mathrm{id}_{\ms x_s}) = g$ for any $g \in \Mul_{\cM}(\{\ms
    x_s\}_{s \in S}; \ms Y)$; and
  \item associativity: $h \circ (g_u \circ (f_t)) = (h \circ (g_u))
    \circ f_t$ for any $S \to T \to U$ of finite sets.
  \end{enumerate}
  The \emph{\index{underlying category}underlying ordinary category} of $\cM$ is the category
  with the same objects as $\cM$ and
  $\Hom_{\cM}(\ms x,\ms y) = \Mul_{\cM}(\ms x; \ms y)$.
  
  If the sets of multimorphisms are given topologies so that
  composition is continuous, we refer to $\cM$ as a
  \emph{\index{multicategory!topological}topological multicategory}.

  A (topological) \emph{\index{multifunctor}multifunctor}
  $F\co \cM \to \cN$ is a map $F\co Ob(\cM) \to Ob(\cN)$ on the level
  of objects, together with (continuous) maps
  \[
    \Mul_\cM(\ms x_1,\dots,\ms x_d; \ms y) \to
    \Mul_\cM(F \ms x_1,\dots,F \ms x_d; F \ms y)
  \]
  that preserve identity morphisms and composition.
\end{definition}

\begin{example}
  An \index{operad}operad is equivalent to a single-object multicategory. For any
  object $\ms x$ in a multicategory $\cM$, the full sub-multicategory
  spanned by $\ms x$ is an operad called the endomorphism operad of
  $\ms x$.
\end{example}

\begin{example}
  A \index{symmetric monoidal topological category}symmetric monoidal topological category $\cM$ can be regarded as a
  multicategory by defining
  \[
    \Mul_{\cM}(X_1,\dots,X_d; Y) = \Map_\cM(X_1 \otimes \dots \otimes X_d,  Y).
  \]
  This recovers the definition of the \index{endomorphism operad}endomorphism operad of an object
  $X$.
\end{example}

The notion of an algebra over a multicategory will extend the notion
of an algebra over an operad.

\begin{definition}
  \label{def:multialgebras}
  For (topological) multicategories $\cM$ and $\cC$, the category
  $\index{$\Alg$}\Alg_\cM(\cC)$ of $\cM$-\index{algebra!over a multicategory}algebras in $\cC$ is the category of (topological)
  multifunctors $\cM \to \cC$ and natural transformations.

  For any object $\ms x \in \cM$, the \emph{\index{evaluation functor}evaluation} functor
  $\ev_{\ms x}\co \Alg_\cM(\cC) \to \cC$ sends an algebra $A$ to the
  value $A(\ms x)$.
\end{definition}

\begin{example}
  \label{ex:ringmodule}
  The multicategory $\index{$\Mod$}\Mod$ parametrizing ``ring-module
  pairs'' has two objects, $\ms a$ and $\ms m$, and
  \[
    \Mul_{\Mod}(\ms x_1,\dots,\ms x_d;\ms y) =
    \begin{cases}
      \ast &\text{if }\ms y = \ms a\text{ and all }\ms x_i\text{ are
      }\ms a,\\
      \ast &\text{if }\ms y = \ms m\text{ and exactly one }\ms
      x_i\text{ is }\ms m,\\
      \emptyset &\text{otherwise.}
    \end{cases}
  \]
  A multifunctor $\Mod \to \cC$ is equivalent to a pair $(A,M)$
  of a commutative monoid $A$ of $\cC$ and an object $M$ with an
  action of $A$.
\end{example}

\begin{example}
  \label{ex:gradedalgebras}
  A commutative monoid $\Gamma$ can be regarded as a symmetric
  monoidal category with no non-identity morphisms, and in the
  associated multicategory we have
  \[
    \Mul_\Gamma(g_1,\dots,g_d; g) = \begin{cases}
      \ast &\text{if }\sum g_s = g,\\
      \emptyset &\text{otherwise.}
    \end{cases}
  \]
  A multifunctor $\Gamma \to \cC$ determines objects $X_g$ of $\cC$, a
  map from the unit to $X_0$, and multiplication maps
  $X_{g_1} \otimes \dots \otimes X_{g_d} \to X_{g_1 + \dots + g_d}$:
  these multiplications are collectively unital, symmetric, and
  associative. We refer to such an object
  as a \index{graded algebra}$\Gamma$-graded
  commutative monoid.
\end{example}

\begin{example}
  \label{ex:filteredalgebras}
  The addition of natural numbers makes the partially ordered set
  $(\mb N, \geq)$ into a symmetric monoidal category. In the
  associated multicategory we have
  \[
    \Mul_{\mb N}(n_1,\dots,n_d; m) = \begin{cases}
      \ast &\text{if }\sum n_i \geq m,\\
      \emptyset &\text{otherwise.}
    \end{cases}
  \]
  A multifunctor $\Gamma \to \cC$ determines a sequence of objects
  \[
    \dots \to X_2 \to X_1 \to X_0
  \]
  of $\cC$ and multiplication maps $X_{n_1} \otimes \dots \otimes
  X_{n_d} \to X_{n_1 + \dots + n_d}$:
  these multiplications are collectively unital, symmetric, and
  associative, as well as being compatible with the inverse system. We
  refer to such an object \emph{\index{filtered algebra}strongly filtered} commutative monoid
  in $\cC$.
\end{example}

\begin{remark}
  If $\cM_1$ and $\cM_2$ are multicategories, there is a
  \index{product!of multicategories}product
  multicategory $\cM_1 \times \cM_2$, obtained by taking products of
  objects and products of multimorphism spaces. Products allow us to
  extend the above constructions. For example, taking the product of
  an operad $\oO$ with the multicategories of the previous examples,
  we construct multicategories that parametrize: pairs $(A,M)$ of an
  $\oO$-algebra and an $\oO$-module; $\Gamma$-graded $\oO$-algebras;
  and strongly filtered $\oO$-algebras.
\end{remark}

\begin{example}
  \label{ex:operationcat}
  Let $\cM$ be the multicategory whose objects are integers, and
  define $\Mul_{\cM}(m_1,\dots,m_d; n)$ to be the set of natural
  transformations
  \[
    \theta\co H^{m_1}(X) \times \dots \times H^{m_d}(X) \to H^n(X)
  \]
  of contravariant functors on the category $\Spaces$ of spaces; composition is
  composition of natural transformations. The category $\cM$ is a
  category of multivariate \index{cohomology operations}cohomology operations. Any fixed space $X$
  determines an \index{evaluation functor}evaluation multifunctor $\ev_X\co \cM \to \Sets$,
  sending $n$ to $H^n(X)$; any homotopy class of map $X \to Y$ of
  spaces determines a natural transformation of multifunctors in the
  opposite direction. Stated concisely, this is a functor
  \[
    h\Spaces^\op \to \Alg_{\cM}(\Sets)
  \]
  that takes a space to an encoding of its cohomology groups and cohomology operations.
  
  More generally, a category $\cD$ with a chosen set of functors
  $\cD \to \Sets$ determines a multicategory $\cM$ spanned by them: we
  can define $\Mul(F_1,\dots,F_d;G)$ to be the set of natural
  transformations $\prod F_i \to G$, so long as there is always a set
  (rather than a proper class) of natural transformations. If we view
  a functor $F$ as assigning an invariant to each object of $\cD$, a
  multimorphism $\prod F_i \to G$ is a natural operation of several
  variables on such invariants. Evaluation on objects of $\cD$ takes
  the form of a functor
  \[
    \cD \to \Alg_{\cM}(\Sets),
  \]
  encoding both the invariants assigned by these functors and the
  natural operations on them. These are examples of
  \emph{\index{multi-sorted theories}multi-sorted algebraic
    theory} in the sense of \index{Bergner}Bergner \cite{bergner-multisorted},
  closely related to the work of \cite{borger-wieland-plethystic,
    stacey-whitehouse-hopfring}. We will return to the discussion of
  this structure in \S\ref{sec:operators}.
\end{example}

Just as with ordinary operads, there are often free-forgetful
adjunctions between objects of $\cC$ and algebras over a
multicategory.

\begin{proposition}
  Suppose that $\cM$ is a small topological multicategory and that
  $\cC$ is a symmetric monoidal topological category with compatible
  colimits in the sense of Definition~\ref{def:enriched-colimits}.
  \begin{enumerate}
  \item For objects $\ms x$ and $\ms y$ of $\cM$, there are \index{extended power construction}extended power
    functors
    \[
      \Sym^k_{\cM, \ms x \to \ms y}\co \cC_{\ms x} \to \cC_{\ms y},
    \]
    given by
    \[
      \Sym^k_{\cM, \ms x \to \ms y} (X) = \Mul_\cM(\underbrace{\ms
        x, \ms x, \dots, \ms x}_{k}; \ms y) \otimes_{\Sigma_k} X^{\otimes k}.
    \]
  \item The evaluation functor
    $\ev_{\ms x}\co \Alg_\cM(\cC) \to \cC$ has a left adjoint
    \[
      \Free_{\cM,\ms x}\co \cC \to \Alg_{\cM}(\cC).
    \]
    The value of \index{free algebra}$\Free_{\cM,\ms x}(X)$ on any object $\ms y$ of $\cM$ is
    \[
      \ev_{\ms y} (\Free_{\cM,\ms x}(X)) = \coprod_{k \geq 0}
      \Sym^k_{\cM, \ms x \to \ms y}(X).
    \]
  \end{enumerate}
\end{proposition}

\begin{remark}
  These generalize the constructions of extended powers and free
  algebras from \S \ref{sec:monads}. If $\cM$ has a single object
  $\ms x$, encoding an operad $\oO$, then
  $\Sym^k_{\cM, \ms x \to \ms x} = \Sym^k_\oO$ and
  $\Free_{\cM, \ms x}$ encodes $\Free_\oO$.
\end{remark}

\begin{example}
  The free $\mb Z$-\index{graded algebra}graded commutative monoid on an object $X$ in degree
  $n \neq 0$ is equal to the symmetric product $\Sym^k(X)$ in degree $kn$
  for $k \geq 0$. All other gradings are the initial object.
\end{example}

\begin{example}
  The free \index{filtered algebra}strongly filtered commutative
  monoid on an object $X_1$ in degree $1$ is a filtered object of the
  form
  \[
    \dots \to \coprod_{k \geq 2} \Sym^k X_1 \to \coprod_{k \geq 1}
    \Sym^k X_1 \to \coprod_{k \geq 0} \Sym^k X_1.
  \]
  If we have a strongly filtered commutative algebra
  $\dots \to X_2 \to X_1 \to X_0$, then this gives action maps
  $\Sym^k X_1 \to X_k$. More generally, there are action maps
  $\Sym^k X_n \to X_{kn}$ that are compatible in $n$.
\end{example}

\index{multicategory|)}
\section{Operations}
\label{sec:operations}
\index{operations|(}

In this section we will fix a spectrum $E$, viewed as a coefficient
object.

\subsection{$E$-homology and $E$-modules}

We can study $\oO$-algebras through their $E$-homology.

\begin{definition}
  Given a spectrum $E$, an \emph{$E$-homology
    \index{homology operations}operation} for $\oO$-algebras is a
  natural transformation of functors $\theta\co E_m(-) \to E_{m+d}(-)$
  of functors on the homotopy category of $\oO$-algebras.
\end{definition}

Such operations can be difficult to classify in general. However, if
$E$ has a commutative ring structure then we can do more. In this
case, any $\oO$-algebra $A$ has an \emph{$E$-\index{homology
    object}homology object} $E \otimes A$ which is an $\oO$-algebra in
$\Mod_E$, and any space $X$ has an \emph{$E$-\index{cohomology
    object}cohomology object} $E^X$ which is an $E_\infty$-algebra
object in $\Mod_E$. By definition, we have
\[
  E_m(A) = [S^m, E \otimes A]_{\Sp}
\]
and 
\[
  E^m(X) = [S^{-m}, E^X]_{\Sp}.
\]
Therefore, we can construct natural operations on the $E$-homology of
$\oO$-algebras or the $E$-cohomology of spaces by finding natural
operations on the homotopy groups of $\oO$-algebras in $\Mod_E$.

\begin{example}
  If $X$ is an $\oO$-algebra in spaces, then
  $E[X] = E \otimes \Sigma^\infty_+ X$ is an $\oO$-algebra in
  $\Mod_E$.
\end{example}

\subsection{Multiplicative operations}

In this section we will construct our first operations on the
homotopy groups of $\oO$-algebras over a fixed commutative \index{ring spectrum}ring spectrum $E$.

The functor $\pi_*$ from the homotopy category of spectra to graded
abelian groups is lax symmetric monoidal under the Koszul \index{signs}sign
rule. The induced functor $\pi_*$ from $\Alg_\oO(\Sp)$ or
$\Alg_\oO(\Mod_E)$ to graded abelian groups naturally takes values in
the category of graded abelian groups, or graded $E_*$-modules, with
an action of the operad $\pi_0 \oO$ in sets.

\begin{example}
  In the case of an $E_n$-operad, $\pi_0 \oO$ is isomorphic to the
  \index{associative operad}associative operad when $n=1$ and the \index{commutative operad}commutative operad when
  $n \geq 2$. The $E$-homology groups of an $E_n$-algebra in $\Sp$ form a
  graded $E_*$-algebra. If $n \geq 2$, this algebra is
  graded-commutative.
\end{example}

By applying $E_*$ to the \index{action map}action maps in the operad, we stronger
information.

\begin{proposition}
  The homology groups $E_* \oO(k)$ form an operad $E_* \oO$ in graded
  $E_*$-modules, and the functor $\pi_*$ from $\Alg_\oO(\Mod_E)$ to
  graded abelian groups has a natural lift to the category of graded
  $E_* \oO$-modules.
\end{proposition}

\begin{example}
  \index{bracket!Browder|see{Browder bracket}}
  \index{Browder bracket|(}
  The homotopy groups of $E_n$-algebras have a natural bilinear
  \emph{Browder bracket}
  \[
    [-,-]\co \pi_q(A) \otimes \pi_r(A) \to \pi_{q + (n-1) + r}(A).
  \]
  This satisfies the following formulas.
  \begin{quote}
    \begin{description}
    \item[\index{antisymmetry}Antisymmetry:]
      $[\alpha,\beta] = - (-1)^{(|\alpha| + n-1)(|\beta| + n-1)}
      [\beta,\alpha]$.
    \item[\index{Leibniz rule}Leibniz rule:]
      $[\alpha, \beta \gamma] = [\alpha, \beta] \gamma +
      (-1)^{|\beta|(|\alpha| + n-1)} \alpha [\beta, \gamma].$
    \item[Graded \index{Jacobi identity}Jacobi identity:]
      \begin{align*}
        0 = &\phantom{+\ }(-1)^{(|\alpha| + n-1)(|\gamma| + n-1)}
              [\alpha, [\beta, \gamma]]\\
            &+ (-1)^{(|\beta| + n-1)(|\alpha| + n-1)} [\beta, [\gamma, \alpha]]\\
            &+ (-1)^{(|\gamma| + n-1)(|\beta| + n-1)} [\gamma, [\alpha, \beta]].
      \end{align*}
    \end{description}
  \end{quote}
  In the case of $E_1$-algebras, this reduces to the ordinary bracket
  \[
    [\alpha, \beta] = \alpha \beta - (-1)^{|\alpha| |\beta|} \beta\alpha
  \]
  in the graded ring $\pi_* (A)$.

  The Browder bracket is defined, just as it was defined in homology
  \cite{cohen-lada-may-homology}, using the image of the generating
  class $\lambda \in \pi_{n-1} \mathcal{E}_{n}(2) \cong \pi_{n-1} S^{n-1}$
  coming from the \index{operad!little cubes/disks}little cubes operad.  The antisymmetry and Jacobi
  identities are obtained by verifying identities in the graded operad
  $\pi_* (\Sigma^\infty_+ \mathcal{E}_{n})$. For example, if $\sigma$ is
  the 2-cycle in $\Sigma_2$ we have
  \[
    \lambda \circ \sigma = (-1)^n \lambda,
  \]
  and if $\tau$ is a 3-cycle in $\Sigma_3$ we have
  \[
    \lambda \circ (1 \otimes \lambda) \circ (1 + \tau + \tau^2) =
    0.
  \]
  However, the \index{signs}signs indicate that there is some care to be taken. In
  particular, the Browder bracket of elements $\alpha \in \pi_q(A)$
  and $\beta \in \pi_r(A)$ is defined to be the following composite:
  \begin{align*}
    S^q \otimes S^{n-1} \otimes S^r
    &\to A \otimes \Sigma^\infty_+ \mathcal{E}_{n}(2) \otimes A \\
    &\to \Sigma^\infty_+ \mathcal{E}_{n}(2) \otimes A \otimes A \\
    &\to A
  \end{align*}
  This order is chosen because it is more consistent with writing the
  Browder bracket as an inline binary operation $[x,y]$ than with
  writing it as an operator $\lambda (x,y)$ on the left. The
  subscript on the range $\pi_{q + (n-1) + r}(A)$ reflects this
  choice (cf. \cite{ni-bracket}). This gives us the definition
  \[
    [\alpha, \beta] = (-1)^{(n-1)|\alpha|} \gamma(\lambda \otimes
    \alpha \otimes \beta),
  \]
  where $\gamma$ is the action map of the operad
  $\pi_* (\Sigma^\infty_+ \mathcal{E}_{n})$ on $\pi_* A$. Both the
  verification of the identities on $\lambda$ in the stable homotopy
  groups of configuration spaces, and the verification of the
  consequent antisymmetry, Leibniz, and Jacobi identities, are
  reasonable but error-prone exercises from this point; compare
  \cite{cohen-configurationspaces}.
  \index{Browder bracket|)}
\end{example}

\subsection{Representability}
\label{sec:representability}

We will ultimately be interested in natural operations on homotopy and
homology groups. However, it is handy to use a more general definition
that replaces $S^m$ by a general object. This accounts for the
possibility of operations of several variables, and can also help
reduce difficulties involving naturality in the input $S^m$.

\begin{definition}
  For spectra $M$ and $X$, we define the $M$-\index{indexed homotopy
    groups}indexed homotopy of $X$ to be
  \[
    \pi_M(X) = [M,X]_{\Sp} \cong [E \otimes M, X]_{\Mod_E}.
  \]

  For spectra $M$, $X$, and $E$ we define the $M$-indexed
  $E$-\index{indexed homology groups}homology of $X$ to be
  \[
    E_M(X) = \pi_M(E \otimes X).
  \]
  If $M$ is $S^m$, we instead use the more standard notation
  $\pi_m(-)$ for $\pi_{S^m}(-)$ or $E_m(-)$ for
  $E_{S^m}(-)$.
\end{definition}

\begin{definition}
  Let $E$ be a commutative ring spectrum. A \emph{\index{homotopy operations}homotopy operation}
  for $\oO$-algebras over $E$ is a natural transformation
  \[
    \theta\co \pi_M \to \pi_N
  \]
  of functors on the homotopy category of $\Alg_\oO(\Mod_E)$. When
  $\oO$ and $E$ are understood, we just refer to such natural
  transformations as homotopy operations.

  We refer to the resulting operation $E_M(-) \to E_N(-)$ on the
  $E$-homology groups of $\oO$-algebras as the induced
  $E$-\index{homology operations}homology operation.
\end{definition}

As in Example~\ref{ex:operationcat}, we can assemble operations with
varying numbers of inputs into an algebraic structure.
\begin{definition}
  Fix an operad $\oO$ and a commutative ring spectrum $E$.  The
  multicategory $\index{$\Op$}\Op^E_\oO$ of \emph{operations} for
  $\oO$-algebras in $\Mod_E$ has, as objects, spectra $N$. For any
  $M_1,\dots,M_d$ and $N$, the group of multimorphisms
  \[
    \Op^E_\oO(M_1, \dots, M_d; N)
  \]
  is the group of natural transformations $\prod \pi_{M_i} \to \pi_N$
  of functors $h\Alg_\oO(\Mod_E) \to \Sets$. If $E$ or $\oO$ are
  understood, we drop them from the notation.

  In the unary case, we write $\index{$\Op$}\Op^E_\oO(M; N)$ for the set of
  homotopy operations $\pi_M \to \pi_N$ for $\oO$-algebras in
  $\Mod_E$.
\end{definition}

The free-forgetful adjunction between spectra and $\oO$-algebras in
$\Mod_E$ allows us to exhibit the functor $\pi_M$ as
representable.

\begin{proposition}
  \index{representability!homotopy groups}Suppose that $E$
  is a commutative ring spectrum, $\oO$ is an operad with associated
  \index{free algebra}free algebra \index{monad}monad $\Free_\oO$. Then there is a natural isomorphism
  \[
    \pi_M(A) \cong [E \otimes \Free_\oO(M), A]_{\Alg_\oO(\Mod_E)}
  \]
  for $A$ in the homotopy category of $\Alg_\oO(\Mod_E)$. In
  particular, the object $E \otimes \Free_\oO(M)$ is a representing
  object for the functor $\pi_M$.
\end{proposition}

\begin{proof}
  The forgetful functor $\Alg_\oO(\Mod_E) \to \Sp$ can be expressed as
  a composite $\Alg_\oO(\Mod_E) \to \Alg_\oO(\Sp) \to \Sp$, and as
  such has a composite left adjoint $M \mapsto \Free_\oO(M) \mapsto E
  \otimes \Free_\oO (M)$; this adjunction passes to the homotopy
  category. Therefore, applying this adjunction we find
  \begin{align*}
    \pi_{M}(A) &\cong [\Free_\oO(M), A]_{\Alg_\oO(\Sp)}\\
    &\cong [E \otimes \Free_\oO(M), A]_{\Alg_\oO(\Mod_E)}
  \end{align*}
  as desired.
\end{proof}

\begin{remark}
  It is possible to index more generally. Given an $E$-module $L$, we
  also have functors $\pi_L^E(-) = [L,-]_{\Mod_E}$; the free
  $\oO$-algebra $\Free_\oO(L)$ in the category of $E$-modules is then
  a representing object for $\pi_L^E$ in $\Alg_{\oO}(\Mod_E)$. We
  recover the above case by setting $L = E \otimes M$.
\end{remark}

The \index{Yoneda lemma}Yoneda lemma now gives the following.

\begin{corollary}
  Let $F$ be a functor from $h\Alg_{\oO}(\Mod_E)$ to the category of
  sets. Natural transformations of functors $\pi_M \to F$ are in
  bijective correspondence with $F(E \otimes \Free_\oO(M)).$

  In particular, there is an isomorphism 
  \[
    \Op^E_\oO(M_1,\dots,M_d; N) \cong E_N(\Free_\oO(\oplus M_i))
  \]
  from the group of natural transformations $\prod \pi_{M_i} \to
  \pi_N$ to the $E$-homology group of the free algebra.
\end{corollary}

The canonical decomposition of \S\ref{sec:monads} for the monad
$\Free_\oO$ into extended powers gives us a canonical decomposition
of operations.

\begin{definition}
  For $k \geq 0$, the group of \emph{operations of
    \index{weight!of an operation}weight $k$} is the subgroup
  \[
    \Op^E_\oO(M_1,\dots,M_d; N)^{\langle k \rangle} =
    E_N(\Sym^k_\oO (\oplus M_i))
  \]
  of $\Op^E_\oO(M_1,\dots,M_d; N) \cong E_N(\Free_\oO(\oplus
  M_i))$.

  A \emph{\index{power operations}power operation of weight $k$} is a
  unary operation of weight $k$: an element of the subgroup
  \[
    \Op^E_\oO(M,N)^{\langle k\rangle} \cong E_N(\Sym^k_\oO(M))
  \]
  of $\Op^E_\oO(M,N)$.
\end{definition}

\begin{remark}
  Composition multiplies weight. Furthermore, if the object $N$ is dualizable,
  the group of all operations is a direct sum: every operation
  decomposes canonically as a sum of operations of varying
  weights.
\end{remark}

\subsection{Structure on operations}
\label{sec:operators}

Even when restricted to ordinary homotopy groups, these operations
between the homotopy groups of $\oO$-algebras in $\Mod_E$ form a
rather rich algebraic structure \index{multi-sorted theory}\cite{bergner-multisorted}, whose
characteristics should be discussed; we learned most of this from Rezk
\cite{rezk-wilkerson, rezk-power-operations}. Recall
\[
  \Op(m_1,\dots,m_d;n) = \Op^E_\oO (m_1,\dots,m_d; n) \cong
  \pi_n(E \otimes \Free_\oO(\oplus S^{m_i})).
\]
Here are some characteristics of this algebraic theory.
\begin{enumerate}
\item We think of the elements in these groups as \index{operations}operators, in the
  sense that they can \emph{act}. Given
  $\alpha \in \Op(m_1,\dots,m_d;n)$, an $\oO$-algebra $R$ in
  $\Mod_E$ and $x_i \in \pi_{m_i} R$, we can apply $\alpha$ to get a
  natural element
  \[
    \alpha \propto (x_1,\dots,x_d) \in \pi_n R.
  \]
  This action is associative with respect to composition, but only
  distributes over addition on the left.
\item For each $1 \leq k \leq d$, there is a \emph{\index{fundamental
      generator}fundamental generator}
  $\iota_k \in \Op(m_1,\dots,m_d;m_k)$ that acts by projecting:
  \[
    \iota_k \propto (x_1,\dots,x_d) = x_k.
  \]
\item These operators can \emph{\index{composition!of operators}compose}: given
  $\alpha \in \Op(m_1,\dots,m_d;n)$ and
  $\beta_i \in \Op(\ell_1,\dots,\ell_c; m_i)$, there
  is a composite operator
  \[
    \alpha \propto (\beta_1,\dots,\beta_d) \in
    \Op(\ell_1,\dots,\ell_c;n).
  \]
  Composition is unital. It is also associative, both with itself and
  with acting on elements. Again, it only distributes over addition on
  the left.
\item Composition respects weight: if $\alpha$ is in weight $a$ and
  $\beta_i$ are in weights $b_i$, then $\alpha \propto (\beta_i)$ is
  in weight $a \cdot (\sum b_i)$.
\end{enumerate}

\begin{example}
  Take $E = HR$ for a commutative ring $R$ and let $\oO$ to be the
  \index{associative operad}associative operad. Then the graded group
  \[
    Op(m_1,\dots,m_d;*) = \oplus_n \Op(m_1,\dots,m_d;n) \cong
    H_*(\Free_{\oO}(\oplus S^{m_i};R))
  \]
  is the free associative graded $R$-algebra on the fundamental
  generators $\iota_1 \dots \iota_d$ with $\iota_i$ in degree $m_i$,
  and the composition operations are \emph{substitution}. For example,
  the element $\iota_1 + \iota_2 \in \Op(n,n;n)$ acts by the
  binary addition operation in degree $n$; the elements
  $\iota_1 \iota_2$ and $\iota_2 \iota_1$ in
  $\Op(n_1,n_2;n_1 + n_2)$ represent binary multiplication in
  either order; the element $(\iota_1)^2 \in \Op(n;2n)$ represents
  the squaring operation; for $r \in R$ the element
  $r \iota_1 \in \Op(n;n)$ represents scalar multiplication by $r$;
  combinations of these operations are represented by identities such
  as
  \[
    \iota_1^2 \propto (\iota_1 + \iota_2) = \iota_1^2 + \iota_1
    \iota_2 + \iota_2 \iota_1 + \iota_2^2.
  \]
  In this structure, each monomial has constant weight equal to its
  degree.
\end{example}

\begin{example}
  Take $\oO$ to be an $E_n$-operad. Then, for any $p$ and $q$, the
  \index{Browder bracket}Browder bracket is a natural transformation
  $\pi_p \times \pi_q \to \pi_{p+(n-1)+q}$, and it is realized by an
  element $[\iota_1, \iota_2]$ in $\Op(p,q;p+(n-1)+q)$ of weight
  two. Relations between the product and the Browder bracket are
  expressed universally by relations between compositions: for
  example, antisymmetry is expressed by an identity
  \[
    [\iota_1, \iota_2] = -(-1)^{(p+n-1)(q+n-1)}[\iota_2,\iota_1].
  \]
\end{example}

\begin{remark}
  Inside the collection of all unary operations, there is a subgroup
  of \emph{\index{additive operations}additive} operations: those
  operations $f$ that satisfy
  \[
    f \propto (\iota_1 + \iota_2) = f \propto \iota_1 + f \propto
    \iota_2.
  \]
  Composition of such operations is bilinear, and so the collection of
  objects and additive operations form a category enriched in abelian
  groups. In some cases, the additive operations can be used to
  determine the general structure \cite{rezk-wilkerson}.
\end{remark}

\subsection{Power operations}
\label{sec:power-operations}

We will begin to narrow our study of power operations and focus on
unary operations, of fixed weight, between integer gradings.

\begin{definition}
  \label{def:poweroperations}
  Fix an operad $\oO$ and a commutative ring spectrum $E$. The group
  of \emph{\index{power operations}power operations of weight $k$ on degree $m$} for
  $\oO$-algebras in $\Mod_E$ is the graded abelian group
  \[
    \index{$\Pow$}\Pow^E_\oO(m,k) = \pi_* (F(S^m, E \otimes \Sym^k_\oO(S^m)))
    \cong \bigoplus_{r \in \mb Z} \Op^E_{\oO}(m,m+r)^{\langle k\rangle}.
  \]
  If $\oO$ or $E$ are understood, we drop them from the notation.
\end{definition}

An element of $\Pow(m,k)$ in grading $r$ represents a weight-$k$ natural
transformation $\pi_m \to \pi_{m+r}$ on the homotopy category of
$\oO$-algebras in $\Mod_E$, and induces a natural transformation
$E_m \to E_{m+r}$ on the homotopy category of $\oO$-algebras. (While
we index these group by integers, they depend on a choice of
representing object and in particular on an orientation of $S^m$;
making implicit identifications will result in \index{signs}sign issues.)

\begin{remark}
  These operations, and the relations between them, are still
  possessed by homotopy $\oO$-algebras in the sense of
  \S\ref{sec:monads}.
\end{remark}

\begin{remark}
  \label{rmk:thom}
  Suppose that $\Sigma_k$ acts freely and properly discontinuously on
  $\oO(k)$. Let $V \subset \mb R^k$ be the subspace of elements which
  sum to $0$, with associated vector bundle
  $\overline \rho \to B\Sigma_k$ of dimension $k-1$. For any $m$ there
  is an associated virtual bundle $\mb R^m \otimes \overline \rho$. If
  we define
  \[
    P(k) = \oO(k) / \Sigma_k,
  \]
  then there is a virtual bundle $m\overline \rho$ on
  $P(k)$. The \index{Thom spectrum}Thom spectrum $P(k)^{m\overline\rho}$ of this virtual
  bundle is canonically equivalent to the spectrum $\Sigma^{-m}
  \Sigma^\infty_+ \oO(k) \otimes_{\Sigma_k} (S^m)^{\otimes k}$ that
  appears in the definition of $\Pow(m, k)$.

  This allows us to give a more concise expression
  \[
    \Pow(m,k) = E_*(P(k)^{m\overline\rho}),
  \]
  which is particularly useful in cases where we can apply a \index{Thom isomorphism}Thom
  isomorphism for $E$-homology.
\end{remark}

\begin{example}
  Consider the case of operations of weight $2$ for
  $E_n$-algebras. The space $P(2) = \mathcal{C}_n(2) / \Sigma_2$ is
  homotopy equivalent to the \index{real projective space}real projective space $\mb{RP}^{n-1}$,
  the line bundle $\overline \rho = \sigma$ is associated to the sign
  representation of $\Sigma_2$, and the \index{Thom spectrum}Thom spectrum
  $(\mb{RP}^{n-1})^{m\sigma}$ is commonly known as the
  \emph{\index{stunted projective space}stunted
    projective space} $\mb{RP}^{m+n-1}_m$ which has a cell
  decomposition with one cell in each dimension between $m$ and
  $m+n-1$. (When $m \geq 0$ this is literally the suspension spectrum of
  $\mb{RP}^{m+n-1} / \mb{RP}^{m-1}$.) Therefore, the operations of
  weight $2$ on degree $m$ are parametrized by the $E$-homology group
  \[
    \Op^E_m(2) = E_*(\mb{RP}^{m+n-1}_m).
  \]
\end{example}

\begin{example}
  When $E = H\mb F_2$, we find $H_*(\mb{RP}^{m+n-1}_m)$ is $\mb F_2$
  in degrees $m$ through $(m+n-1)$, and so we obtain unique
  \emph{\index{Dyer--Lashof operations!construction}Dyer--Lashof operations} $Q^r$ for $m \leq r \leq m+n-1$ that
  send elements in $\pi_m $ to elements in $\pi_{m+r}$.
\end{example}

\begin{example}
  \label{ex:cuponesquare}
  Consider the \index{cup-1 operad}cup-1 operad $\mathcal{Q}_1$ defined in
  Example~\ref{ex:cuponeoperad}. Then the weight-2 operations on the
  $E$-homology of $\mathcal{Q}_1$-algebras are parametrized by
  $E_* (\mb{RP}^{m+1}_m)$. This stunted projective space is the \index{Thom spectrum}Thom
  spectrum of $m$ times the M\"obius line bundle over $S^1$.

  For example, we can take $E$ to be the sphere spectrum. If $m = 2k$
  there is a splitting
  \[
    \mb{RP}_{2k}^{2k+1} \simeq S^{2k} \oplus S^{2k+1}.
  \]
  Chosen generators in $\pi_{2k}(S^{2k} \oplus S^{2k+1})$ and
  $\pi_{2k+1}(S^{2k} \oplus S^{2k+1})$ give operations that increase
  degree by $2k$ and $2k+1$, respectively. A choice of splitting
  $S^{2k+1} \to \mb{RP}_{2k}^{2k+1}$ determines an operation
  $\Sq_1\co \pi_{2k}(-) \to \pi_{4k+1}(-)$ called the \index{cup-1
    square}\emph{cup-1 square}. It satisfies $2\Sq_1(a) = [a,a]$.
  
  In the case that we have an $E_\infty$ ring spectrum, this has been studied in \cite[\S
  V]{bmms-hinfty} and \cite{baues-muro-cupone}, and can be chosen in
  such a way that it satisfies the following addition and
  multiplication identities on even-degree homotopy elements:
  \[
    2 \Sq_1(a) = 0
  \]
  \[
    \Sq_1(a+b)
    = \Sq_1(a) + \Sq_1(b) + (\tfrac{|a|}{2}+1)ab\eta
  \]
  \[
    \Sq_1(ab)
    = a^2 \Sq_1(b) + \Sq_1(a) b^2 + \tfrac{|ab|}{4} a^2 b^2 \eta.
  \]
  For example, $\Sq_1(n) = \binom{n}{2} \eta$ for $n \in \mb Z$. In
  the absence of higher commutativity, these identities should have
  correction terms involving the \index{Browder bracket}Browder bracket.
\end{example}

\subsection{Stability}
\label{sec:stability}
\index{stability|(}

In this section we will consider compatibility relations between
operations on different homotopy degrees.

Recall from \S\ref{sec:monads} that the monad $\Free_\oO$ decomposed
into the homogeneous functors defined by
\[
  \Sym^k_\oO(X) = \Sigma^\infty_+ \oO(k) \otimes_{\Sigma_k}
  X^{\otimes k}.
\]

In particular, these functors are \emph{continuous}: they induce
functions 
\[
  \Map(X,Y) \to \Map(\Sym^k_\oO(X), \Sym^k_\oO(Y))
\]
between mapping spaces, and for $k > 0$ they have the property that
they are \emph{pointed}: $\Sym^k_\oO(*) = *$ and hence the
functor $\Sym^k_\oO$ induces continuous maps of
\emph{pointed} mapping spaces.

\begin{definition}
  For any spectrum $M$, any pointed space $Z$, and any $k > 0$, the
  \emph{\index{assembly map}assembly map}
  \[
    \Sym^k_\oO(M) \otimes \Sigma^\infty Z \to
    \Sym^k_\oO(M \otimes \Sigma^\infty Z)
  \]
  is adjoint to the composite map of pointed spaces
  \begin{align*}
    Z &\to \Map_\Sp(S^0, \Sigma^\infty Z)\\
      &\to \Map_\Sp(M, M \otimes \Sigma^\infty Z)\\
      &\to \Map_\Sp(\Sym^k_\oO(M),
        \Sym^k_\oO(M \otimes \Sigma^\infty Z)).
  \end{align*}
  The \emph{\index{suspension!of operations}suspension map}
  \[
    \sigma_n\co \Pow(m,k) \to \Pow(m+n,k)
  \]
  is induced by the composite map of function spectra
  \begin{align*}
    F(S^m,E \otimes \Sym^k_\oO(S^m))
    &\to F(S^m \otimes S^n, E \otimes \Sym^k_\oO(S^m) \otimes S^n)\\
    &\to F(S^m \otimes S^n, E \otimes \Sym^k_\oO (S^m \otimes S^n)).
  \end{align*}
\end{definition}

\begin{remark}
  The operation $\sigma = \sigma_1$ has a concrete meaning: it is
  designed for \emph{compatibility with the \index{Mayer--Vietoris
      sequence}Mayer--Vietoris sequence}. To illustrate this, first
  recall that for a homotopy commutative diagram
  \[
    \xymatrix{
      A \ar[r] \ar[d] & B \ar[d] \\
      C \ar[r] & D
    }
  \]
  of spectra, we have natural maps $A \to P \leftarrow \Sigma^{-1} D$
  where $P$ is the homotopy pullback.

  Now suppose that we are given a diagram of $\oO$-algebras as above
  which is a homotopy pullback, inducing a boundary map
  $\partial\co \Sigma^{-1} D \to P \simeq A$. Given maps
  $\theta\co N \to E \otimes \Sym^k_\oO(M)$ and
  $\alpha\co \Sigma M \to D$, we can map in a trivial homotopy
  pullback diagram to the above, then apply action maps and naturality
  of the connecting homomorphisms. We get a commuting diagram:
  \[
    \xymatrix{
      N \ar[r]^-\theta \ar[d]_\sim &
      E \otimes \Sym^k_\oO M
      \ar[r]^-{\partial \alpha} \ar[d] &
      A \ar[d]^\sim \\
      \Sigma^{-1} \Sigma N \ar[r] &
      P' \ar[r] &
      P \\
      \Sigma^{-1} \Sigma N \ar[u]^\sim \ar[r]_-{\sigma \theta} &
      \Sigma^{-1} E \otimes \Sym^k_\oO(\Sigma M) \ar[u]
      \ar[r]_-{\Sigma^{-1}\alpha} &
      \Sigma^{-1} D \ar[u]_\partial
    }
  \]
  Therefore, for an operation $\theta\co [M,-] \to [N,-]$ for
  $\oO$-algebras in $\Mod_E$, we find that
  \[
    \partial \circ \sigma\theta \sim \theta \circ \partial.
  \]
  This description makes implicit choices about the
  orientation of the circle that appears in the operation $\Omega$
  when taking homotopy pullbacks, and this can result in \index{signs}sign
  headaches.
\end{remark}

\begin{proposition}
  For $k, r > 0$, the suspension $\sigma_r\co \Pow(m,k) \to \Pow(m+r,k)$
  is the map
  \[
    E_*(P(k)^{m \overline \rho}) \to
    E_*(P(k)^{(m+r) \overline \rho})
  \]
  on $E$-homology induced by the inclusion of virtual bundles $m
  \overline \rho \to m \overline \rho \oplus r \overline \rho$.
\end{proposition}

\begin{proof}
  The assembly map $\Sym^k_\oO(S^m) \otimes S^n \to
  \Sym^k_\oO(S^{m+n})$ is the map
  \[
    (\Sigma^\infty_+ \oO(k) \otimes_{\Sigma_k} S^{m\rho}) \otimes S^n
    \to (\Sigma^\infty_+ \oO(k) \otimes_{\Sigma_k} S^{(m+n)\rho}),
  \]
  which is the map
  \[
    P(k)^{m\rho} \otimes S^r \to P(k)^{(m+r)\rho}
  \]
  induced by the direct sum inclusion
  $m\rho \oplus r \to (m\rho \oplus r) \oplus r \overline \rho$ of
  virtual bundles. The map $\sigma_r$ is obtained by desuspending both
  sides $(m+r)$ times, which gives the map induced by the
  direct sum inclusion
  $m \overline \rho \to m \overline \rho \oplus r \overline \rho$ of
  virtual bundles.
\end{proof}

\begin{example}
  The Dyer--Lashof operations for $E_n$-algebras are explicitly
  \emph{unstable}. For example, in weight two the $n$-fold suspension
  maps $\mb{RP}^{m+n-1}_m \to \mb{RP}^{(m+n) + n-1}_{m+n}$ are
  trivial, and so the map $\Op^E_m(2) \to \Op^E_{m+n}(2)$ is
  trivial. This recovers the well-known fact that all Dyer--Lashof
  operations for $E_n$-algebras map to zero under $n$-fold suspension.

  By contrast, the Dyer--Lashof operations for $E_\infty$-algebras are
  \emph{stable}: the maps $H_* \mb{RP}^\infty_m \to H_*
  \mb{RP}^\infty_{m+1}$ are surjections, and so the quadratic
  operations all lift to elements in the homotopy of
  \[
    \lim_m (H \otimes \mb{RP}^\infty_m).
  \]
  By \cite[16.1]{greenlees-may-tate}, this is the desuspended \index{Tate spectrum}Tate
  spectrum $(\Sigma^{-1} H)^{t\Sigma_2}$.
\end{example}

\begin{remark}
  More generally, the fully stable operations of prime weight $p$ on
  the homotopy of $E_\infty$ $E$-algebras are detected by the
  $p$-localized \index{Tate spectrum}Tate spectrum
  \[
    (\Sigma^{-1} E_{(p)})^{t\Sigma_p}.
  \]
  See \cite[II.5.3]{bmms-hinfty} and \cite{stablepoweroperations}.
\end{remark}
\index{stability|)}

\subsection{Pro-representability}
\label{sec:pro-representability}
\index{representability!by pro-objects|(}

Suppose that $E = \colim E_\alpha$ is an expression of $E$ as a
filtered colimit of finite spectra. Then there is an identification
\[
  E_m A = \colim_\alpha [S^m, E_\alpha \otimes A] =
  \colim_\alpha [S^m \otimes DE_\alpha, A],
\]
where $D$ is the \index{Spanier--Whitehead dual}Spanier--Whitehead dual. We cannot move the
colimit inside, but we can view $\{S^m \otimes DE_\alpha\}$ as a
\emph{\index{pro-object}pro-object} in the category of spectra. This
makes the functor $E_m$ representable by embedding the category of
spectra into the category of pro-spectra.

For algebras over an operad $\oO$, we can go even further and find that
\[
  E_m(A) = [\{\Free_\oO(S^m \otimes DE_\alpha)\},A]_{\pro \oO}
\]
is now a representable functor in the homotopy category of
pro-$\oO$-algebras, and in this category we can determine all the natural
operations $E_m \to E_n$:
\begin{align*}
  Nat_{\pro\oO}(E_m(-), E_n(-))
  &= [\{\Free_\oO(S^n \otimes DE_\alpha)\},
  \{\Free_\oO(S^m \otimes DE_\beta)\}]_{\pro\oO}\\
  &= \pi_0 \lim_\beta \colim_\alpha
    \Map_\oO(\Free_\oO(S^n \otimes DE_\alpha),
    \Free_\oO(S^m \otimes DE_\beta))\\
  &= \pi_0 \lim_\beta \colim_\alpha \Map_{\Sp} (S^n \otimes DE_\alpha,
    \Free_\oO(S^m \otimes DE_\beta))\\
  &= \pi_0 \lim_\beta \Map_{\Sp} (S^n,
    E \otimes \Free_\oO(S^m \otimes DE_\beta))\\
  &= \pi_n \lim_\beta E \otimes (\Free_\oO(S^m \otimes DE_\beta)).
\end{align*}

The algebra of natural transformations has natural maps in from the
group
\[
  [S^m \otimes E, S^n \otimes E]
\]
of cohomology operations (and these maps are isomorphisms if $\oO$ is
trivial), and it has a natural map to the limit
\[
  \lim_\beta E_n(\Free_\oO(S^n \otimes DE_\beta)).
\]
This map to the limit is an isomorphism if no higher derived functors
intrude. We can think of this as the algebra of \emph{continuous}
\index{continuous operations}operations on $E$-homology.

\index{operations|)}
\index{representability!by pro-objects|)}

\section{Classical operations}
\label{sec:classical}

\subsection{$E_n$ Dyer--Lashof operations at $p=2$}

We will now specialize to the case of ordinary mod-$2$ homology. When
we do so, we have \index{Thom isomorphism}Thom isomorphisms for many bundles and we have
explicit computations of the homology of \index{configuration space}configuration spaces due to
\index{Cohen}Cohen \cite{cohen-lada-may-homology}. Similar results with more
complicated identities hold at odd primes.

\begin{proposition}
  Let $H = H\mb F_2$ be the mod-$2$ \index{Eilenberg--Mac Lane spectrum}Eilenberg--Mac Lane spectrum. Then
  the group $\Op^H_m(2)$ of weight-2 operations for $E_n$-algebras has
  exactly one nonzero operation in each degree between $m$ and
  $m+n-1$, and no others.
\end{proposition}

\begin{proof}
  By Remark~\ref{rmk:thom}, this is a calculation
  $H_*(\mb{RP}^{n+m-1}_m)$ of the mod-$2$ homology of \index{stunted projective space}stunted
  projective spaces.
\end{proof}

\index{Dyer--Lashof operations!properties}
\begin{theorem}[{\cite[III.3.1, III.3.2, III.3.3]{bmms-hinfty}}]
  \label{thm:omnibus-en}
  Let $H = H\mb F_2$ be the mod-$2$ Eilenberg--Mac Lane spectrum. 
  Then $E_n$-algebras in $\Mod_{H}$ have \emph{Dyer--Lashof
    operations}
  \[
    Q_i\co \pi_m \to \pi_{2m+i}
  \]
  for $0 \leq i \leq {n-1}$. These satisfy the following formulas.
  \begin{quote}
    \begin{description}
    \item[\index{additivity}Additivity:] $Q_r(x+y) = Q_r(x) + Q_r(y)$ for $r < n-1$.
    \item[Squaring:] $Q_0 x = x^2$.
    \item[Unit:] $Q_j 1 = 0$ for $j > 0$.
    \item[\index{Cartan formula}Cartan formula:] $Q_r(xy) = \sum_{p+q = r} Q_p(x) Q_q(y)$
      for $r < n-1$.
    \item[\index{Adem relations}Adem relations:]
      $Q_r Q_s(x) = \sum \binom{j-s-1}{2j-r-s} Q_{r + 2s - 2j} Q_j(x)$
      for $r > s$.
    \item[\index{stability}Stability:] $\sigma Q_0 = 0$, and $\sigma Q_r = Q_{r-1}$ for
      $r > 0$.
    \item[Extension:] If an $E_n$-algebra structure extends to an
      $E_{n+1}$-algebra structure, the operations $Q_r$ for
      $E_{n+1}$-algebras coincide with the operations $Q_r$ for
      $E_n$-algebras.
    \end{description}
  \end{quote}
  There is also a bilinear \emph{\index{Browder bracket}Browder bracket}
  \[
    [-,-]\co \pi_r \otimes \pi_s \to \pi_{r+(n-1)+s}
  \]
  satisfying the following formulas.
  \begin{quote}
    \begin{description}
    \item[\index{antisymmetry}Antisymmetry:] $[x,y] = [y,x]$ and $[x,x] = 0$.
    \item[Unit:] $[x,1] = 0$.
    \item[\index{Leibniz rule}Leibniz rule:] $[x,yz] = [x,y]z + y[x,z]$.
    \item[\index{Jacobi identity}Jacobi identity:] $[x,[y,z]] + [y,[z,x]] + [z,[x,y]] = 0$.
    \item[Dyer--Lashof vanishing:] $[x,Q_r y] = 0$ for $r < n-1$.
    \item[Top \index{additivity}additivity:]
      $Q_{n-1}(x+y) = Q_{n-1} x + Q_{n-1} y + [x,y]$.
    \item[Top \index{Cartan formula}Cartan formula:]
      $Q_{n-1}(xy) = \sum_{p+q = n-1} Q_p(x) Q_q(y) + x [x,y] y$.
    \item[\index{adjoint identity}Adjoint identity:] $[x,Q_{n-1}y] = [y,[y,x]]$.
    \item[Extension:] If an $E_n$-algebra structure extends to an
      $E_{n+1}$-algebra structure, the bracket is identically zero.
    \item[$E_1$-bracket:] $[x,y] = xy + yx$ if $n=1$.
    \end{description}
  \end{quote}
\end{theorem}

\begin{remark}
  There are two common \index{indexing convention}indexing conventions
  for the Dyer--Lashof operations. This lower-indexing convention is
  designed to emphasize the range where the operations are
  defined, and is especially useful for \index{$E_n$-algebra}$E_n$-algebras. The upper-indexing convention defines
  $Q^s x = Q_{s - |x|} x$ so that $Q^s$ is always a natural
  transformation $\pi_m \to \pi_{s+m}$, with the understanding that
  $Q^s x = 0$ for $s < |x|$.
\end{remark}

\begin{example}
  Suppose that $X$ is an $n$-fold loop space, so that $H[X]$
  is an $E_n$-algebra in left $H$-modules. Then we recover the
  classical Dyer--Lashof operations
  \[
    Q_r\co H_n(X) \to H_{2n+r}(X)
  \]
  in the homology of \index{iterated loop space}iterated loop spaces.
\end{example}

\index{free algebra!homology}
\begin{theorem}[{\cite[IX.2.1]{bmms-hinfty},
    \cite[III.3.1]{cohen-lada-may-homology}}]
  For any spectrum $X$ and any $1 \leq n \leq \infty$,
  $H_*(\Free_{E_n}(X))$ is the free object $\mb Q_{E_n}(H_* X)$ in the
  category of graded $\mb F_2$-algebras with Dyer--Lashof operations
  and Browder bracket satisfying the identities of
  Theorem~\ref{thm:omnibus-en}.
\end{theorem}

\begin{remark}
  This theorem is the analogue of the calculation of the cohomology of
  Eilenberg--Mac Lane spaces as free algebras in a category of
  algebras with \index{Steenrod operations}Steenrod operations. As such, it means that we have a
  \emph{complete} theory of \index{homotopy operations}homotopy operations for $E_n$-algebras
  over $H$.
\end{remark}

\begin{example}
  In the case $n < \infty$ we can give a straightforward description
  of $\mb Q_{E_n} V$ if $V$ has a basis with a single generator
  $e$. In this case, the antisymmetry, unit, and Dyer--Lashof
  vanishing axioms can be used to show that the free algebra has
  trivial Browder bracket, and so the free algebra $\mb Q_{E_n}(V)$ is
  a graded polynomial algebra
  \[
    \mb F_2[Q_J e]
  \]
  as we range over generators $Q_J e = (Q_1)^{j_1} (Q_2)^{j_2} \dots
  (Q_{n-1})^{j_{n-1}} e$.
\end{example}

\subsection{$E_\infty$ Dyer--Lashof operations at $p=2$}

When $n = \infty$, the results of the previous section become
significantly simpler, and it is worth expressing using the upper
indexing for Dyer--Lashof operations.

\index{Dyer--Lashof operations!properties}
\begin{theorem}[{\cite[III.1.1]{bmms-hinfty}}]
  \label{thm:omnibus}
  Let $H = H\mb F_2$ be the mod-$2$ Eilenberg--Mac Lane spectrum. 
  Then $E_\infty$-algebras in $\Mod_{H}$ have \emph{Dyer--Lashof
    operations}
  \[
    Q^r\co \pi_m \to \pi_{m+r}
  \]
  for $r \in \mb Z$. These satisfy the following formulas.
  \begin{quote}
    \begin{description}
    \item[\index{additivity}Additivity:] $Q^r(x+y) = Q^r(x) + Q^r(y)$.
    \item[\index{instability relations}Instability:] $Q^r x = 0$ if $r < |x|$.
    \item[Squaring:] $Q^r x = x^2$ if $r = |x|$.
    \item[Unit:] $Q^r 1 = 0$ for $r \neq 0$.
    \item[\index{Cartan formula}Cartan formula:] $Q^r(xy) = \sum_{p+q = r} Q^p(x) Q^q(y)$.
    \item[\index{Adem relations}Adem relations:]
      $Q^r Q^s = \sum \binom{i-s-1}{2i-r} Q^{s+r-i} Q^i$ for $r > 2s$.
    \item[Stability:] $\sigma Q^r = Q^r$.
    \end{description}
  \end{quote}
\end{theorem}

\begin{example}
  For any space $X$, $H^X$ is an $E_\infty$-algebra in the category of
  left $H$-modules, and hence it has Dyer--Lashof operations
  \[
    Q^i\co H^n(X) \to H^{n-i}(X).
  \]
  It turns out that these are precisely the \emph{\index{Steenrod operations}Steenrod operations}:
  \[
    Sq^i = Q^{-i}.
  \]
  From this point of view, the identity $Q^0 x = x$ is not obvious. In
  fact, \index{Mandell}Mandell has shown that this identity is characteristic of
  algebras that come from spaces: the functor
  $X \mapsto (H\overline{\mb F}_p)^X$ from spaces to
  $E_\infty$-algebras over the \index{Eilenberg--Mac Lane spectrum}Eilenberg--Mac Lane spectrum
  $H\overline{\mb F}_p$ is fully faithful, and the essential image is
  detected in terms of the coefficient ring being generated by classes
  that are annihilated by the analogue at arbitrary primes of the
  identity $(Q^0 - 1)$ \cite{mandell-einftypadic}.
\end{example}

\begin{example}
  In the case $n = \infty$ there is always a straightforward basis for
  the \index{free algebra}free algebra. If $\{e_i\}$ is a basis of a graded vector space
  $V$ over $\mb F_2$, then the free algebra $\mb Q_{E_\infty}(V)$ is a
  graded polynomial algebra
  \[
    \mb F_2[Q^J e_i]
  \]
  as we range over generators $Q^J e_i = Q^{j_1} \dots Q^{j_p} e_i$
  such that $j_i \leq 2j_{i+1}$ and $j_1 - j_2 - \dots - j_p > |e_i|$.
\end{example}

\subsection{Iterated loop spaces}
\index{iterated loop space}
The following is an unpointed \index{group-completion}group-completion theorem for
$E_n$-spaces.

\begin{theorem}[{\cite[III.3.3]{cohen-lada-may-homology}}]
  For any space $X$ and any $1 \leq n \leq \infty$, the map
  $X \to \Omega^n \Sigma^n X_+$ induces a map \index{free algebra}
  $\Free_{E_n}(X) \to \Omega^n \Sigma^n X_+$, and the resulting ring map
  \[
    \mb Q_{E_n}(H_* X) = H_*(\Free_{E_n}(X)) \to
    H_*(\Omega^n \Sigma^n X_+)
  \]
  is a localization which inverts the images of $\pi_0(X)$.
\end{theorem}

\begin{remark}
  A pointed version of the group-completion theorem, involving
  $\Omega^n \Sigma^n X$, is much more standard and implies this
  one. This theorem holds for $\Omega^n \Sigma^n$ if we replace
  $\Free_{E_n}$ with a version that takes the basepoint to a unit and
  we replace $\mb Q_{E_n}(H_* X)$ with either
  $\mb Q_{E_n}(\widetilde H_* X)$ a reduced version
  $\widetilde{\mb Q}_{E_n}$ that sends a chosen element to the unit.
  However, we wanted to give a version that de-emphasizes implicit
  basepoints for comparison with \S\ref{sec:TAQ}.
\end{remark}

\begin{proposition}
  Suppose $Y$ is a pointed space. Then the suspension map
  \[
    \sigma\co \widetilde H_*(\Omega^n Y) \to \widetilde H_{*+1}
    (\Omega^{n-1} Y),
  \]
  induced by the map $\Sigma \Omega^n Y \to \Omega^{n-1} Y$, is
  compatible with the Dyer--Lashof operations and the \index{Browder bracket}Browder bracket:
  \begin{align*}
    \sigma(Q^r x) &= Q^r(\sigma x)\\
    \sigma[x,y] &= [\sigma x, \sigma y]
  \end{align*}
  In particular, in the bar \index{spectral sequence!bar}spectral sequence 
  \[
    \Tor^{H_* \Omega^n Y}_{**} (\mb F_2, \mb F_2) \Rightarrow
    H_* \Omega^{n-1} Y,
  \]
  the operations on the image $\widetilde H_* \Omega^n Y
  \twoheadrightarrow \Tor_1^{H_* \Omega^n Y}(\mb F_2, \mb
  F_2)$ are representatives for the operations on $H_* \Omega^{n-1}
  Y$.
\end{proposition}

This provides some degree of conceptual interpretation for the bracket
and the Dyer--Lashof operations. Since $H_* \Omega^2 Y$ is
commutative, the $\Tor$-algebra is also commutative even though it is
converging to the possibly noncommutative ring $H_* \Omega Y$, and so
the noncommutativity is tracked by multiplicative extensions in the
spectral sequence \cite{ni-bracket}. The Browder bracket in
$H_* \Omega^n Y$ exists to remember that, after $n-1$ deloopings,
there are commutators $xy \pm yx$ in $H_* \Omega Y$.

Similarly, elements in positive filtration in the $\Tor$-algebra of a
commutative ring always satisfy $x^2 = 0$, even though this may not be
the case in $H_* \Omega^{n-1} Y$. The element $Q_0 x$ is $x^2$; the
elements $Q_1 x, Q_2 x, \dots, Q_{n-1} x$ determine the line of
succession for the property of being $x^2$ as the delooping process is
iterated.

\begin{remark}
  \label{rmk:barsseq}
  The group-completion theorem allows us to relate the homology of a
  delooping to certain nonabelian \index{derived functors!nonabelian}derived functors
  \cite{miller-delooping}. Similar spectral sequences computing
  $E_n$-homology of chain complexes have been studied by \index{Richter}Richter and
  \index{Ziegenhagen}Ziegenhagen \cite{richter-ziegenhagen-spectralsequence}.

  Associated to the $n$-fold loop space $\Omega^n Y$ of an
  $(n-1)$-connected space, which is an $E_n$-algebra (or an infinite
  loop space $\Omega^\infty Y$ associated to a connective
  spectrum), we can construct three augmented simplicial objects:
  \begin{align*}
    \cdots \Free_{E_n} \Free_{E_n} \Free_{E_n} \Omega^n Y \Rrightarrow
    \Free_{E_n} \Free_{E_n} \Omega^n Y \Rightarrow \Free_{E_n}
    \Omega^n Y
    & \rightarrow \Omega^n Y\\
    \cdots \Omega^n \Sigma^n_+ \Free_{E_n} \Free_{E_n}  \Omega^n Y \Rrightarrow \Omega^n
    \Sigma^n_+ \Free_{E_n} \Omega^n Y \Rightarrow \Omega^n \Sigma^n_+
    \Omega^n Y
    & \rightarrow \Omega^n Y\\
    \cdots \Sigma^n_+ \Free_{E_n} \Free_{E_n} \Omega^n Y \Rrightarrow
    \Sigma^n_+ \Free_{E_n} \Omega^n Y \Rightarrow \Sigma^n_+ \Omega^n Y
    & \rightarrow Y
  \end{align*}
  These are, respectively, two-sided \index{bar construction}bar constructions: $B(\Free_{E_n}, \Free_{E_n},
  \Omega^n Y)$, $B(\Omega^n \Sigma^n_+, \Free_{E_n}, \Omega^n Y)$, and
  $B(\Sigma^n_+, \Free_{E_n}, \Omega^n Y)$.

  The first augmented bar construction $B(\Free_{E_n}, \Free_{E_n}, \Omega^n
  Y)$ has an extra degeneracy, and so its \index{geometric realization}geometric realization is
  homotopy equivalent to $\Omega^n Y$ as $E_n$-spaces. Therefore, it
  is a group-complete $E_n$-space.

  There is a natural map
  \[
    B(\Free_{E_n}, \Free_{E_n}, \Omega^n Y) \to B(\Omega^n \Sigma^n_+,
    \Free_{E_n}, \Omega^n Y)
  \]
  which is, levelwise, a \index{group-completion}group-completion map \cite[Appendix
  Q]{quillen-appendixq}, \cite{mcduff-segal-groupcompletion}, and induces a
  group-completion map on geometric realization. However, the source
  is already group-complete, and so this map is an equivalence on
  geometric realizations. Thus, the augmentation
  $|B(\Omega^n \Sigma^n_+, \Free_{E_n}, \Omega^n Y)| \to \Omega^n Y$ is an
  equivalence.

  The bar construction $B(\Sigma^n_+, \Free_{E_n}, \Omega^n Y)$ is a
  simplicial diagram of $(n-1)$-connected pointed spaces,
  and so by a theorem of May \cite{may-loopspaces} we can commute
  $\Omega^n$ across geometric realization. The natural augmentation
  \[
    \Omega^n |B(\Sigma^n_+, \Free_{E_n}, \Omega^n Y)| \to |B(\Omega^n
    \Sigma^n_+, \Free_{E_n}, \Omega^n Y)| \to \Omega^n Y
  \]
  is an equivalence. By assumption, $Y$ is $(n-1)$-connected and so
  $|B(\Sigma^n_+, \Free_{E_n}, \Omega^n Y)| \to Y$ is also an
  equivalence. Therefore, the simplicial object
  $B(\Sigma^n_+, \Free_{E_n}, \Omega^n Y)$ can be used to compute
  $H_* Y$.
  
  Let $A = H_*(\Omega^n Y)$. The reduced homology of
  $B(\Sigma^n_+, \Free_{E_n}, \Omega^n Y)$ is
  \[
    \cdots \Sigma^n \mb Q_{E_n} \mb Q_{E_n} A \Rrightarrow \Sigma^n
    \mb Q_{E_n} A \Rightarrow \Sigma^n  A,
  \]
  which is a bar complex $\Sigma^n B(Q,\mb Q_{E_n}, A)$ computing
  nonabelian derived functors. These are specifically the derived
  functors of an \emph{\index{indecomposables}indecomposables} functor
  $Q$, which takes an augmented $\mb Q_{E_n}$-algebra $A \to \mb F_2$
  and returns the quotient of the augmentation ideal by all products,
  brackets, and Dyer--Lashof operations. The result is a \index{spectral sequence!Miller} Miller spectral
  sequence that begins with nonabelian derived functors of $Q$ on
  $H_*(\Omega^n Y)$ and converges to $\widetilde H_* Y$.
\end{remark}

\subsection{Classical groups}

The Dyer--Lashof operations on the homology of the spaces $BO$ and
$BU$, and hence on the homology of the \index{Thom spectrum}Thom spectra $\index{$MO$}MO$
and $\index{$MU$}MU$, was determined by work of \index{Kochman}Kochman \cite{kochman-dyerlashof};
here we will state a form due to \index{Priddy}Priddy \cite{priddy-dyerlashof}.

\begin{theorem}
  \label{thm:bordismoperations}
  The ring $H_* MO \cong H_* BO$ is a polynomial algebra on classes
  $a_i$ in degree $i$. The Dyer--Lashof operations are determined
  by the identities of formal series
  \[
    \sum_j Q^j a_k = \left(\sum_{n=k}^\infty \sum_{u=0}^k
      \binom{n-k+u-1}{u} a_{n+u} a_{k-u}\right)\left(\sum_{n=0}^\infty
      a_n\right)^{-1},
  \]
  where $a_0 = 1$ by convention. In particular, $Q^n a_k \equiv
  \binom{n-1}{k} a_{n+k}$ mod decomposable elements.

  The ring $H_* MU \cong H_* BU$ is a polynomial algebra on classes
  $b_i$ in degree $2i$, The Dyer--Lashof operations are determined
  by the identities of formal series
  \[
    \sum_j Q^j b_k = \left(\sum_{n=k}^\infty \sum_{u=0}^k
      \binom{n-k+u-1}{u} b_{n+u} b_{k-u}\right)\left(\sum_{n=0}^\infty
      b_n\right)^{-1},
  \]
  where $b_0 = 1$ by convention. In particular, $Q^{2n} b_k \equiv
  \binom{n-1}{k} b_{n+k}$ mod decomposable elements, and $Q^{2n+1} b_k
  = 0$.
\end{theorem}

\begin{remark}
  Implicit in this calculation is the fact that the \index{Thom isomorphism}Thom isomorphisms
  $H_* MO \cong H_* BO$ and $H_* MU \cong H_* BU$ preserve
  Dyer--Lashof operations. \index{Lewis}Lewis showed that, for an $E_n$-map $f\co X
  \to BGL_1(\mb S)$, the Thom isomorphism $H_* X \cong H_* Mf$ lifts
  to an equivalence of $E_n$ ring spectra
  \[
    H [X] \to H \otimes Mf
  \]
  called the Thom diagonal \cite[7.4]{lewis-may-steinberger}. As a
  result, the Thom isomorphism is automatically compatible with
  Dyer--Lashof operations for $H$-algebras.
\end{remark}

\begin{example}
  \label{ex:MUcalcs}
  We have explicit calculations of the first few Dyer--Lashof
  operations in $H_* MO$:
  \begin{align*}
    Q^2 a_1 &= a_1^2\\
    Q^4 a_1 &= a_3 + a_1 a_2 + a_1^3\\
    Q^6 a_1 &= a_1^4\\
    Q^8 a_1 &= a_5 + a_1 a_4 + a_2 a_3 + a_1^2 a_3 + a_1 a_2^2 + a_1^3
    a_2 + a_1^5\\
    Q^6 a_2 &= a_5 + a_1 a_4 + a_2 a_3 + a_1 a_2^2
  \end{align*}
  These same formulas hold for the $b_i$ in $H_* MU$.
\end{example}

\subsection{The Nishida relations and the \index{dual Steenrod algebra}dual Steenrod algebra}
\index{Nishida relations|(}

Recall that, if $R$ is an $E_n$-algebra in $\Sp$, $H \otimes R$
is an $E_n$-algebra in $\Mod_H$ whose homotopy groups are the
homology groups of $R$. As a result, there are two types of operations
on $H_* R$:
\begin{itemize}
\item The $E_n$-algebra structure gives $H_*(R)$ Dyer--Lashof
  operations $Q_0, \dots, Q_{n-1}$ and a \index{Browder bracket}Browder bracket.
\item The property of being homology gives $H_*(R)$ \index{Steenrod operations}Steenrod
  operations $P_d\co H_m R \to H_{m-d} R$. To make these dual to the
  Steenrod operations $Sq^d$ in cohomology, $P_d(x)$ is defined as a
  composite
  \[
    S^{m-d} \too{\Sigma^{-d} x} (\Sigma^{-d} H) \otimes R \too{\chi
      Sq^d} H \otimes R.
  \]
  This implicitly reverses multiplication order: for example,
  the \index{Adem relations}Adem relation $Sq^3 = Sq^1 Sq^2$ becomes $P_3 = P_2 P_1$.
\end{itemize}

The Nishida relations express how these structures interact.
\begin{theorem}[{\cite[III.1.1, III.3.2]{bmms-hinfty}}]
  Suppose that $R$ is an $E_n$-algebra in $\Sp$. Then the Steenrod
  operations in homology satisfy relations as follows.
  \begin{quote}
    \begin{description}
    \item[\index{Cartan formula}Cartan formula:] $P_r (xy) = \sum_{p+q=r} P_p(x) P_q(y)$.
    \item[\index{Browder Cartan formula}Browder Cartan formula:]
      $P_r [x,y] = \sum_{p+q=r} [P_p x, P_q y]$.
    \item[Nishida relations:]
      $P_r Q^s = \sum \binom{s-r}{r-2i} Q^{s-r+i} P_i$ if $s < n-1$.
    \item[Top Nishida relation:]
      \[P_r Q^{n-1}(x) = \sum \binom{n-1-r}{r-2i} Q^{n-1-r+i} P_i +
      \sum_{p+q=r, p<q} [P_p x, P_q x].\]
    \end{description}
  \end{quote}
\end{theorem}

\index{lower-indexed operation|see{indexing convention}}
\index{upper-index operation|see{indexing convention}}
\begin{remark}
  By contrast with the \index{Adem relations}Adem relations, the Nishida relations behave
  very differently if we use lower \index{indexing convention}
  indexing. We find
  \[
    P_r Q_s(x) = \sum \binom{|x| + s - r}{r - 2i} Q_{s-r+i} P_i(x).
  \]
  In particular, the lower-indexed Nishida relations depend on the
  degree of $x$ \cite{campbell-cohen-peterson-selich-selfmapsI}.
\end{remark}

\begin{remark}
  \index{representability!by pro-objects}
  If we use the pro-representability of homology as in
  \S\ref{sec:pro-representability}, we can obtain a combined
  algebraic object that encodes both the $Q^r$ and the $P_d$ together
  with the Nishida relations.
\end{remark}

\index{Nishida relations|)}

\subsection{Eilenberg--Mac Lane objects}
\index{dual Steenrod algebra|(}

If the homology $H_* R$ is easily described a module over the Steenrod
algebra, the Nishida relations can completely determine the
Dyer--Lashof operations. This was applied by \index{Steinberger}Steinberger to compute
the Dyer--Lashof operations in the \index{dual Steenrod algebra}dual Steenrod algebra
explicitly. (Conversely, \index{Baker}Baker showed that the \index{Nishida relations}Nishida relations
themselves are completely determined by the Dyer--Lashof operation
structure of the dual Steenrod algebra
\cite{baker-power-operations-coactions}.)

\index{Dyer--Lashof operations!in the dual Steenrod algebra}
\begin{theorem}[{\cite[III.2.2, III.2.4]{bmms-hinfty}}]
  Let $A_*$ be the \index{dual Steenrod algebra}dual Steenrod algebra
  \[
    \mb F_2[\xi_1, \xi_2, \dots]
  \]
  where $|\xi_i| = 2^i - 1$, with conjugate generators $\xx_i$ (here
  $\xi_i$ is denoted by $\zeta_i$ in \cite{milnor-steenrod}). Then the
  Dyer--Lashof operations on the generators are determined by the
  following formulas.
  \begin{enumerate}
  \item There is an identity of formal series
    \[
      (1 + \xi_1 + Q^1 \xi_1 + Q^2 \xi_1 + Q^3 \xi_1 + \dots) =
      (1 + \xi_1 + \xi_2 + \xi_3 + \dots)^{-1}.
    \]
  \item For any $i$, we have
    \[
      Q^s \xx_i = \begin{cases}
        Q^{s + 2^i - 2} \xi_1 &\text{if }s\equiv 0, -1 \mod 2^i,\\
        0 &\text{otherwise.}
      \end{cases}
    \]
  \item In particular, $Q^{2^i - 2} \xi_1 = \xx_i$, and
    $Q_1 \xx_i = \xx_{i+1}$.
  \end{enumerate}
\end{theorem}

\begin{remark}
  This allows us to say that the dual Steenrod algebra can
  be re-expressed as follows:
  \[
    A_* \cong \mb F_2[x, Q_1 x, (Q_1)^2 x, \dots]
  \]
  This is the same as the homology of $\Omega^2 S^3$: both are
  identified with the homology of the free $E_2$-algebra on a
  generator $x = \xi_1$ in degree $1$. \index{Mahowald}Mahowald showed that it was
  possible to realize this isomorphism of graded algebras: he
  constructed a \index{Thom spectrum}Thom spectrum over $\Omega^2 S^3$ such that the Thom
  isomorphism realizes the isomorphism $A_* \cong H_* \Omega^2 S^3$
  \cite{mahowald-thomcomplexes}.  This has a rather remarkable
  interpretation: there exists a construction of the
  \index{Eilenberg--Mac Lane spectrum!as an $E_2$-algebra}Eilenberg--Mac Lane spectrum $H$ as the free $E_2$-algebra $R$ such
  that the unit map $\mb S \to R$ has a chosen nullhomotopy of the
  image of $2$. This result has been extended to odd primes by
  \index{Blumberg}Blumberg--\index{Cohen}Cohen--\index{Schlichtkrull}Schlichtkrull
  \cite{blumberg-cohen-schlichtkrull-thom}.
\end{remark}

\begin{proposition}
  Let $Hk$ be the Eilenberg--Mac Lane spectrum for an algebra $k$
  over $\mb F_2$. Then there is an isomorphism
  \[
    H_* H\mb F \cong A_* \otimes k
  \]
  of graded rings, and under this identification the Dyer--Lashof
  operation $Q^r$ on $H_* H k$ is given by $Q^r \otimes \varphi$,
  where $\varphi$ is the Frobenius on $k$.
\end{proposition}

\begin{proof}
  For any $H$-module $N$, the action $H \otimes N \to N$ induces an
  isomorphism $H_* H \otimes \pi_* N \to H_* N$. We already know
  $Q^0(1 \otimes \alpha) = 1 \otimes \alpha^2$, and so by the Cartan
  formula it suffices to show that $Q^s(1 \otimes \alpha) = 0$ for $s
  > 0$.
  
  We now proceed inductively by applying the \index{Nishida relations}Nishida relations. If we
  know $Q^t (1 \otimes \alpha) = 0$ for $0 < t < s$, we find that for
  all $r > 0$ we have
  \begin{align*}
    P_r Q^s(1 \otimes \alpha)
    &= \sum \binom {s-r}{r-2i} Q^{s-r+i} P_i(1 \otimes \alpha)\\
    &= \binom{s-r}{r} Q^{s-r} (1 \otimes \alpha).
  \end{align*}
  By the inductive hypothesis, this vanishes unless $s=r$, but in the
  case $s=r$ the binomial coefficient vanishes. However, the only
  elements in $H_* Hk$ that are acted on trivially by all the
  Steenrod operations are the elements in the image of $\pi_* Hk$,
  and those are concentrated in degree zero. Thus
  $Q^s(1 \otimes \alpha) = 0$.
\end{proof}

\begin{remark}
  The same proof can be used to show that the \index{Browder bracket}Browder bracket is
  trivial on $H_* Hk$.
\end{remark}

\begin{example}
  The composite map $\index{$MU$}MU \to \index{$MO$}MO \to H$, on homology, is given in terms
  of the generators of Theorem~\ref{thm:bordismoperations} by
  $b_1 \mapsto a_1^2 \mapsto \xi_1^2$ and $b_2 \mapsto 0$. The image
  of $H_* MU$ in $A_*$ is $\mb F_2[\xi_1^2, \xi_2^2, \dots]$, the
  homology of the \index{Brown--Peterson spectrum}Brown--Peterson spectrum $BP$.

  In $H_* MU$, Example~\ref{ex:MUcalcs} implies we have the
  identities
  \[
    Q^6 b_2 = b_5 + b_1 b_4 + b_2 b_3 + b_1 b_2^2 = Q^8 b_1 + b_1^2
    Q^4(b_1).
  \]
  By contrast, in the dual Steenrod algebra we have the identity
  $0 = Q^8 (\xi_1^2) + \xi_1^4 Q^4(\xi_1^2)$. Even though the
  map $H_* MU \to H_* BP$ splits as a map of algebras, and the target
  is closed under the Dyer--Lashof operations, we have
  \[
    Q^8 b_1 + b_1^2 Q^4 (b_1) = Q_6(b_1) + b_1^2 Q_2(b_1) \neq 0
  \]
  but its image is zero. This implies that the map $H_* MU \to H_* BP$
  does not have a splitting that respects the Dyer--Lashof operations
  for $E_7$-algebras. As a result, there exists no map
  $BP \to MU_{(2)}$ of $E_7$-algebras. This result, and its analogue
  at odd primes, is due to Hu--\index{Kriz}Kriz--May \cite{hu-kriz-may-cores}.
\end{example}

\index{dual Steenrod algebra|)}
\subsection{Nonexistence results}

The tremendous amount of structure present in the homology of a \index{ring spectrum}ring spectrum allows us to produce a rather large number of nonexistence
results. The following is a generalization of the classical result
that the mod-$2$ \index{Moore spectrum}Moore spectrum does not admit a
multiplication due to the existence of a nontrivial Steenrod operation
$Sq^2$ in its cohomology; we learned this line of argument from
Charles Rezk.

\begin{proposition}
  Suppose that $R$ is a homotopy associative ring spectrum containing
  an element $u$ in nonzero degree such that $P_k(u)$ vanishes either
  in the range $k > |u|$ or in the range $0 < k < |u|$. Then either
  $P_{|u|}(u)$ is nilpotent or $H_* R$ is nonzero in infinitely many
  degrees.
\end{proposition}

\begin{proof}
  We find, by the Cartan formula, that
  \[
    P_{d|u|}(u^d) = (P_{|u|} u)^d.
  \]
  Therefore, either the elements $u^d$ are nonzero for all $d$ or the
  element $P_{|u|} u$ is nilpotent.
\end{proof}

\begin{corollary}
  Suppose that $R$ is a connective homotopy associative ring spectrum
  such that $H_0(R) = \pi_0(R) / 2$ has no nilpotent elements. If any
  nonzero element in $H_0(R)$ is in the image of the Steenrod
  operations, then $H_* R$ must be nonzero in infinitely many degrees.
\end{corollary}

\begin{corollary}
  Suppose that $R$ is a homotopy associative ring spectrum and that
  some \index{Hopf invariant element}Hopf invariant element $2$,
  $\eta$, $\nu$, or $\sigma$ maps to zero under the unit map
  $\mb S \to R$. Then either $H_* R = 0$ or $H_* R$ is
  infinite-dimensional.
\end{corollary}

\begin{proof}
  Writing $h$ for Hopf invariant element in degree $2^k-1$ with trivial
  image, the unit $\mb S \to R$ extends to a map $f\co C(h) \to R$
  from the mapping cone. The homology of $C(h)$ has a basis of
  elements $1$ and $v$ with one nontrivial Steenrod operation acting
  via $P_{2^k} v = 1$, and $u = f_*(v)$ has the desired properties.
\end{proof}

Recall from \S\ref{sec:connectivity} that, for $R$ connective, a map
$\pi_0 R \to A$ of commutative rings automatically extends to a map
$R \to HA$ compatible with the multiplicative structure that exists on
$R$; e.g., if $R$ is homotopy commutative then the map $R \to HA$ has
the structure of a map of $\mathcal{Q}_1$-algebras. This has the
following consequence.
\begin{proposition}
  Suppose that $R$ is a connective ring spectrum with a ring
  homomorphism $\pi_0 R \to k$ where $k$ is an $\mb F_2$-algebra
  (equivalently, a map $H_0 R \to k$). Then there is a map
  $R \to Hk$ which induces a homology map
  $H_* R \to A_*\otimes k$ with the following properties.
  \begin{enumerate}
  \item The map $H_* R \to A_* \otimes k$ is a map of rings which
    is surjective in degree zero.
  \item If $R$ is homotopy commutative, then there is an operation $Q_1$ on
    $H_* R$ that is compatible with the operation $Q_1$ on
    $A_* \otimes k$.
  \item If $R$ has an $E_n$-algebra structure, the map $R \to Hk$
    is a map of $E_n$-algebras and so $H_* R \to A_* \otimes k$ is
    compatible with the Dyer--Lashof operations
    $Q_0,\dots,Q_{n-1}$.
  \end{enumerate}
  In particular, the image of $H_* R$ in $A_* \otimes k$ is a
  subalgebra $B_* \subset A_*$ closed under multiplication and
  some number of Dyer--Lashof operations.
\end{proposition}

\begin{example}
  For $n > 0$ there are connective \index{Morava
    $K$-theory}Morava $K$-theories $k(n)$, with
  coefficient ring $\mb F_2[v_n]$, that have homology
  \[
    \mb F_2[\xx_1, \dots, \xx_n, \xx_{n+1}^2, \xx_{n+2},
    \dots]
  \]
  as a subalgebra of the dual Steenrod algebra. This subring is not
  closed under the Dyer--Lashof operation $Q_1$ unless $n=0$, and so
  the connective Morava $K$-theories are not
  homotopy--commutative. (By convention we often define the connective
  Morava $K$-theory $k(0)$ to be $H\mb Z_2$, which is commutative.)

  Similarly, for $n > 0$ the integral connective Morava $K$-theories
  $k_{\mb Z}(n)$, with coefficient ring $\mb Z_2[v_n]$, have homology
  \[
    \mb F_2[\xx_1^2, \xx_2, \dots, \xx_n, \xx_{n+1}^2, \xx_{n+2},
    \dots]
  \]
  as a subalgebra of the dual Steenrod algebra. This subring is not
  closed under the Dyer--Lashof operation $Q_1$ unless $n=1$, and so
  the only possible homotopy-commutative integral Morava $K$-theory is
  $k_{\mb Z}(1)$---the connective complex $K$-theory spectrum.

  There are obstruction-theoretic proofs which show that all of these have
  $A_\infty$ structures \cite{angeltveit, lazarev-ainfty}.
\end{example}

\begin{example}
  The Dyer--Lashof operations satisfy $Q_2(\xx_i^2) = \xx_{i+1}^2$,
  and so the smallest possible subring of $A_*$ that contains
  $\xi_1^2 = \xx_1^2$ and is closed under $Q_2$ is an infinite
  polynomial algebra $\mb F_2[\xx_i^2] = \mb F_2[\xi_i^2]$. If $R$
  is a connective ring spectrum with a quotient map
  $\pi_0 R \to \mb F$ such that the Hopf element
  $\eta \in \pi_1(\mb S)$ maps to zero in $\pi_* R$, then there is a
  commutative diagram
  \[
    \xymatrix{
      C(\eta) \ar[r] \ar[d] & R \ar[d] \\
      H\mb Z/2 \ar[r] & H\mb F.
    }
  \]
  We conclude that $\xi_1^2$ is in the image of the map $H_* R \to
  H_* H\mb F$.
  
  The spectra $X(n)$ appearing in the nilpotence and periodicity
  theorems of \index{Devinatz}Devinatz--\index{Hopkins}Hopkins--\index{Smith}Smith fit into a sequence
  \[
    X(1) \to X(2) \to X(3) \to \dots
  \]
  of \index{Thom spectrum}Thom spectra on the spaces $\Omega SU(n)$. They have $E_2$-ring
  structures, and each ring $H_* X(n)$ is a polynomial algebra
  $\mb F_2[x_1,\dots,x_{n-1}]$ on finitely many generators. For
  $n = 2$ the map $H_* X(2) \to A_*$ is the map
  $\mb F_2[\xi_1^2] \to A_*$, and this implies that each $X(n)$ has
  $\xi_1^2$ in the image of its homology. As $H_* X(n)$ is finitely
  generated as an algebra, its image in the dual Steenrod algebra is
  too small to be closed under the operation $Q_2$. This excludes the
  possibility that $X(n)$ has an $E_3$-structure.
\end{example}

\subsection{Ring spaces}
\label{sec:ring-spaces}
\index{ring space|(}

Associated to an $E_\infty$ \index{ring spectrum}ring spectrum $E$, there is a sequence of
infinite loop spaces $\{E_n\}_{n \in \mb Z}$ in an \index{$\Omega$-spectrum}$\Omega$-spectrum
representing $E$. These spaces are extremely strongly structured: they
inherit both \emph{additive} structure from the spectrum structure on
$E$, and \emph{multiplicative} structure from the $E_\infty$ ring
structure. In the case of the sphere spectrum, these operations were
investigated in-depth in relationship to surgery theory
\cite{milgram-spherical, madsen-dyerlashof, may-homologyoperations}.
\index{Ravenel}Ravenel and \index{Wilson}Wilson discussed the structure coming from a ring spectrum
$E$ extensively in \cite{ravenel-wilson-hopfring}, encoding it in the
structure of a \emph{\index{Hopf ring}Hopf ring}, and the interaction
between additive and multiplicative operations is developed in-depth
in \cite[\S II]{cohen-lada-may-homology}. These structures are very
tightly wound.

\begin{enumerate}
\item Because the $E_n$ are spaces, the diagonals
  $E_n \to E_n \times E_n$ gives rise to a \emph{coproduct}
  \[
    \Delta\co H_* (E_n) \to H_* (E_n) \otimes H_*(E_n),
  \]
  For an element $x$ we write $\sum x' \otimes x''$ for its coproduct.
  The path components $E^{n} = \pi_0 E_{n}$ also give rise to elements
  $[\alpha] \in H_0 E_{n}$.
\item The homology groups $H_* E_n$ have \index{Steenrod operations}Steenrod operations $P_r$.
\item The suspension maps $\Sigma E_n \to E_{n+1}$ in the spectrum
  structure give \index{stabilization}stabilization maps
  \[
    H_m(E_n) \to H_{m+1}(E_{n+1}).
  \]
\item The infinite loop space structure on $E_n$ gives $H_* E_n$ an
  \emph{additive} \index{Pontrjagin product}Pontrjagin product
  \[
    \hash\co H_* (E_n) \otimes H_*(E_n) \to H_*(E_n)
  \]
  making it into a Hopf algebra, and it has \emph{additive}
  Dyer--Lashof operations
  \[
    Q^r\co H_m (E_n) \to H_{m+r}(E_n).
  \]
\item If $E$ has a \index{ring spectrum}ring spectrum structure, the multiplication
  $E \otimes E \to E$ gives \emph{multiplicative} Pontrjagin products
  \[
    \circ\co H_*(E_n) \otimes H_*(E_m) \to H_*(E_{n+m}).
  \]
  These are appropriately unital, associative, or graded-commutative
  if $E$ has these properties.
\item If $E$ has an $E_\infty$ \index{ring spectrum}ring spectrum structure, there are
  \index{Dyer--Lashof operations!multiplicative}\emph{multiplicative} Dyer--Lashof operations
  \[
    \widetilde Q^r\co H_m (E_0) \to H_{m+r}(E_0)
  \]
  on the homology of the $0$'th space. In general, we cannot say more.
  If $E$ has further structure---an \index{$H^d_\infty$-structure}$H^d_\infty$-structure---there are
  also multiplicative Dyer--Lashof operations outside degree zero
  \cite[\S 4.1]{secondary}.
\end{enumerate}

These are subject to a large number of identities discussed in
\cite[1.12, 1.14]{ravenel-wilson-hopfring}, \cite[II.1.5, II.1.6,
II.2.5]{cohen-lada-may-homology}, and
\cite[1.5]{kuhn-jameshopfhomology}. Here are the most fundamental 
identities:
\begin{quote}
  \begin{description}
  \item[\index{distributive rule}Distributive rule:]
    $(x\hash y)\circ z= \sum (x\circ z')\hash(y \circ z'')$
  \item[\index{projection formula}Projection formula:]
    $x \circ Q^s y = \sum Q^{s+k} (P_k x \circ y)$
  \item[\index{Cartan formula!mixed}Mixed Cartan formula:]
    \[\widetilde Q^n(x \hash y) = \sum_{p+q+r = n} \widetilde Q^p(x')
    \hash Q^q(x'' \circ y') \hash \widetilde Q^r(y'')\]
  \item[\index{Adem relations!mixed}Mixed Adem relations:]
    \[\widetilde Q^r Q^s x = \sum_{i+j+k+l = r+s}
    \binom{r-i-2l-1}{j+s-i-l} Q^i \widetilde Q^j x' \hash Q^k
    \widetilde Q^l x''\]
  \end{description}
\end{quote}
\begin{example}
  There is an identity 
  \[
    Q^1 [a] \hash [-2a] = \eta \cdot a
  \]
  which allows us to determine information about the
  multiplication-by-$\eta$ map
  $\pi_0 R \to \pi_1 R \to H_1 \Omega^\infty R$ from the additive
  Dyer--Lashof structure. Similarly $\widetilde Q^1$ determines
  information about its multiplicative version
  $\eta_m\co \pi_0(R) \to \pi_1(R)$. For example, the mixed Cartan
  formula implies that
  \[
    \eta_m(x + y) = \eta_m(x) + \eta_m(y) + \eta \cdot xy
  \]
  in $H_1(R)$. In particular,
  $\widetilde Q^1 [n] = \binom {n}{2} \eta \hash [n^2]$ for
  $n \in \mb Z$ (cf. Example~\ref{ex:cupone}).
\end{example}

\index{ring space|)}
\section{Higher-order structure}
\index{secondary operations|(}
\subsection{Secondary composites}

Secondary operations, at their core, arise
when there are relations between relations. Suppose that we are a
sequence $X_0 \too{f} X_1 \too{g} X_2 \too{h} X_3$ of maps such that
the double composites are nullhomotopic. Then $hgf$ is nullhomotopic
for two reasons. Choosing nullhomotopies of $gf$ and $hg$, we can glue
the nullhomotopies together to determine a loop in the space of maps
$X \to W$: a \emph{value} of the associated secondary
operation. Because we must make choices of nullhomotopy, there is some
natural \index{indeterminacy}indeterminacy in this construction, and
so it typically takes a set of values $\langle h, g, f\rangle$. To
construct secondary operations, we minimally need to work in a
category $\cC$ with mapping spaces; we also need canonical basepoints
of the spaces $\Map_\cC(X_i, X_j)$ for $j \geq i+2$ that are preserved
under composition \cite[\S 2]{secondary}.

\begin{example}
  Suppose that $A$ is a subspace of $X$ and
  $\alpha \in H^n(X,A)$ is a cohomology element that
  restricts to zero in $H^n(X)$. Then the long exact sequence in
  cohomology implies that we can lift $\alpha$ to an element in
  $H^{n-1}(A)$, but there are multiple choices of lift. This can be
  represented by a sequence of maps
  \[
    A \to X \to X/A \to K(\mb Z,n)
  \]
  where the double composites are nullhomotopic; the secondary
  operation is then a map $A \to \Omega K(\mb Z,n) = K(\mb Z,n-1)$.
\end{example}

Secondary operations enrich the \index{homotopy category}homotopy category $h\cC$ with extra
structure.
\begin{enumerate}
\item Every test object $T \in \cC$ represents a functor
  $[T,-] = \pi_0 \Map_{\cC}(T,-)$ on $h\cC$, and if $T$ has an
  \index{augmentation}augmentation $T \to 0$ to an initial object then this functor has
  a canonical null element. If the values of $[T,-]$ differ on $X$ and
  $Y$, $X$ and $Y$ cannot be equivalent in $h\cC$.
\item Every map of test objects $\Theta\co S \to T$ determines an \index{operations}operation: a
  natural transformation of functors $\theta \co [T,-] \to [S,-]$ on
  $h\cC$. If $S$ and $T$ are augmented and the map $\Theta$ is compatible with
  the augmentations, then $\theta$ preserves the null element. If
  $\theta$ has different behaviour for $X$ and $Y$, $X$ and $Y$ cannot
  be equivalent.
\item Given an augmented map $\Phi\co R \to S$ and a map $\Theta\co S \to T$ such that
  the double composite $\Theta \Phi\co R \to T$ is trivial, we get an identity
  $\varphi \theta = 0$ of associated operations. There is an associated secondary
  operation $\langle-,\Theta, \Phi\rangle$. It is only defined on those elements $\alpha \in [T,X]$
  with $\theta(\alpha) = 0$; it takes values in
  $\pi_1 \Map_{\cC}(R,X)$; it is only well-defined up to
  indeterminacy.
\item We can also associate information to maps in the same way. Suppose we
  have an augmented map $\Theta\co S \to T$ of test objects representing an
  operation $\theta$. Given any map $f\co X \to Y$, there is an
  associated \index{functional operations}functional operation
  $\langle f, -, \Theta\rangle$. It is
  only defined on those elements $\alpha \in [T,X]$ such that
  $f(\alpha) = 0$ and $\theta(\alpha) = 0$; it takes values in
  $\pi_1 \Map_{\cC}(S,Y)$; it is only well-defined up to
  indeterminacy.
\end{enumerate}

Applying this to the test objects $S^n$ in the category of pointed
spaces, we get \index{Toda bracket} Toda's bracket construction that enriches the homotopy
groups of spaces with secondary composites. Applying this to the test
objects $K(A,n)$ in the opposite of the category of spaces, we get
Adams' secondary operations that enrich the cohomology groups of
spaces with secondary \index{cohomology operations!secondary}cohomology operations.

\subsection{Secondary operations for algebras}
\label{sec:secondaryalgs}

We recall from \S \ref{sec:representability} that, for a spectrum $M$
and an $\oO$-algebra in $\Mod_E$, we have
\[
  \pi_M(A) = [M,A]_{\Sp} \cong [E \otimes \Free_{\oO}(M),
  A]_{\Alg_\oO(\Mod_E)}.
\]
Using \index{free algebra}free algebras as our test objects, we already used this
representability of homotopy groups to classify the natural operations
on the homotopy groups of $\oO$-algebras in $\Mod_E$. The space of
maps now means that we can construct secondary operations.

\begin{proposition}
  Suppose that we have zero-preserving operations
  $\theta\co \pi_M \to \pi_N$ and $\varphi\co \pi_N \to \pi_P$ on the
  homotopy category of $\oO$-algebras in $\Mod_E$, and that there is a
  relation $\varphi \circ \theta = 0$. Then there exists a
  \index{secondary operations}secondary
  operation
  \[
    \langle -, \Theta, \Phi\rangle\co \pi_M(A) \supset \ker \theta \to \pi_{P+1}(A) / Im(\sigma \varphi),
  \]
  where $\sigma(\varphi)$ is a \index{suspension!of operations}suspended operation (see \S\ref{sec:stability}).
\end{proposition}

Such a secondary operation is constructed from a sequence
\[
  E \otimes \Free_\oO(P) \too{\Phi}
  E \otimes \Free_\oO(N) \too{\Theta}
  E \otimes \Free_\oO(M) \to A
\]
where the double composites are null; the nullhomotopy of
$\Theta \circ \Phi$ is chosen once and for all, while the second
nullhomotopy is allowed to vary. This produces elements in
\[
  \pi_1 \Map_{\Alg_\oO(\Mod_E)}(E \otimes \Free_{\oO}(P), A) \cong
  \pi_1 \Map_\Sp(P, A) \cong [\Sigma P, A].
\]

\begin{example}
  Every \index{Adem relations}Adem relation between Dyer--Lashof operations produces a
  secondary Dyer--Lashof operation. For example, the relation
  $Q^{2n+2} Q^n + Q^{2n+1} Q^{n+1} = 0$ produces a natural transformation
  \[
    \pi_m(A) \supset \ker(Q^n, Q^{n+1}) \to \pi_{m + 2n + 3}(A) /
    Im(Q^{2n+2}, Q^{2n+1})
  \]
  on the homotopy of $H$-algebras.
\end{example}

\begin{example}
  \label{ex:relativemassey}
  Relations involving operations other than composition and addition
  can also produce secondary operations, and the canonical examples of
  these are \emph{\index{Massey product}Massey products}. An $\mathcal{A}_2$-algebra $R$ has
  a binary multiplication operation $R \otimes R \to R$, and if $R$ is
  an $\mathcal{A}_3$-algebra it has a chosen associativity
  homotopy. As a result, if we have elements $x$, $y$, and $z$ in
  $\pi_* R$ such that $xy = yz = 0$, then we can glue together two
  nullhomotopies of $xyz$ to obtain a \emph{bracket}
  $\langle x, y, z\rangle$ that specializes to definitions of Massey
  products or Toda brackets.

  In trying to express nonlinear relations as secondary operations,
  however, we rapidly find that we want to move into a \emph{relative}
  situation.  A Massey product is defined on the kernel of the map
  $\pi_p \times \pi_q \times \pi_r \to \pi_{p+q} \times \pi_{q+r}$
  sending $(x,y,z)$ to $(xy,yz)$. However, the relation
  $x(yz) = (xy)z$ is not expressible solely as some operation on $xy$
  and $yz$: we need to remember $x$ and $z$ as well, but we \emph{do
    not} want to enforce that they are zero.

  We find that the needed expression is homotopy commutativity of the
  following diagram:
  \[
    \xymatrix{
      \Free_{\mathcal{A}_3}(S^p \oplus S^{p+q+r} \oplus S^r)
      \ar^-\Phi[r] \ar[d] &
      \Free_{\mathcal{A}_3}(S^p \oplus S^{p+q} \oplus S^{q+r} \oplus S^r)
      \ar[d]^\Theta \\
      \Free_{\mathcal{A}_3}(S^p \oplus S^r) \ar[r] &
      \Free_{\mathcal{A}_3}(S^p \oplus S^q \oplus S^r)
    }
  \]
  The right-hand map classifies the operation
  $\theta(x,y,z) = (x,xy,yz,z)$, and the top map classifies the
  operation $\varphi(x,u,v,z) = (x,xv - uz,z)$. The bottom-left object
  is not the initial object in the category of
  $\mathcal{A}_3$-algebras, so we \emph{enforce} this by switching to
  the category $\cC$ of $\mathcal{A}_3$-algebras under
  $\Free_{\mathcal{A}_3}(S^p \oplus S^r)$. In this category, we
  genuinely have augmented objects with a nullhomotopic double
  composite
  \[
    \Free_\cC(S^{p+q+r}) \to \Free_\cC(S^{p+q} \oplus S^{q+r})
    \to \Free_\cC(S^q)
  \]
  that defines a Massey product.
\end{example}

\subsection{Juggling}
\label{sec:juggling}

\index{juggling formulas|(}
Secondary operations are part of the homotopy theory of $\cC$, and
there is typically no method to determine secondary operations purely
in terms of the homotopy category. However, there are many
composition-theoretic tools that use one secondary operation to
determine information about another: typically, one starts with a
$4$-fold composite
\[
  X \too{f} Y \too{g} Z \too{h} U \too{k} V,
\]
with some assortment of double-composites being nullhomotopic, and
relates various associated secondary operations. This process is
called \emph{juggling}, and learning to
juggle secondary operations is one of the main steps in applying
them. For instance, one of the main juggling formulas---the
\index{Peterson--Stein formula}Peterson--Stein formula---asserts that the sets
$\langle k, h, g\rangle f$ and $k \langle h, g, f\rangle$ are inverse
in $\pi_1$ when both sides make sense.

% \begin{example}[cf. {\cite{adams-hopfinvariantone}}]
%   Relations in the Steenrod algebra give rise to secondary
%   operations. The relations $\Sq^1 \Sq^1 = 0$ and $\Sq^2
%   \Sq^2 + \Sq^3 \Sq^1 = 0$ gives rise to secondary
%   operations
%   \[
%     \Phi_{0,0} = \langle \Sq^1, \Sq^1, -\rangle,
%     \Phi_{1,1} = \left\langle
%       \begin{bmatrix}\Sq^2 & \Sq^3\\0 & \Sq^1\end{bmatrix},
%       \begin{bmatrix}\Sq^2 \\ \Sq^1\end{bmatrix},
%       -\right\rangle.
%   \]
%   The relations $\Sq^3 \Sq^2 = 0$ and $\Sq^3 \Sq^3 + \Sq^5 \Sq^1 = 0$
%   further give rise to a juggling formula
%   \[
%     \Sq^3 \Phi_{1,1}(\alpha) + \Sq^5 \Phi_{0,0}(\alpha) = 
%     \left\langle
%       \begin{bmatrix}\Sq^3 & \Sq^5\end{bmatrix},
%       \begin{bmatrix}\Sq^2 & \Sq^3\\0 & \Sq^1\end{bmatrix},
%       \begin{bmatrix}\Sq^2 \\ \Sq^1\end{bmatrix} \right\rangle(\alpha)
%   \]
%   In other words, $\Sq^3 \Phi_{1,1}$ and $\Sq^5 \Phi_{0,0}$ differ by
%   a primary operation. Up to indeterminacy, this becomes
%   \[
%     \Sq^3 \Phi_{1,1}(\alpha) + \Sq^5 \Phi_{0,0}(\alpha) = \lambda
%     \Sq^8(\alpha)
%   \]
%   for some coefficient $\lambda$. This coefficient turns out to be
%   zero---otherwise there could not exist a map $S^{15} \to S^8$ of
%   Hopf invariant one.
% \end{example}

\begin{example}
  The \index{Adem relations}Adem relations $Q^{2n+1} Q^n$ and $Q^{4n+3} Q^{2n+1}$ give rise
  to a secondary operation $\langle \ms Q^n, \ms Q^{2n+1}, \ms Q^{4n+3}\rangle$,
  an element of $\pi_{7n + 5 + m} (H \otimes \Free_{E_\infty}(S^m))$
  representing an operation that increases degree by $7n+5$. The
  juggling formula says that, for any element $\alpha \in \pi_m(A)$
  with $Q^n(A) = 0$, we have
  \[
    Q^{4n+3} \langle \alpha, \ms Q^n, \ms Q^{2n+1} \rangle =
    \langle \ms Q^n, \ms Q^{2n+1}, \ms Q^{4n+3}\rangle (\alpha).
  \]
  In other words, this secondary composite of operations gives a
  universal formula for how to apply $Q^{4n+3}$ to this secondary
  operation.
\end{example}

\index{juggling formulas|)}
\subsection{Application to the Brown--Peterson spectrum}
\index{Brown--Peterson spectrum|(}

In this section we will give a brief account of the main result of
\cite{secondary}, which uses secondary operatons to show that the
2-primary Brown--Peterson spectrum $BP$ does not admit the structure
of an $E_{12}$ \index{ring spectrum}ring spectrum. These results have been generalized by
\index{Senger}Senger to show that, at the prime $p$, $BP$ does not have an
$E_{2p^2+4}$ ring structure  \cite{senger-bp}.

As in \S\ref{sec:connectivity}, if the Brown--Peterson spectrum has an
$E_n$-algebra structure then the map
\[
  BP \to H\mb Z_{(2)} \to H
\]
can be given the structure of a map of $E_n$-algebras. On homology,
this would then induce a monomorphism
\[
  \mb F_2[\xi_1^2, \xi_2^2, \dots] \to \mb F_2[\xi_1, \xi_2, \dots]
\]
of algebras equipped with $E_n$ Dyer--Lashof operations and secondary
Dyer--Lashof operations. The \index{dual Steenrod algebra}dual Steenrod algebra, on the right, has
operations that are completely forced. Therefore, if we can
calculate enough to show that the subalgebra $H_* BP$ is not closed
under secondary operations for $E_n$-algebras, we arrive at an
obstruction to giving $BP$ the structure of $E_n$ ring
spectrum.

The calculation of secondary operations in $H_* H$ is accomplished
with judicious use of juggling formulas, ultimately reducing questions
about secondary operations to questions about primary ones.
\begin{itemize}
\item There is a pushout diagram of $E_\infty$ ring spectra
  \[
    \xymatrix{
      H \otimes MU \ar[r]^i \ar[d] & H \otimes H \ar[d]^j \\
      H \ar[r] & H \otimes_{MU} H.  }
  \]
  This makes $H \otimes MU$ into an augmented $H$-algebra, and gives a
  nullhomotopy of the composite
  $H \otimes MU \to H \otimes H \to H \otimes_{MU} H$. The elements 
  $\alpha$ in $H_* MU$ that map to zero in $H_* H$ are then candidates
  for secondary operations: we can construct
  $\langle j, i, \alpha \rangle$ in the \index{dual Steenrod
    algebra!for $MU$}$MU$-dual Steenrod algebra
  $\pi_* (H \otimes_{MU} H)$.
\item These elements are concretely detected: they have explicit
  representatives on the 1-line of a two-sided \index{spectral sequence!bar}bar spectral sequence
  \[
    \Tor^{H_* MU} (H_* H, \mb F_2) \Rightarrow \pi_* (H \otimes_{MU}
    H).
  \]
\item If we can determine \emph{primary} operations
  $\theta(\langle i, j, \alpha\rangle)$ in the $MU$-dual Steenrod
  algebra, the \index{juggling formulas}juggling formulas of \S\ref{sec:juggling} tell us about
  \index{functional operations}functional operations $\langle j, \alpha, \Theta\rangle$ in the
  ordinary dual Steenrod algebra.
\item \index{Steinberger}Steinberger's calculations of primary operations $\varphi$ in the
  dual Steenrod algebra then allow us to determine the values of
  $\varphi\langle j, \alpha, \Theta\rangle$, and juggling
  formulas again allow us to determine information about secondary
  operations $\langle \alpha, \Theta, \Phi\rangle$ in the dual
  Steenrod algebra.
\end{itemize}
This method, then, reduces us to carrying out some key
computations.

We must determine primary operations in the $MU$-dual Steenrod
algebra. Some of these, by work of \index{Tilson}Tilson \cite{tilson-kunneth}, are
determined by \index{Kochman}Kochman's calculations from
Theorem~\ref{thm:bordismoperations}: the \index{spectral sequence!K\"unneth} K\"unneth spectral sequence
\[
  \Tor^{\pi_* MU}(\mb F_2, \mb F_2) \Rightarrow \pi_*(H \otimes_{MU} H)
\]
calculating the $MU$-dual Steenrod algebra is compatible with
Dyer--Lashof operations. However, there are remaining extension
problems in the $\Tor$, and these turn out to be precisely what we are
interested in when juggling.

The $MU$-dual Steenrod algebra is an exterior algebra, whose
generators correspond to the \index{indecomposables}indecomposables in $\pi_* MU$. The
\index{extension problem}extension problems in the $\Tor$ spectral sequence arise because some
generators in $\pi_* MU$ have nontrivial image in $H_* MU$ and are
detected by $\Tor_1$, while others have trivial image in $H_* MU$ and
are detected by $\Tor_0$. The solution is to find an algebra $R$
mapping to $MU$ that does not have this problem. If we can find one so
that the map $\pi_* R \to \pi_* MU$ is surjective, the map from the
$R$-dual Steenrod algebra to the $MU$-dual Steenrod algebra is
surjective. If the generators of $\pi_* R$ have nontrivial image in
$H_* R$, then the spectral sequence
\[
  \Tor^{H_* R}(H_* H, \mb F_2) \Rightarrow \pi_* (H \otimes_R H)
\]
detects all needed classes with $\Tor_1$ and hence eliminates the
extension problem.

For this purpose, we used the \index{spherical group algebra}spherical group algebra
$\mb S[SL_1(MU)]$. The Dyer--Lashof operations in
$H_* SL_1(MU)$ are derived from the \index{Dyer--Lashof
operations!multiplicative}multiplicative Dyer--Lashof
operations $\widetilde Q^n$ in $\Omega^\infty MU$. This is a lengthy
calculation of power operations in the \index{Hopf ring}Hopf ring, and it is ultimately
determined by calculations of \index{Johnson}Johnson--\index{Noel}Noel of power operations in the
formal group theory of $\index{$MU$}MU$ \cite{noel-johnson-ptypical}.

Finally, we must determine a candidate secondary operation in $H_* H$
to which we can apply this procedure---there are many candidate
operations and many dead ends. The secondary operation is rather
large: it was found using a calculation in Goerss--Hopkins
\index{obstruction theory}obstruction theory that is detailed at length in \cite{bpobstructions}.

\index{secondary operations|)}
\index{Brown--Peterson spectrum|)}
\section{Coherent structures}
\label{sec:coherence}

\index{coherent multiplication|(}

In \S\ref{sec:algebras} we discussed algebras over an operad in a
general topological category, or more generally algebras over a
\index{multicategory}multicategory $\cM$, including \index{extended
  power construction}extended power and \index{free algebra}free algebra
functors. The definitions we used made heavy use of a strict
symmetric monoidal structure on the category of spectra.

In this section we will discuss the coherent viewpoint on these
constructions that makes use of the machinery of \index{Lurie}Lurie
\cite{lurie-higheralgebra}, and with the goal of connecting
different strata in the literature. To begin, we should point some of
the problems that this discussion is meant to solve.

We would like to demonstrate that our constructions are
model-independent. There are several different symmetric monoidal
categories of spectra \cite{ekmm, hovey-shipley-smith-symmetric,
  schwede-gammaspace} with several different model structures, and
there is a nontrivial amount of work involved in showing that an
equivalence between two different categories of spectra gives an
equivalence between categories of algebras
\cite{shipley-schwede-algebrasmodules}. These issues are compounded
when we attempt to relate notions of commutative algebras in different
categories, even if they have equivalent homotopy theory
\cite{white-commutativemonoids}.
  
We would also like to allow weaker structure than a symmetric monoidal
structure. For example, given a fixed $E_n$-algebra $R$ we will use
this to discuss the classification of power operations on
$E_n$-algebras under $R$. Our natural home for this discussion will
be the category of \index{module category} \index{$E_n$-module} $E_n$
$R$-modules (as in Example~\ref{ex:relativemassey}).

\subsection{Structured categories}
\label{sec:structuredcategories}

As discussed in \S\ref{sec:introduction}, classical symmetric
monoidal categories are analogues of commutative monoids with the
difference that they require natural isomorphisms to express
associativity, commutativity, and the like. We can express this
structure using simplicial \index{operad}operads. For any categories $\cC$ and
$\cD$, there is a groupoid $\Fun(\cC, \cD)^\simeq$ of functors and
natural isomorphisms. Taking the nerve, we get a simplicially enriched
category $\Cat$, and it makes sense to ask whether $\cC$ has the
structure of an algebra over a simplicial operad $\oO$.

\begin{example}
  A symmetric monoidal category can be expressed as an algebra over
  the \index{Barratt--Eccles operad}Barratt--Eccles operad \cite{barratt-eccles-operad}.
\end{example}

\begin{example}
  In classical category theory, a \emph{\index{braided monoidal category}braided monoidal
    category} in the sense of \cite{joyal-street-braided} can be
  encoded by a sequence of maps 
  \[
    NP_n \to \Fun(\cC^n, \cC)
  \]
  from the nerves of the pure braid groups to the categories of
  functors $\cC^n \to \cC$. The required compatibilities between these
  maps can be concisely expressed by noting that these nerves assemble
  into an $E_2$-operad, and that a braided monoidal category is an
  algebra over this operad.
\end{example}

We would like to discuss $E_n$-analogues of these structures in the
context of categories with morphism spaces. We will give some
definitions in this section, on the point-set level, with the purpose
of interpolating the older and newer definitions. We would like to say
that an \index{monoidal category!over an operad}$\oO$-monoidal category is an algebra over the operad $\oO$ in
$\Cat$, but this requires us to be clever enough to have a
well-behaved definition of a \emph{space} of functors between two
enriched categories; the failure of enriched categories to have a
well-behaved enriched functor category is a principal motivation for
the use of quasicategories.

\index{multicategory|(}
Until further notice, all categories and multicategories are
assumed to be enriched in spaces and all functors are functors of
enriched categories.

\begin{definition}
  Suppose that $p\co \cC \to \cM$ is a multifunctor, and write
  $\cC_{\ms x}$ for the category $p^{-1}(\ms x)$. Given objects $X_i
  \in \cC_{\ms x_i}$ and a map
  \[
    \alpha\co A \to \Mul_{\cM}(\ms x_1, \dots, \ms x_d; \ms y)
  \]
  of spaces, an \emph{$\alpha$-\index{twisted product}twisted product}
  is an object $Y \in \cC_{\ms y}$ and a map
  $A \to \Mul_{\cC}(X_1,\dots,X_d;Y)$ such that, for any $Z \in \cC$
  with $p(Z)= \ms z$, the diagram
  \[
    \xymatrix{
      \Map_{\cC}(Y,Z) \ar[r] \ar[d] &
      \Map(A,\Mul_{\cC}(X_1,\dots,X_d;Z)) \ar[d] \\
      \Map_\cM(\ms y,\ms z) \ar[r] &
      \Map(A,\Mul_\cM(\ms x_1,\dots,\ms x_d; \ms z))
    }
  \]
  is a pullback. If it exists, we denote it by $A \ltimes_\alpha
  (X_1,\dots,X_d)$.
\end{definition}

\begin{definition}
  \label{def:omonoidal-weak}
  An \emph{\index{weakly $\cM$-monoidal category}weakly $\cM$-monoidal
    category} is a multifunctor $p\co \cC \to \cM$ that has
  $\alpha$-twisted product for any inclusion
  \[
    \alpha\co \{f\} \subset \Mul_{\cM}(\ms x_1,\dots,\ms x_d;\ms y).
  \]
  A \emph{\index{strongly $\cM$-monoidal category}strongly
    $\cM$-monoidal category} is a category that has $\alpha$-twisted
  products for all $\alpha$.
\end{definition}

\begin{remark}
  In particular, for any point
  $f \in \Mul_\cM (\ms x_1,\dots,\ms x_d;\ms y)$, this universal
  property can be used to produce a functor
  \[
    \{f\} \ltimes (-)\co \cC_{\ms x_1} \times \dots \times \cC_{\ms
      x_d} \to \cC_{\ms y},
  \]
  and these are compatible with composition (up to natural
  isomorphism). A weakly $\cM$-monoidal category determines, up to
  natural equivalence, a multifunctor $\cM \to \Cat$.
\end{remark}

\begin{example}
  Every multicategory $\cC$ has a multifunctor to the one-object
  multicategory $\index{$\Comm$}\Comm$ associated to the commutative operad. The
  multicategory is $\Comm$-monoidal if and only if multimorphisms
  $(X_1,\dots,X_d) \to Y$ are always representable by an object
  $X_1 \otimes \dots \otimes X_d$, which is precisely when $\cC$ comes
  from a symmetric monoidal category. It is strongly $\Comm$-monoidal
  only if it is also tensored over spaces in a way compatible with the
  monoidal structure as in Definition~\ref{def:enriched-colimits}.
\end{example}

\begin{example}
  Associated to a monoidal category $\cC$ we can build a
  multicategory: multimorphisms $(X_1,\dots,X_d) \to Y$ are pairs of a
  permutation $\sigma \in \Sigma_d$ and a map
  $f\co X_{\sigma(1)} \otimes \dots \otimes X_{\sigma(d)} \to Y$. There is
  a multifunctor from this category to the multicategory $\index{$\Assoc$}\Assoc$
  corresponding to the associative operad: it sends all objects to the
  unique object, and sends each multimorphism $(\sigma,f)$ as above to
  the permutation $\sigma$. Conversely, an $\Assoc$-monoidal category
  comes from a monoidal category.
\end{example}

\begin{example}
  Suppose that $A$ is a commutative ring and $B$ is an
  $A$-algebra. Then there is a multicategory $\cC$ as follows.
  \begin{enumerate}
  \item An object of $\cC$ is either an $A$-module or a right
    $B$-module.
  \item The set $\Mul_{\cC}(M_1,\dots,M_d;N)$ of multimorphisms is
    \[
      \begin{cases}
        \Hom_A(M_1 \otimes_A \dots \otimes_A M_d,N)
        &\text{if $N$ and all $M_i$ are $A$-modules,}\\
        \Hom_B(M_1 \otimes_A \dots \otimes_A M_d,N)
        &\text{if $N$ and exactly one $M_i$ are $B$-modules,}\\
        \emptyset &\text{otherwise.}
      \end{cases}
    \]
  \end{enumerate}
  This comes equipped with a functor from $\cC$ to the multicategory
  $\index{$\Mod$}\Mod$ from Example~\ref{ex:ringmodule} that parametrizes
  ring-module pairs: any $A$-module is sent to $\ms a$ and any
  $B$-module is sent to $\ms m$. This makes $\cC$ into a
  $\Mod$-monoidal category, expressing the fact that $\Mod_A$ has a
  tensor product and that objects of $\RMod_B$ can be tensored with
  objects of $\Mod_A$. This makes $\RMod_B$ \emph{\index{left-tensored}left-tensored} over
  $\Mod_A$.
\end{example}

\begin{example}
  \label{ex:fiber-spaces}
  \index{fiberwise homotopy theory}Fiberwise homotopy theory studies
  the category $\Spaces_{/B}$ of spaces over $B$. Let $\oO$ be an
  operad and $B$ be a space with the structure of an
  $\oO$-algebra. Then $\Spaces_{/B}$ has the structure of a strongly
  $\oO$-monoidal category in the following way. For spaces
  $X_1,\dots, X_d$ and $Y$ over $B$, the space of multimorphisms is
  the pullback
  \[
    \xymatrix{
      \Mul_{/B}(X_1,\dots,X_d;Y) \ar[rr] \ar[d] &&
      \Map(X_1\times \dots \times X_d,Y) \ar[d] \\
      \oO(d) \ar[r] &
      \Map(B^d, B) \ar[r] &
      \Map(X_1 \times \dots \times X_d, B).
    }
  \]
  That is, a multimorphism consists of a point $f \in \oO(d)$ and a
  commutative diagram
  \[
    \xymatrix{
      X_1 \times \dots \times X_d \ar[r] \ar[d] & Y \ar[d]\\
      B^d \ar[r]_f & B.
    }
  \]
  With this definition, it is straightforward to verify that for
  $\alpha\co A \to \oO(d)$, the $\alpha$-twisted product
  $A \ltimes_\alpha (X_1,\dots,X_d)$ is the following space over $B$:
  \[
    A \times X_1 \times \dots \times X_d \to \oO(d) \times B^d \to B.
  \]
  In general, this should not be expected to be part of a symmetric
  monoidal structure on the category of spaces over $B$, even up to
  equivalence.
\end{example}

\begin{example}
  \label{ex:indexed-spectra}
  Let $\mathcal{L}$ be the category of
  \emph{\index{universe}universes}: an object is a countably infinite
  dimensional inner product space $U$. These objects have an
  associated multicategory: the space
  $\Mul_{\mathcal{L}}(U_1,\dots,U_d;V)$ of multimorphisms is the
  (contractible) space of \index{linear isometries operad}linear isometric embeddings
  $U_1 \oplus \dots \oplus U_d \into V$. Over $\mathcal{L}$, there is
  a category $\Sp_{\mathcal{L}}$ of \emph{\index{indexed spectra}indexed spectra}. An object
  is a pair $(U,X)$ of a universe $U$ and a spectrum $X$ (in the
  \index{Lewis}Lewis--\index{May}May--\index{Steinberger}Steinberger sense \cite{lewis-may-steinberger}) indexed
  on $U$; a multimorphism $((U_1,X_1),\dots,(U_d,X_d)) \to (V,Y)$ is a
  pair of a linear isometric embedding
  $i\co U_1 \oplus \dots \oplus U_d \to V$ and a map
  $i_*(X_1 \wedge \dots \wedge X_d) \to Y$ of spectra indexed on $V$.

  This does not describe the topology on the multimorphisms in this
  category. Given a map $A \to \mathcal{L}(U_1,\dots,U_d;V)$ and
  spectra $X_i$ indexed on $U_i$, there is a \emph{\index{twisted
      half-smash product}twisted half-smash product}
  $A \ltimes (X_1,\dots,X_d)$ indexed on $V$ \cite[\S
  VI]{lewis-may-steinberger}, equivalent to the \index{smash product}smash product
  $A_+ \wedge X_1 \wedge \dots \wedge X_d$. There exists a topology on
  the multimorphisms so that a continuous map in from $A$ is
  equivalent to a map $A \to \mathcal{L}(U_1,\dots,U_d;V)$ and a map
  $A \ltimes (X_1,\dots,X_d) \to Y$. By design, then, the projection
  $\Sp_\mathcal{L} \to \mathcal{L}$ makes the category of indexed
  spectra strongly $\mathcal{L}$-monoidal.
\end{example}

\begin{example}
  Fix an $E_n$-algebra $A$ in $\Sp$, and consider the category of
  $E_n$-algebras $R$ with a factorization $A \to R \to A$ of the
  identity map. This has an associated \emph{\index{stable category}stable category}, serving
  as the natural target for \index{Goodwillie calculus}Goodwillie's
  calculus of functors: the category of \index{$E_n$-module}$E_n$ $A$-modules
  \cite{francis-tangentcomplex}. This category should also not be
  expected to have a symmetric monoidal structure, but the tensor
  product over $R$ does give it the structure of an $E_n$-monoidal
  category. For example, for an associative algebra $A$ in $\Sp$, the tensor
  product over $A$ gives the category of
  $A$-\index{bimodules}bimodules a monoidal  structure.
\end{example}

\subsection{Multi-object algebras}
\label{sec:multiobjectalgebras}
\index{algebra!over a multicategory}
\index{multicategory|(}

Just as we cannot make sense of a commutative monoid in a nonsymmetric
monoidal category, we need relationships between an operad $\oO$ and
any multiplicative structure on a category $\cC$ before $\oO$ can
act on objects.

\begin{definition}
  \label{def:section-algebras}
  Suppose $p\co \cC \to \cN$ and $\cM \to \cN$ are multifunctors. An
  \emph{$\cM$-algebra in $\cC$} is a lift in the diagram
  \[
    \xymatrix{
      &\cC \ar[d]^p \\
      \cM \ar@{.>}[ur] \ar[r] & \cN
    }
  \]
  of multifunctors. We write $\index{$\Alg$}\Alg_{\cM/\cN}(\cC)$ for this category of
  $\cM$-algebras.
\end{definition}

\begin{example}
  If $\cC$ and $\cM$ are arbitrary multicategories, then using the
  unique maps from $\cC$ and $\cM$ to the terminal multicategory
  $\index{$\Comm$}\Comm$ we recover the definition of $\index{$\Alg$}\Alg_\cM(\cC)$, the category
  of $\cM$-algebras in $\cC$ from Definition~\ref{def:multialgebras}.
\end{example}

\begin{example}
  Let the space $B$ be an algebra over an operad $\oO$ and consider the \index{fiberwise homotopy theory}fiberwise category
  $\Spaces_{/B}$ of spaces over $B$ with the strongly $\oO$-monoidal
  structure from Example~\ref{ex:fiber-spaces}. An $\oO$-algebra in
  $\Spaces_{/B}$ is an $\oO$-algebra $X$ with a map of $\oO$-algebras
  $X \to B$.
\end{example}

\begin{example}
  Consider the category of \index{indexed spectra}indexed spectra
  $\Sp_{\mathcal{L}}$ from Example~\ref{ex:indexed-spectra}. The fact
  that the \index{external smash product}external smash product $(X_1 \wedge \dots \wedge X_n)$
  is naturally indexed on the direct sum of the associated \index{universe}universes
  obstructed making the category of spectra indexed on any individual
  universe $\Sp$ strictly symmetric monoidal, and so we cannot ask
  about commutative monoids in $\Sp_{\mathcal{L}}$---but the structure
  available is still enough to do multiplicative homotopy theory. An 
  $\mathcal{L}$-algebra in $\Sp_{\mathcal{L}}$ recovers the classical
  definition of an $E_\infty$ \index{ring spectrum}ring spectrum from
  \cite{may-quinn-ray-ringspectra}. Similarly we can define
  $\oO$-algebras for any operad $\oO$ with an augmentation to
  $\mathcal{L}$ \cite[VII.2.1]{lewis-may-steinberger}.
\end{example}

\begin{proposition}
  Suppose that $\cC$ is strongly $\cN$-monoidal and that $\cM \to \cN$ is a
  map of multicategories. In addition, suppose that $\cC$ has enriched
  colimits and that formation of $\alpha$-twisted products
  preserves enriched colimits in each variable.
  \begin{enumerate}
  \item For objects $\ms x$ and $\ms y$ of $\cM$, there are \index{extended power construction}extended power
    functors
    \[
      \Sym^k_{\cM, \ms x \to \ms y}\co \cC_{\ms x} \to \cC_{\ms y},
    \]
    given by
    \[
      \Sym^k_{\cM, \ms x \to \ms y}(X) = \Mul_\cM(\underbrace{\ms x, \ms x, \dots, \ms x}_{k};
      \ms y) \ltimes (X,X,\dots,X) / \Sigma_k.
    \]
  \item The \index{evaluation functor}evaluation functor $\ev_{\ms x}\co \Alg_\cM(\cC) \to
    \cC_{\ms x}$ has a left adjoint
    \index{free algebra}
    \[
      \Free_{\cM,\ms x}\co \cC_{\ms x} \to \Alg_{\cM}(\cC).
    \]
    The value of $\Free_{\cM,\ms x}(X)$ on any object $\ms y$ of $\cM$ is
    \[
      \ev_{\ms y} (\Free_{\cM,\ms x}(X)) = \coprod_{k \geq 0}
      \Sym^k_{\cM, \ms x \to \ms y}(X).
    \]
  \end{enumerate}
\end{proposition}

\begin{example}
  Let $B$ be a space with an action of an operad $\oO$, and let
  $X$ a space over $B$. Then the extended powers are
  \[
    \Sym^k_{\oO}(X) = \left(\oO(k) \times_{\Sigma_k} X^k \to
      \oO(k) \times_{\Sigma_k} B^k \to B\right).
  \]
\end{example}

\begin{example}
  \index{graded algebra}
  Suppose that $\Gamma$ is a commutative monoid and that $X$ is a
  $\Gamma$-graded $E_\infty$ ring spectrum, as in
  Example~\ref{ex:gradedalgebras}. Then there are action maps
  $\Sym^k X_g \to X_{kg}$. These give rise to Dyer--Lashof
  operations $Q^i\co H_* X_g \to H_{*+i} (X_{2g})$.
\end{example}

\begin{example}
  Suppose that $\dots \to X_2 \to X_1 \to X_0$ is a \index{strongly
    filtered}strongly filtered $E_\infty$ ring spectrum, as in
  Example~\ref{ex:filteredalgebras}. Then there are action maps
  $\Sym^k X_n \to X_{kn}$ that are compatible. These give rise to
  power operations $Q^i\co H_* X_n \to H_{*+i} X_{2n}$ that are
  compatible as $n$ varies, and there are induced \index{power operations}power operations on
  the associated spectral sequence.
\end{example}

\begin{example}
  Given a spectrum $X$ indexed on a \index{universe}universe $U$ as in
  Example~\ref{ex:indexed-spectra}, the extended powers are modeled by
  \index{twisted half-smash product}twisted half-smash products:
  \[
    \Sym^k_{U \to U}(X) \simeq E\Sigma_k \ltimes_{\Sigma_k} (X^{\wedge k})
  \]
  This recovers the machinery that was put to effective use in the
  1970s and 1980s for studying $E_\infty$ \index{ring spectrum}ring spectra and
  $H_\infty$-ring spectra, before the development of strictly monoidal
  categories of spectra.
\end{example}

\index{multicategory|)}
\subsection{$\infty$-operads}

\index{$\infty$-operad|(}

The point-set discussion of the previous sections provides a library
of examples. As the basis for a theory it relies on the existence of
rigid models and preservation of colimits.

\begin{example}
  Consider the \index{fiberwise homotopy theory}fiberwise category of
  spaces over a fixed base space $B$. This category has a symmetric
  monoidal \index{fiber product}fiber product $X \times_B Y$. The fiber product typically
  needs fibrant input to represent the homotopy fiber product; the
  fiber product typically does not produce cofibrant output. This
  makes it difficult to use the standard machinery to study algebras
  and modules in this category. These problems have received
  significant attention in the setting of \index{parametrized homotopy theory}
 parametrized stable homotopy
  theory
  \cite{may-sigurdsson-parametrized, lind-bundles,
    malkiewich-lind-parametrized}.
\end{example}

\begin{example}
  The category of nonnegatively graded \index{chain complexes}chain complexes over a
  commutative ring $R$ is equivalent to the category of simplicial
  $R$-modules via the \index{Dold--Kan correspondence}Dold--Kan correspondence. This correspondence is
  lax symmetric monoidal in one direction, but only lax monoidal in
  the other. Moreover, while both sides have morphism spaces, the
  \index{Dold}Dold--\index{Kan}Kan correspondence only preserves these up to weak
  equivalence, even for fibrant-cofibrant objects.
\end{example}

\begin{example}
  In the standard models of \index{equivariant stable homotopy theory}equivariant stable homotopy theory the
  notion of strict $G$-commutativity is equivalent to one encoded by
  \index{operad!equivariant}equivariant operads rather than ordinary ones
  \cite{mandell-may-orthogonal,hhr-kervaire,blumberg-hill-multiplications}. This means that
  an $E_\infty$-algebra $A$ (in the sense of an ordinary $E_\infty$
  operad) may not have a strictly commutative model
  \cite{mcclure-tate,hill-hopkins-multiplicative}, and this makes it
  more difficult to construct a symmetric monoidal model for the
  category of $A$-modules.
\end{example}

% \begin{example}
%   Categories of sheaves typically have the problem that
%   local-to-global problems are resolved using injective sheaves, while
%   problems about deriving a tensor product typically rely on
%   projectives. This means that doing structured algebra with sheaves
%   is more difficult, even though projective and injective resolutions
%   present the same homotopy theory.
% \end{example}

The framework of $\infty$-operads
\cite{lurie-higheralgebra} (or, alternatively, that of
\index{dendroidal sets}dendroidal sets
\cite{moerdijk-weiss-dendroidal}) is one method to express coherent
multiplicative structures. Here are some of the salient points.
\begin{itemize}
\item This generalization takes place in the theory of
  \index{$\infty$-category}$\infty$-categories (specifically \index{quasicategory}quasicategories), equivalent to
  the study of categories enriched in spaces. Every category
  enriched in spaces gives rise to an $\infty$-category; every
  $\infty$-category has morphism spaces between its objects.
\item In this framework, for $\infty$-categories $\cC$ and $\cD$ there
  is a space $\index{$\Fun$}\Fun(\cC, \cD)$ encoding the structure of functors and
  natural equivalences.
\item In an $\infty$-category, \index{homotopy (co)limit}homotopy limits and colimits are
  intrinsic notions rather than arising from a particular
  construction. Many common constructions produce presentable
  $\infty$-categories, which have all homotopy limits and colimits.
\item \index{multicategory}Multicategories generalize to so-called
  \emph{$\infty$-operads}. These have an underlying $\infty$-category,
  and there are spaces of multimorphisms to an object from a
  tuple of objects. Every \index{multicategory!topological}topological multicategory gives rise to an
  $\infty$-operad; every $\infty$-operad can be realized by a
  topological multicategory. The precise definitions are similar in
  spirit to \index{Segal}Segal's encoding of $E_\infty$-spaces
  \cite{segal-categoriescohomology}.
\item An $\infty$-operad $\oO$ has an associated notion of an
  $\oO$-monoidal $\infty$-category. An $\oO$-monoidal
  $\infty$-category is expressed in terms of maps $\cC \to \oO$ of
  $\infty$-operads with properties analogous to that 
  \index{weakly $\cM$-monoidal category} from
  Definition~\ref{def:omonoidal-weak}, with the main difference that
  spaces of morphisms are respected.  An $\oO$-monoidal
  $\infty$-category is also equivalent to a functor from $\oO$ to a
  category of categories: each object $\ms x$ of $\oO$ has an
  associated category $\cC_{\ms x}$, and one can associate a map
  \[
    \Mul_{\oO}(\ms x_1,\dots,\ms x_d; \ms y) \to \Fun(\cC_{\ms
      x_1},\dots,\cC_{\ms x_d};\cC_{\ms y})
  \]
  of spaces.
\item We can discuss algebras and modules in terms of sections, just
  as in Definition~\ref{def:section-algebras}.
\end{itemize}

All of this structure is systematically invariant under
equivalence. Equivalent $\infty$-operads give rise to equivalent notions of an
$\oO$-algebra structure on $\cC$; $\infty$-categories equivalent to
$\cC$ have equivalent notions of $\oO$-algebra structures to those on
$\cC$; equivalent $\oO$-monoidal $\infty$-categories have equivalent
categories of $\cM$-algebras for any map $\cM \to \oO$ of
$\infty$-operads.

\begin{example}
  An $E_n$-\index{operad}operad has an associated $\infty$-operad $\oO$, and as a
  result we can define an $E_n$-monoidal $\infty$-category $\cC$ to be
  an $\oO$-monoidal $\infty$-category. When $n=1$, $2$, or $\infty$
  we can recover monoidal, \index{braided monoidal category}braided monoidal, and symmetric monoidal
  structures.
\end{example}

\index{$\infty$-operad|)}
\subsection{Modules}

\index{module category!of an $E_n$-algebra|(}

\index{Mandell}Mandell's theorem (\ref{thm:mandell-e4}), which is about structure on
the homotopy category of left \index{module category!of an $E_n$-algebra}modules over an \index{$E_n$-algebra}$E_n$-algebra, is a
reflection of higher structure on the category of left modules
itself.

\begin{theorem}[{\cite[5.1.2.6, 5.1.2.8]{lurie-higheralgebra}}]
  Suppose that $\cC$ is an $E_k$-monoidal $\infty$-category which
  has \index{geometric realization}geometric realization of simplicial objects, and such that the
  tensor product preserves such geometric realizations in each
  variable separately. Then the category of left modules over an
  $E_{k}$-algebra $A$ is $E_{k-1}$-monoidal, and has all colimits that
  exist in $\cC$.
\end{theorem}

As previously discussed, the category of left modules over an
associative algebra $R$ is not made monoidal under the tensor product
over $R$, but the category of \index{bimodules}bimodules is. The
generalization of this result to $E_n$-algebras is the following.

\begin{theorem}[{\cite[3.4.4.2]{lurie-higheralgebra}}]
  Suppose $\cC$ is an $E_n$-monoidal presentable $\infty$-category such that the
  monoidal structure preserves homotopy colimits in each variable
  separately. Then for any $E_n$-algebra $R$ in $\cC$, there is a
  category $\Mod^{E_n}_R(\cC)$ of $E_n$ $R$-modules. This is a
  presentable $E_n$-monoidal $\infty$-category whose underlying
  monoidal operation is the tensor product over $R$.

  In particular, if $\cC$ is a presentable $\infty$-category with a
  symmetric monoidal structure that preserves colimits in each
  variable, and $R$ is an $E_n$-algebra in $\cC$, the category of
  $E_n$ $R$-modules in $\cC$ has an $E_n$-monoidal structure that
  preserves colimits in each variable.
\end{theorem}

% \begin{remark}
%   This theorem also holds in much greater generality: when there is a
%   coherent $\infty$-operad \cite[3.3.1.9]{lurie-higheralgebra} and
%   a presentable $\oO$-monoidal $\infty$-category
%   \cite[3.4.4.1]{lurie-higheralgebra}.
% \end{remark}

Roughly, an \index{module category!of an $E_n$-algebra}$E_n$ $R$-module $M$ has multiplication operations
$R^{\otimes k} \otimes M \to M$ parametrized by $(k+1)$-tuples of
points of \index{configuration space}configuration space, where one point is marked by $M$ and
the rest by $R$. This has the more precise description of
$E_n$-modules as left modules.

\begin{theorem}[{\cite[5.5.4.16]{lurie-higheralgebra}, \cite{francis-tangentcomplex}}]
  Suppose that $\cC$ is a symmetric monoidal $\infty$-category and
  that the monoidal product preserves colimits in each variable
  separately. For an $E_n$-algebra $R$ in $\cC$, the
  \index{factorization homology}factorization homology
  $\int_{D^n \setminus 0} R$ has the structure of an $E_1$-algebra,
  and the category of $E_n$ $R$-modules is equivalent to the category
  of left modules over $\int_{D^n \setminus 0} R$.
\end{theorem}

\begin{remark}
  In the category of spectra, this could be regarded as a consequence
  of the \index{Schwede--Shipley theorem}Schwede--Shipley theorem \cite{schwede-shipley-stablemodules}
  or its generalizations. There is a \index{adjunction!free-forgetful}free-forgetful adjunction between
  $E_n$ $R$-modules and $\Sp$, and the image $\Free_{E_n\mbox{-}R}(\mb S)$ of
  the sphere spectrum under the left adjoint is a compact generator
  for the category of $E_n$ $R$-modules. Therefore, $E_n$ $R$-modules
  are equivalent to the category of modules over the \index{endomorphism ring}endomorphism ring
  \[
    F_{E_n\mbox{-}R}(\Free_{E_n\mbox{-}R}(\mb S),
    \Free_{E_n\mbox{-}R}(\mb S)) \simeq \Free_{E_n\mbox{-}R}(\mb S).
  \]
  This theorem, then, is an identification of the free
  $E_n$ $R$-module.
\end{remark}

\begin{example}
  When $n=1$, the category of $E_1$ $R$-modules is the category of
  left modules over $R\otimes R^\op$. When $n=2$, the category of
  $E_2$-$R$-modules is the category of left modules over the
  \index{topological Hochschild homology|see{$\THH$}}topological Hochschild homology $\index{$\THH$}\THH(R)$.
\end{example}

\index{module category!of an $E_n$-algebra|)}

\subsection{Coherent powers}
\label{sec:powers}

In the classical case, we described an $\oO$-algebra structure
on $A$ in terms of \index{action map}action maps
\[
  \Sym^k_\oO(A) = \oO(k) \otimes_{\Sigma_k} A^{\otimes k} \to A
\]
from \index{extended power construction}extended power constructions
to $A$, and gave a formula
\index{free algebra}
\[
  \Free_\oO(X) = \coprod_{k \geq 0} \Sym^k_{\oO}(A)
\]
for the free $\oO$-algebra on an object in the case where the monoidal
structure is compatible with \index{enriched colimit}enriched colimits; we also discussed the
multi-object analogue in \S \ref{sec:multiobjectalgebras}. The
analogous constructions for $\infty$-operads are carried out in
\cite[\S 3.1.3]{lurie-higheralgebra}, and we will sketch these results
here.

Fix an $\infty$-operad $\oO$. For any objects $\ms x_1,\dots, \ms x_d,
\ms y$ of $\oO$, we can construct a space
\[
  \Mul_{\oO}(\ms x_1,\dots, \ms x_d;\ms y)
\]
of multimorphisms in $\oO$; if the $\ms x_i$ are equal, this further
can be given a natural action of the symmetric group.

Let $\cC$ be an $\oO$-monoidal $\infty$-category $\cC$. In particular,
$\cC$ encodes categories $\cC_{\ms x}$ parametrized by the objects
$\ms x$ of $\oO$, and functors
$f\co \cC_{\ms x_1} \times \dots \times \cC_{\ms x_d} \to \cC_{\ms y}$
parametrized by the multimorphisms
$f\co (\ms x_1,\dots,\ms x_d) \to \ms y$ of $\oO$. Suppose that the
categories $\cC_{\ms x}$ have homotopy colimits and the functors
preserve homotopy colimits in each variable. Then there exist
\emph{\index{extended power construction}extended power functors}
\[
  \Sym^k_{\oO,\ms x \to \ms y}\co \cC_{\ms x} \to\cC_{\ms y},
\]
whose value on $X \in \cC_{\ms u}$ is a homotopy colimit
\[
  \left(\hocolim_{\alpha \in \Mul_{\oO}(\ms x, \dots, \ms x; \ms y)}
    \alpha(X \oplus \dots \oplus X)\right)_{h\Sigma_k}.
\]
These extended powers have the property that an $\oO$-algebra $A$ has
natural maps $\Sym^k_{\oO,\ms x \to \ms y} (A(\ms x)) \to A(\ms
y)$. Moreover, there is a \index{adjunction!free-forgetful}free-forgetful adjunction between
$\oO$-algebras and $\cC_{\ms x}$, and the free object\index{free algebra}
$\Free_{\oO,\ms x}(X)$ on $X \in \cC_U$ has the property that its
value on $\ms y$ is exhibited as the coproduct
\[
  \ev_{\ms y}(\Free_{\oO,\ms x}(X)) \simeq \coprod_{k \geq 0}
  \Sym^k_{\oO, \ms x \to \ms y}(X).
\]

\begin{remark}
  Composing with the diagonal $\cC_{\ms x} \to \prod \cC_{\ms x}$
  gives a $\Sigma_k$-equivariant map
  \[
    \Mul_{\oO}(\underbrace{\ms x,\dots, \ms x}_k;\ms y) \to
    \Fun(\cC_{\ms x}\times\dots\times\cC_{\ms x}, \cC_{\ms y}) \to
    \Fun(\cC_{\ms x}, \cC_{\ms y})
  \]
  that factors through the homotopy orbit space
  \[
    P(k) = \Mul_{\oO}(\ms x,\dots, \ms x;\ms
    y)_{h\Sigma_k}.
  \]
  This space $P(k)$ then serves as a parameter space for tensor-power
  functors $\cC_{\ms x} \to \cC_{\ms y}$.
  
  In the case of an ordinary single-object $\infty$-operad $\oO$ such
  as an $E_n$-operad, we can rephrase in terms of $P(k)$. Such an
  $\infty$-operad $\oO$ is equivalent to an ordinary operad in
  spaces and an $\oO$-monoidal $\infty$-category is equivalent
  to an $\infty$-category $\cC$ with a map
  $\oO \to \End(\cC)$. We recover a formula
  \[
    \Free_{\oO}(X) \simeq \coprod_{k \geq 0} \hocolim_{\alpha \in
      P(k)} \alpha(X,\dots,X)
  \]
  for the free algebra on $X$. When $X = S^m$, this is the \index{Thom spectrum}Thom
  spectrum
  \[
    \coprod_{k \geq 0} P(k)^{m \rho},
  \]
  closely related to Remark~\ref{rmk:thom}.

  When $\oO$ is an $E_n$-operad, the space $P(k)$ is equivalent to the
  space $\mathcal{C}_n(k) / \Sigma_k$, a model for the space of
  unordered \index{configuration space}configurations of $k$ points in $\mb R^n$. When
  $n = \infty$ the space $P(k)$ is a model for $B\Sigma_k$, and we
  find that the we recover the ordinary homotopy symmetric power:
  \[
    \Sym^k_{E_\infty}(X) \simeq (X^{\otimes k})_{h\Sigma_k}.
  \]
\end{remark}

\begin{example}
  Fix a space $B$ and consider the \index{fiberwise homotopy theory}fiberwise category $\Spaces_{/B}$. The
  homotopy \index{fiber product}fiber product $X \times^h_B Y$ gives this the structure of
  a symmetric monoidal $\infty$-category, breaking up independently
  over the components of $B$. If $B$ is path-connected, then the
  extended power and free functors on $(X \to B)$ are those obtained by
  applying the extended power and free functors to the fiber.
\end{example}

\begin{example}
  \index{$E_n$-module}
  Given an $E_n$ $R$-module $M$, the free $E_n$ $R$-algebra on an
  $E_n$ $R$-module $M$ is
  \[
    \coprod_{k \geq 0} \hocolim_{\alpha \in C_n(k)/\Sigma_k}
    M^{\otimes_\alpha k},
  \]
  where each point $\alpha$ of \index{configuration space}configuration space determines a
  functor $M^{\otimes_\alpha k} \simeq M \otimes_R \dots \otimes_R
  M$.

  More can be said under the identification between $E_n$-modules and
  modules over \index{factorization homology}factorization homology. If $M$ is the free
  $E_n$ $R$-module on $S^m$, then we obtain an identification of the
  free $E_n$-algebra under $R$ on $S^m$:
  \[
    R \amalg^{E_n} \Free_{E_n}(S^m) \simeq \coprod_{k \geq 0}
    \left(\int_{\mb R^k
      \setminus \{p_1,\dots,p_k\}} R\right) \otimes_{\Sigma_k} S^{m \rho_k}.
  \]
\end{example}

\begin{remark}
  The interaction between connective objects and their
  \index{Postnikov tower!algebras}Postnikov truncations from \S
  \ref{sec:connectivity} generalizes to the case where we have an
  $\oO$-monoidal $\infty$-category $\cC$ with a \emph{compatible
    \index{$t$-structure}$t$-structure} in the sense of
  \cite[2.2.1.3]{lurie-higheralgebra}. This means that the categories
  $\cC_{\ms x}$ indexed by the objects $\ms x$ of $\oO$ all have
  $t$-structures, and the functors induced by the morphisms in $\oO$
  are all additive with respect to connectivity. Then
  {\cite[2.2.1.8]{lurie-higheralgebra}} implies that connective
  $\oO$-algebras have Postnikov towers: the collection of truncation
  functors $\tau_{\leq n}$ is compatible with the $\oO$-monoidal
  structure on $\cC_{\geq 0}$.
\end{remark}
\index{coherent multiplication|)}

\section{Further invariants}
\label{sec:further}

\subsection{Units and Picard spaces}

\begin{definition}
  For an $E_n$-monoidal $\infty$-category $\cC$ with unit $\mb I$, the
  \emph{\index{Picard space}Picard space} $\index{$\Pic$}\Pic(\cC)$ is the full subgroupoid of $\cC$
  spanned by the \emph{\index{invertible object}invertible objects}: objects $X$ for which
  there exists an object $Y$ such that
  $Y \otimes X \simeq X \otimes Y \simeq \mb I$.
\end{definition}

\begin{remark}
  The classical \index{Picard group}Picard group of the homotopy category $h\cC$ is the
  set $\pi_0 \Pic(\cC)$ of path components.
\end{remark}

In particular, $\Pic(\cC)$ is closed under the $E_n$-monoidal
structure on $\cC$, giving it a canonical $E_n$-space
structure. Moreover, by construction
$\pi_0 \Pic(\cC) = (\pi_0 \cC)^\times$ is a group, and so $\Pic(\cC)$
is an \index{iterated loop space}$n$-fold loop space. The loop space $\Omega \Pic(\cC)$ is the
space of homotopy self-equivalences of the unit $\mb I$; in the case
of the category $\LMod_R$ of left modules, it is homotopy equivalent
to the unit group $GL_1(R)$ of \index{unit group!of a ring spectrum}
$R$.
\index{loop space|seealso{iterated loop space}}

\begin{proposition}[{\cite[\S 7]{ando-blumberg-gepner-twisted}}]
  If $R$ is an $E_n$ ring spectrum, then the space $GL_1(R)$ of
  homotopy self-equivalences of the left module $R$ has an $n$-fold
  delooping. If $n \geq 2$, the space $\Pic(R) = \Pic(\LMod_R)$ has an
  $(n-1)$-fold delooping. 
\end{proposition}

\subsection{Topological Andr\'e--Quillen cohomology}
\label{sec:TAQ}
\index{$\TAQ$|(}

\index{topological Andr\'e-Quillen homology|see{$\TAQ$}}Topological
Andr\'e-Quillen homology and cohomology are invariants of ring spectra
developed by \index{Kriz}Kriz and \index{Basterra}Basterra \cite{kriz-towers,
  basterra-andrequillen}. For a fixed map of $E_\infty$ \index{ring spectrum}ring spectra
$A \to B$, we can define a topological Andr\'e--Quillen homology
object $\TAQ(A \to R \to B)$ for any object $R$ in the category of
$E_\infty$ rings between $A$ and $B$. This is characterized by the
following properties \cite{basterra-mandell-taqcohomology}:
\begin{enumerate}
\item It naturally takes values in the category of $B$-modules.
\item It takes homotopy colimits of $E_\infty$ ring spectra between
  $A$ and $B$ to homotopy colimits of $B$-modules.
\item There is a natural map $B \otimes_A (R/A) \to \TAQ(A \to R \to
  B)$.
\item For a left $A$-module $X$ with a map $X \to B$, the composite
  natural map
  \[
    B \otimes_A X \to B \otimes_A \Free^A_{E_\infty}(X) \to \TAQ(A \to
    \Free^A_{E_\infty}(X) \to B)
  \]
  of $B$-modules is an equivalence.
\item Under the above equivalence, the natural map
  \[
    \TAQ(A \to \Free^A_{E_\infty} \Free^A_{E_\infty} (X) \to B) \to
    \TAQ(A \to \Free^A_{E_\infty}(X) \to B)
  \]
  is equivalent to the map
  \[
    B \otimes_A \Free^A_{E_\infty}(X) \to B \otimes_A X
  \]
  that collapses $B \otimes_A (\amalg \Sym^k(X))$ to the factor with $k=1$.
\end{enumerate}

Topological Andr\'e-Quillen homology measures how difficult it is to
build $R$ as an $A$-algebra: any description of $R$ as an iterated
pushout along maps of free of $E_\infty$-algebras, starting from $A$,
determines a description of the topological Andr\'e--Quillen
cohomology of $R$ as an iterated pushout of $B$-modules. \index{Basterra}Basterra
showed that $\TAQ$-cohomology groups
\[
  \TAQ^n(R;M) = [\TAQ(\mb S \to R \to R), \Sigma^n M]_{\Mod_R}
\]
plays the role for \index{Postnikov tower}Postnikov towers of $E_\infty$ ring spectra that
ordinary cohomology does for spectra.

From this point of view, $\TAQ$ also has natural generalizations to
$\TAQ^\oO$ for algebras over an arbitrary operad
\cite{basterra-mandell-taqcohomology,harper-quillen}, although there
may be a choice of target category that takes more work to
describe. In particular, for $E_n$-algebras these are related to an
iterated bar construction \cite{basterra-mandell-iteratedthh}.

Topological Andr\'e--Quillen homology also enjoys the following
properties, proved in \cite{basterra-andrequillen,
  basterra-mandell-taqcohomology}.
\begin{quote}
  \begin{description}
  \item[\index{base-change}Base-change:] For a map $B \to C$, we have
    a natural equivalence
    \[
      C \otimes_B \TAQ(A \to R \to B) \simeq \TAQ(A \to R \to C).
    \]
    In particular, if we define $\Omega_{R/A} = \TAQ(A \to R \to R)$,
    then
    \[
      \TAQ(A \to R \to B) = B \otimes_R \Omega_{R/A}.
    \]
  \item[\index{transitivity}Transitivity:] For a composite
    $A \to R \to S \to B$, there is a natural cofiber sequence
    \[
      \TAQ(A \to R \to B) \to \TAQ(A \to S \to B) \to \TAQ(R \to S \to
      B).
    \]
    In particular, for $A \to R \to S$ we have cofiber sequences
    \[
      S \otimes_R \Omega_{R/A} \to \Omega_{S/A} \to \Omega_{S/R}.
    \]
  \item[\index{representability!$\TAQ$-cohomology}Representability:] Suppose
    that there is a functor $h^*$ from the category of pairs
    $(R \to S)$ of $E_\infty$ ring spectra between $A$ and $B$ to the
    category of graded abelian groups. Suppose that this is a
    \index{cohomology theory!for ring spectra}cohomology theory on the category of $E_\infty$ ring spectra
    between $A$ and $B$: it satisfies homotopy invariance, has a long
    exact sequence, satisfies excision for homotopy pushouts of pairs,
    and takes coproducts to products. Then there is a $B$-module $M$
    with a natural isomorphism
    \begin{align*}
      h^n(S,R) &\cong \TAQ^n(S,R;M) \\
      & = [\TAQ(R \to S \to B), \Sigma^n M]_{\Mod_B}
    \end{align*}
    of abelian groups.
  \end{description}
\end{quote}

For any $E_\infty$ ring spectrum $B$, algebras mapping to $B$ have
$\TAQ$-homology $\TAQ(\mb S \to R \to B)$, valued in the category of
$B$-modules. The \index{square-zero algebra}square-zero algebras
\[
  B \oplus M
\]
are representing objects for $\TAQ$-cohomology $\TAQ^*(R;M)$.

Representability allows us to construct and classify operations in
\index{cohomology operations!$\TAQ$}$TAQ$-cohomology by $B$-algebra maps between such square-zero
extensions.

\begin{proposition}
  Any element in $[\Sigma \Sym^2 M, N]_{\Mod_B}$ has a naturally
  associated map $B \oplus M \to B \oplus N$ of augmented commutative
  $B$-algebras and hence gives rise to a natural $\TAQ$-cohomology
  operation $\TAQ(-;M) \to \TAQ(-;N)$ for
  commutative algebras mapping to $B$.
\end{proposition}

\begin{proof}
  By viewing $B$ as concentrated in grading 0 and $M$ as concentrated
  in grading $1$, we can give a \index{graded algebra}$\mb Z$-graded construction (as in
  Example~\ref{ex:gradedalgebras}) of $B \oplus M$ as an iterated
  sequence of pushouts along maps of \index{free algebra}free algebras. The first such
  pushout is
  \[
    \Free^B_{E_\infty}(M) \leftarrow \Free^B_{E_\infty}(\Sym^2 M)
    \rightarrow B
  \]
  Further pushouts only alter gradings $3$ and higher.

  We now view $B \oplus N$ as graded by putting $N$ in grading $2$. We
  find that homotopy classes of maps of graded algebras $B \oplus M \to B \oplus N$ are
  equivalent to maps $\Sigma \Sym^2 M \to N$.
\end{proof}

\begin{example}
  Letting $M = B \otimes S^m$, we have
  \[
    \Sigma \Sym^2(M) \simeq B \otimes \Sigma^{m+1} \mb{RP}_m^\infty.
  \]
  Therefore, we get a map from the $B$-cohomology
  $B^n(\Sigma^{m+1}\mb{RP}_m^\infty)$ of \index{stunted projective space}stunted projective spaces to
  the group of natural cohomology operations $TAQ^m(-;B) \to
  TAQ^n(-;B)$.
\end{example}

\begin{remark}
  The fact that elements in the $B$-homology of stunted projective
  spaces produce \index{homotopy operations}homotopy operations while elements in their
  $B$-cohomology produce \index{cohomology operations!$\TAQ$}$\TAQ$-cohomology operations with a shift is
  a reflection of \index{Koszul duality}Koszul duality.
\end{remark}

\begin{example}
  Letting $M = (B \otimes S^q) \oplus (B \otimes S^r)$, and using the
  projection
  \begin{align*}
    \Sigma \Sym^2(B \otimes (S^q \oplus S^r))
    &\simeq \Sigma\Sym^2 (B \otimes S^q) \oplus \Sigma\Sym^2 (B
      \otimes S^r) \oplus \Sigma (B \otimes S^q \otimes S^r) \\
    &\to B \otimes S^{q+1+r},
  \end{align*}
  we get a binary operation
  \[
    [-,-]\co \TAQ^q(-;B) \times \TAQ^r(-;B) \to \TAQ^{q+1+r}(-;B)
  \]
  that (up to a normalization factor) we call the
  \emph{\index{bracket!$\TAQ$}$\TAQ$-bracket}.
\end{example}

\begin{example}
  If $B = H\mb F_2$, then there are
  $\TAQ$-cohomology operations
  \[
    R^a\co \TAQ^m(-;H\mb F_2) \to \TAQ^{m+a}(-;H\mb F_2)
  \]
  for $a \geq m+1$, and a bracket
  \[
    \TAQ^q(-;H\mb F_2) \times \TAQ^r(-;H\mb F_2) \to
    \TAQ^{q+1+r}(-;H\mb F_2).
  \]
  In this form, the operation $R^{a+1}$ is \index{Koszul dual}Koszul dual
  to $Q^{a}$, in the sense that nontrivial values of $R^{a+1}$ in
  $\TAQ$-cohomology detect relations on the operator $Q^a$ in
  homology. Similarly, the bracket in $\TAQ$ is Koszul dual to the
  multiplication.
\end{example}

The operations were constructed by to \index{Basterra}Basterra--\index{Mandell}Mandell
\cite{basterra-mandell-BP}. In further unpublished work, they showed
that these operations (and their odd-primary analogues) generate all
the natural operations on $\TAQ$-cohomology with values in $H\mb F_p$
and determined the relations between them. In particular, the
operations $R^a$ above satisfy the same \index{Adem relations}Adem relations that the
\index{Steenrod operations}Steenrod operations $Sq^a$ do; the \index{bracket!$\TAQ$}$\TAQ$-bracket has the structure of a
\index{restricted Lie algebra}shifted restricted Lie bracket, whose restriction is the bottommost
defined operation $R^a$.

\index{Basterra}Basterra--\index{Mandell}Mandell's proof uses a variant of the
\index{spectral sequence!Miller}Miller spectral sequence from
\cite{miller-delooping}. We will close out this section with a sketch
of how such spectral sequences are constructed, parallel to the
delooping spectral sequence from Remark~\ref{rmk:barsseq}.
\begin{proposition}
  Suppose that $R$ is an $E_\infty$ ring spectrum with a
  chosen map $R \to H\mb F_p$. Then there is a \index{spectral sequence!Miller}Miller spectral sequence
  \[
    AQ^{DL}_* (\pi_* (H\mb F_p \otimes R)) \Rightarrow \TAQ_*(\mb S \to R \to
    H\mb F_p),
  \]
  where the left-hand side are the \index{derived functors!nonabelian}
  nonabelian derived functors of an
  \index{indecomposables}indecomposables functor $Q$ that sends an
  augmented graded-commutative $\mb F_p$-algebra with Dyer--Lashof
  operations to the quotient of the augmentation ideal by all products
  and Dyer--Lashof operations.
\end{proposition}

\begin{proof}
  We construct an augmented simplicial object:
  \[
    \cdots \Free_{E_\infty} \Free_{E_\infty} \Free_{E_\infty} R
    \Rrightarrow \Free_{E_\infty} \Free_{E_\infty} R \Rightarrow
    \Free_{E_\infty} R
    \to R.
  \]
  If $U$ is the forgetful functor, from commutative ring spectra
  mapping to $H\mb F_p$ to spectra mapping to $H\mb F_p$, this is the
  \index{bar construction}bar construction $B(\Free_{E_\infty}, U\Free_{E_\infty}, UR)$. The
  underlying simplicial spectrum
  $B(U\Free_{E_\infty}, U\Free_{E_\infty}, UR)$ has an extra
  degeneracy, so its \index{geometric realization is}geometric realization is equivalent to
  $R$. Moreover, the forgetful functor from $E_\infty$ rings to
  spectra preserves sifted homotopy colimits, and hence geometric
  realization because the simplicial indexing category is
  sifted. Therefore, applying the homotopy colimit preserving functor
  $\TAQ = \TAQ(\mb S \to (-) \to H\mb F_p)$ and the natural
  equivalence
  $\TAQ \circ \Free_{E_\infty}(R) \simeq H\mb F_p \otimes R$, we get
  an equivalence
  \[
    |B(H\mb F_p \otimes (-), U\Free_{E_\infty}, UR)| \simeq \TAQ(R).
  \]
  However, this bar construction is a simplicial object of the form
  \[
    \cdots H\mb F_p \otimes \Free_{E_\infty} \Free_{E_\infty} R \Rrightarrow
    H\mb F_p \otimes \Free_{E_\infty} R \Rightarrow H\mb F_p \otimes R.
  \]
  Taking homotopy groups, we get a simplicial object
  \[
    \mb Q_{E_\infty} \mb Q_{E_\infty} H_* R 
    \Rrightarrow \mb Q_{E_\infty} H_* R \Rightarrow H_* R.
  \]
  Moreover, the structure maps make this the bar construction
  \[
    B(Q, \mb Q_{E_\infty}, H_* R)
  \]
  that computes derived functors of $Q$ on graded-commutative algebras
  with Dyer--Lashof operations. Therefore, the spectral sequence
  associated to the geometric realization computes $\TAQ_*(\mb S \to R
  \to H\mb F_p)$ and has
  the desired $E_2$-term.
\end{proof}

\begin{remark}
  We can also apply cohomology rather than homology and get a
  spectral sequence computing topological Andr\'e--Quillen cohomology.
  
  This leaves open a hard algebraic part of \index{Basterra}Basterra--\index{Mandell}Mandell's
  work: actually calculating these derived functors, and in particular
  finding relations amongst the operations $R^a$ and the bracket
  $[-,-]$ that give a complete description of $\TAQ$-cohomology
  operations.
\end{remark}

\index{$\TAQ$|)}
\section{Further questions}

We will close this paper with some problems that we think are useful
directions for future investigation.

\begin{problem}
  Develop useful \index{obstruction theory}obstruction theories which can determine the
  existence of or maps between \index{$E_n$-algebra}$E_n$-algebras in a wide variety of
  contexts.
\end{problem}

The obstruction theory due to \index{Goerss}Goerss--\index{Hopkins}Hopkins
\cite{goerss-hopkins-summary} is the prototype for these results. In
unpublished work \cite{senger-obstr}, \index{Senger}Senger has given a development
of this theory for $E_\infty$-algebras where the obstructions occur in
nonabelian $\Ext$-groups calculated in the category of
graded-commutative rings with Dyer--Lashof operations and Steenrod
operations satisfying the Nishida relations, and provided tools for calculating with them. This played
a critical role in \cite{secondary,bpobstructions}.

In closely related situations, the tools available remain
rudimentary. For example, there is essentially no workable obstruction
theory for the construction of commutative rings of any type in
\index{equivariant stable homotopy theory}equivariant stable theory. Tools arising from the Steenrod algebra
have been essential in most of the deep results in homotopy theory,
such as the \index{Segal conjecture}Segal conjecture \cite{lin-davis-mahowald-adams} and the
\index{Sullivan conjecture}Sullivan conjecture \cite{miller-sullivanconjecture}. Without the
analogues, there is a limit to how much structure can be revealed.

\begin{problem}
  Give a modern redevelopment of \index{homology operations}homology operations for $E_\infty$
  \index{ring space}ring spaces and $E_n$ ring spaces.
\end{problem}

The observant reader may have noticed that, despite the rich structure
present, the principal material that we have referenced for $E_\infty$
ring spaces is several decades old. Several major advances have
happened in multiplicative stable homotopy theory since then, and the
author feels that there is still a great deal to be mined. Having this
material accessible to modern toolkits would be extremely useful.

For one example, the theory of $E_\infty$ ring spaces from the point
of view of symmetric spectra has been studied in detail by \index{Sagave}Sagave and
\index{Schlichtkrull}Schlichtkrull \cite{schlichtkrull-units,
  sagave-schlichtkrull-gradedunits,
  sagave-schlichtkrull-groupcompletion}. For another, the previous
emphasis on $E_\infty$ ring spaces should be tempered by the variety
of examples that we now know only admit $A_\infty$ or $E_n$ ring
structures.

\begin{problem}
  Give a unified theory of graded Hopf algebras and
  \index{Hopf ring}Hopf rings, capable of encoding some combination of
  non-integer gradings, \index{power operations}power operations, \index{group-completion}group-completion theorems, and the interaction with the unit.
\end{problem}

\index{Ravenel}Ravenel--\index{Wilson}Wilson's theory of Hopf rings is integer-graded. We now know
many examples---\index{motivic homotopy theory}motivic homotopy
theory, equivariant homotopy theory, \index{$K(n)$-local}$K(n)$-local
theory, modules over \index{module category!of an $E_n$-algebra}$E_n$ ring spectra---that may have natural gradings of
a much wider variety than this, such as a \index{Picard group}Picard
group. Moreover, multiplicative theory should involve much more
structure: we should have a sequence of spaces graded not just by a
Picard group, but by the \index{Picard space}Picard space that also encodes structure
nontrivial higher interaction between gradings and the unit group.

\begin{problem}
  Give a precise general description of the \index{Koszul duality}
  Koszul duality relationship between homotopy operations
  and $\TAQ$-cohomology
  \index{cohomology operations!$\TAQ$}operations.  Give a
  complete construction of the algebra of operations on
  $\TAQ$-cohomology for $E_n$-algebras with coefficients in $Hk$, for
  $k$ a commutative ring. Give complete descriptions of the
  $\TAQ$-cohomology for a large library of \index{Eilenberg--Mac Lane spectrum}Eilenberg--Mac Lane spectra
  $Hk$ and \index{Morava}Morava's \index{forms of $K$-theory}forms of $K$-theory.
\end{problem}

Because $\TAQ$-cohomology governs the construction of ring spectra via
their \index{Postnikov tower}Postnikov tower, essentially any information that we can provide
about these objects is extremely useful.

\begin{problem}
  Determine an algebro-geometric expression for power operations and
  their relationship to the \index{Steenrod operations}Steenrod operations. Do the same for the
  operations which appear in the Hopf ring associated to an $E_\infty$
  ring space.
\end{problem}

At the prime $2$, it has been known for some time that the action of
the Steenrod algebra can be concisely packaged as a coaction of the
\index{dual Steenrod algebra}dual Steenrod algebra, a Hopf algebra corresponding to the group
scheme of automorphisms of the additive \index{formal group}formal group over $\mb
F_2$. The Dyer--Lashof operations on infinite loop spaces generate an
algebra analogous to the Steenrod algebra, and its dual was described
by \index{Madsen}Madsen \cite{madsen-dyerlashof}; the result is closely related to
\index{Dickson invariants}Dickson invariants. However, the full action
of the Dyer--Lashof operations or the interaction between the
Dyer--Lashof algebra and the Steenrod algebra does not yet have a
geometric packaging.

\begin{conjecture}
  For \index{Lubin--Tate theory}Lubin--Tate cohomology theories $E$
  and $F$ of height $n$, there is a natural algebraic structure
  parametrizing operations from continuous $E$-homology to continuous
  $F$-homology for certain $E_\infty$ ring spectra, expressed in terms
  of the algebraic geometry of \emph{\index{isogenies}isogenies} of \index{formal group}formal groups.

  This is complete: there is an \index{obstruction theory}obstruction theory for the
  construction of and mapping between \index{$K(n)$-local}$K(n)$-local
  $E_\infty$ \index{ring spectrum}ring
  spectra whose algebraic input is completed $E$-homology equipped
  with these operations.
\end{conjecture}

In this paper we have not really touched on the extensive study of
power operations in chromatic homotopy theory
(cf. \cite{stapleton-character,barthel-beaudry-chromatic}). Given
Lubin--Tate cohomology theories $E$ and $F$ associated to formal
groups of height $n$ at the prime $p$, we have both \index{cohomology operations}cohomology operations and power operations. In \cite{hovey-operations} the
algebra of cohomology operations is expressed in terms of isomorphisms
of formal groups. Extensive work of \index{Ando}Ando, \index{Strickland}Strickland, and \index{Rezk}Rezk has
shown that \index{power operations}power operations are expressed in terms of quotient
operations for subgroups of the formal groups. It has been known for
multiple decades \cite[\S 28]{strickland-fpfp} that the natural home
combining these two types of operations is the theory of
\emph{isogenies} of formal groups. However, there are important
details about formal topologies which have never been
resolved.\footnote{The reader should be advised that, even at height
  1, there are difficult issues with $E$-theory here involving
  left-derived functors of completion.}

\begin{problem}
  Determine the natural \index{instability relations}instability
  relations for operations in \emph{unstable} \index{elliptic cohomology}elliptic cohomology and
  in unstable Lubin--Tate cohomology.
\end{problem}

\index{Strickland}Strickland states that isogenies are a natural interpretation for
unstable cohomology operations in $E$-theory. However, isogenies
encode the analogue of the cohomological \index{Steenrod operations}Steenrod operations, the
multiplicative Dyer--Lashof operations, and the \index{Nishida relations}Nishida relations
between them. They do \emph{not} encode any analogue of the
\index{instability relations}instability relation $Sq^n = Q^{-n}$ that we see in the cohomology of
spaces.

In \index{chromatic homotopy theory}chromatic theory, our only accessible example so far is
$K$-theory. For $p$-completed \index{$K$-theory}$K$-theory, the cohomology operations
are generated by the \index{Adams operations}Adams operations $\psi^k$ for
$k \in \mb Z_p^\times$. For torsion-free algebras, the \index{power
operations}power
operations are controlled by the operation $\psi^p$ and its
congruences \cite{hopkins-k1-local, rezk-wilkerson}. The
\index{unstable operations}\emph{unstable} operations in the $K$-theory of spaces, by contrast,
arise from the algebra of \index{symmetric polynomials}symmetric polynomials and are essentially
governed by the $\psi^n$ for $n \in \mb N$; the fact that the other
$\psi^k$ are determined by these enforces some form of
continuity. This is also closely tied to the question of whether there
are geometric interpretations of some type for elliptic cohomology
theories or Lubin--Tate cohomology theories.\footnote{One possible viewpoint is that we could interpret $\mb N$ as the
monoid of endomorphisms of the \emph{multiplicative monoid} $M_1$,
which contains the \index{unit group!of a ring spectrum}unit group
$GL_1$.}
\index{$GL_1$|see{unit group of a ring spectrum}}

\begin{problem}
  Determine a useful way to encode \index{secondary operations}secondary operation structures on
  $E_\infty$ or $E_n$ rings.
\end{problem}

In the case of secondary Steenrod operations, there is a useful
formulation due to \index{Baues}Baues of an extension of the Steenrod algebra that
can be used to encode all of the secondary operation structure
\cite{baues-secondary, nassau-secondary}. No such systematic
descriptions are known for secondary Dyer--Lashof operations,
especially since the Dyer--Lashof operations are expressed in a
more complicated way than the action of an algebra on a module.

\begin{problem}
  Determine useful relationships between the homotopy types of an
  $E_n$ ring spectrum, the \index{unit group}unit group $GL_1(R)$ and the
  \index{Picard space}Picard space $\Pic(R)$, and the spaces
  $BGL_n(R)$.
\end{problem}

This is closely tied to \index{orientation theory}orientation theory, \index{algebraic $K$-theory}algebraic $K$-theory, and
the study of spaces involved in \index{surgery theory}surgery theory.

Investigations in these directions due to \index{Mathew}Mathew--\index{Stojanoska}Stojanoska revealed
that there is a nontrivial relationship between the $k$-invariants for
$R$ and the unit spectrum $gl_1(R)$ at the edge of the stable range at
the prime 2 \cite{mathew-stojanoska-pictmf}, and forthcoming work of
\index{Hess}Hess has shown that this relation can be recovered from the
\index{mixed Cartan formula}mixed Cartan formula. The odd-primary analogues of this are not yet
known.

\begin{problem}
  Find an odd-primary formula for the \index{mixed Adem relations}mixed Adem relations similar to
  the Kuhn--Tsuchiya formula.
\end{problem}

There is a description of the mixed \index{Adem relations!mixed}Adem
relations \cite{cohen-lada-may-homology}, valid at any prime, but it
is difficult to apply in concrete examples. The 2-primary formula
described in \S\ref{sec:ring-spaces} is much more direct; it was
originally stated by \index{Tsuchiya}Tsuchiya and proven by \index{Kuhn}Kuhn
\cite{kuhn-jameshopfhomology}. There is no known odd-primary analogue
of this formula.

%\index{Dyer--Lashof operations|)}
%%% Local Variables:
%%% mode: latex
%%% TeX-master: "dl"
%%% End:
\bibliographystyle{plain}
\bibliography{../masterbib}

\begin{thebibliography}{10}

\bibitem{adams-stablehomotopy}
J.~F. Adams.
\newblock {\em Stable homotopy and generalised homology}.
\newblock University of Chicago Press, Chicago, Ill., 1974.
\newblock Chicago Lectures in Mathematics.

\bibitem{ando-blumberg-gepner-twisted}
Matthew Ando, Andrew~J. Blumberg, and David Gepner.
\newblock Parametrized spectra, multiplicative {T}hom spectra and the twisted
  {U}mkehr map.
\newblock {\em Geom. Topol.}, 22(7):3761--3825, 2018.

\bibitem{angeltveit}
Vigleik Angeltveit.
\newblock Topological {H}ochschild homology and cohomology of {$A_\infty$} ring
  spectra.
\newblock {\em Geom. Topol.}, 12(2):987--1032, 2008.

\bibitem{badzioch-algebraictheories}
Bernard Badzioch.
\newblock Algebraic theories in homotopy theory.
\newblock {\em Ann. of Math. (2)}, 155(3):895--913, 2002.

\bibitem{baker-power-operations-coactions}
Andrew Baker.
\newblock Power operations and coactions in highly commutative homology
  theories.
\newblock {\em Publ. Res. Inst. Math. Sci.}, 51(2):237--272, 2015.

\bibitem{barratt-eccles-operad}
M.~G. Barratt and Peter~J. Eccles.
\newblock {$\Gamma ^{+}$}-structures. {I}. {A} free group functor for stable
  homotopy theory.
\newblock {\em Topology}, 13:25--45, 1974.

\bibitem{barthel-beaudry-chromatic}
Tobias {Barthel} and Agn{\`e}s {Beaudry}.
\newblock {Chromatic structures in stable homotopy theory}.
\newblock {\em arXiv e-prints}, page arXiv:1901.09004, Jan 2019.

\bibitem{basterra-andrequillen}
M.~Basterra.
\newblock Andr\'e-{Q}uillen cohomology of commutative {$S$}-algebras.
\newblock {\em J. Pure Appl. Algebra}, 144(2):111--143, 1999.

\bibitem{basterra-mandell-taqcohomology}
Maria Basterra and Michael~A. Mandell.
\newblock Homology and cohomology of {$E_\infty$} ring spectra.
\newblock {\em Math. Z.}, 249(4):903--944, 2005.

\bibitem{basterra-mandell-iteratedthh}
Maria Basterra and Michael~A. Mandell.
\newblock Homology of {$E_n$} ring spectra and iterated {$THH$}.
\newblock {\em Algebr. Geom. Topol.}, 11(2):939--981, 2011.

\bibitem{basterra-mandell-BP}
Maria Basterra and Michael~A. Mandell.
\newblock The multiplication on {BP}.
\newblock {\em J. Topol.}, 6(2):285--310, 2013.

\bibitem{baues-secondary}
Hans-Joachim Baues.
\newblock {\em The algebra of secondary cohomology operations}, volume 247 of
  {\em Progress in Mathematics}.
\newblock Birkh\"{a}user Verlag, Basel, 2006.

\bibitem{baues-muro-cupone}
Hans-Joachim Baues and Fernando Muro.
\newblock Toda brackets and cup-one squares for ring spectra.
\newblock {\em Comm. Algebra}, 37(1):56--82, 2009.

\bibitem{bergner-multisorted}
Julia~E. Bergner.
\newblock Rigidification of algebras over multi-sorted theories.
\newblock {\em Algebr. Geom. Topol.}, 6:1925--1955, 2006.

\bibitem{blumberg-cohen-schlichtkrull-thom}
Andrew Blumberg, Ralph Cohen, and Christian Schlichtkrull.
\newblock {$THH$} of {T}hom spectra and the free loop space.
\newblock Preprint.

\bibitem{blumberg-hill-multiplications}
Andrew~J. Blumberg and Michael~A. Hill.
\newblock Operadic multiplications in equivariant spectra, norms, and
  transfers.
\newblock {\em Adv. Math.}, 285:658--708, 2015.

\bibitem{boardman-vogt-recognition}
J.~M. Boardman and R.~M. Vogt.
\newblock {\em Homotopy invariant algebraic structures on topological spaces}.
\newblock Lecture Notes in Mathematics, Vol. 347. Springer-Verlag, Berlin-New
  York, 1973.

\bibitem{borger-wieland-plethystic}
James Borger and Ben Wieland.
\newblock Plethystic algebra.
\newblock {\em Adv. Math.}, 194(2):246--283, 2005.

\bibitem{bckqrs-unstableadams}
A.~K. Bousfield, E.~B. Curtis, D.~M. Kan, D.~G. Quillen, D.~L. Rector, and
  J.~W. Schlesinger.
\newblock The {${\rm mod}-p$} lower central series and the {A}dams spectral
  sequence.
\newblock {\em Topology}, 5:331--342, 1966.

\bibitem{bousfield-kan-unstableadams}
A.~K. Bousfield and D.~M. Kan.
\newblock The homotopy spectral sequence of a space with coefficients in a
  ring.
\newblock {\em Topology}, 11:79--106, 1972.

\bibitem{bmms-hinfty}
R.~R. Bruner, J.~P. May, J.~E. McClure, and M.~Steinberger.
\newblock {\em {$H_\infty $} ring spectra and their applications}, volume 1176
  of {\em Lecture Notes in Mathematics}.
\newblock Springer-Verlag, Berlin, 1986.

\bibitem{campbell-cohen-peterson-selich-selfmapsI}
H.~E.~A. Campbell, F.~P. Peterson, and P.~S. Selick.
\newblock Self-maps of loop spaces. {I}.
\newblock {\em Trans. Amer. Math. Soc.}, 293(1):1--39, 1986.

\bibitem{cohen-configurationspaces}
F.~R. Cohen.
\newblock On configuration spaces, their homology, and {L}ie algebras.
\newblock {\em J. Pure Appl. Algebra}, 100(1-3):19--42, 1995.

\bibitem{cohen-lada-may-homology}
Frederick~R. Cohen, Thomas~J. Lada, and J.~Peter May.
\newblock {\em The homology of iterated loop spaces}.
\newblock Lecture Notes in Mathematics, Vol. 533. Springer-Verlag, Berlin-New
  York, 1976.

\bibitem{cooke-homotopyactions}
George Cooke.
\newblock Replacing homotopy actions by topological actions.
\newblock {\em Trans. Amer. Math. Soc.}, 237:391--406, 1978.

\bibitem{dunn-additivity}
Gerald Dunn.
\newblock Tensor product of operads and iterated loop spaces.
\newblock {\em J. Pure Appl. Algebra}, 50(3):237--258, 1988.

\bibitem{dwyer-kan-liftingdiagrams}
W.~G. Dwyer and D.~M. Kan.
\newblock An obstruction theory for diagrams of simplicial sets.
\newblock {\em Nederl. Akad. Wetensch. Indag. Math.}, 46(2):139--146, 1984.

\bibitem{ekmm}
A.~D. Elmendorf, I.~Kriz, M.~A. Mandell, and J.~P. May.
\newblock {\em Rings, modules, and algebras in stable homotopy theory}.
\newblock American Mathematical Society, Providence, RI, 1997.
\newblock With an appendix by M. Cole.

\bibitem{francis-tangentcomplex}
John Francis.
\newblock The tangent complex and {H}ochschild cohomology of {$E_n$}-rings.
\newblock {\em Compos. Math.}, 149(3):430--480, 2013.

\bibitem{quillen-appendixq}
Eric~M. Friedlander and Barry Mazur.
\newblock Filtrations on the homology of algebraic varieties.
\newblock {\em Mem. Amer. Math. Soc.}, 110(529):x+110, 1994.
\newblock With an appendix by Daniel Quillen.

\bibitem{stablepoweroperations}
Saul Glasman and Tyler Lawson.
\newblock Stable power operations.
\newblock Preprint.

\bibitem{goerss-hopkins-summary}
Paul~G. Goerss and Michael~J. Hopkins.
\newblock Moduli spaces of commutative ring spectra.
\newblock In {\em Structured ring spectra}, volume 315 of {\em London Math.
  Soc. Lecture Note Ser.}, pages 151--200. Cambridge Univ. Press, Cambridge,
  2004.

\bibitem{greenlees-may-tate}
J.~P.~C. Greenlees and J.~P. May.
\newblock Generalized {T}ate cohomology.
\newblock {\em Mem. Amer. Math. Soc.}, 113(543):viii+178, 1995.

\bibitem{harper-quillen}
J.E. Harper.
\newblock Bar constructions and quillen homology of modules over operads.
\newblock ar{X}iv:0802.2311.

\bibitem{hill-hopkins-multiplicative}
M.~A. Hill and M.~J. Hopkins.
\newblock Equivariant multiplicative closure.
\newblock In {\em Algebraic topology: applications and new directions}, volume
  620 of {\em Contemp. Math.}, pages 183--199. Amer. Math. Soc., Providence,
  RI, 2014.

\bibitem{hhr-kervaire}
M.~A. Hill, M.~J. Hopkins, and D.~C. Ravenel.
\newblock On the nonexistence of elements of {K}ervaire invariant one.
\newblock {\em Ann. of Math. (2)}, 184(1):1--262, 2016.

\bibitem{hopkins-k1-local}
Michael~J. Hopkins.
\newblock {$K(1)$}-local {$E_\infty$}-ring spectra.
\newblock In {\em Topological modular forms}, volume 201 of {\em Math. Surveys
  Monogr.}, pages 287--302. Amer. Math. Soc., Providence, RI, 2014.

\bibitem{hovey-operations}
Mark Hovey.
\newblock Operations and co-operations in {M}orava {$E$}-theory.
\newblock {\em Homology Homotopy Appl.}, 6(1):201--236, 2004.

\bibitem{hovey-shipley-smith-symmetric}
Mark Hovey, Brooke Shipley, and Jeff Smith.
\newblock Symmetric spectra.
\newblock {\em J. Amer. Math. Soc.}, 13(1):149--208, 2000.

\bibitem{hu-kriz-may-cores}
P.~Hu, I.~Kriz, and J.~P. May.
\newblock Cores of spaces, spectra, and {$E_\infty$} ring spectra.
\newblock {\em Homology Homotopy Appl.}, 3(2):341--354, 2001.
\newblock Equivariant stable homotopy theory and related areas (Stanford, CA,
  2000).

\bibitem{noel-johnson-ptypical}
Niles Johnson and Justin Noel.
\newblock For complex orientations preserving power operations,
  {$p$}-typicality is atypical.
\newblock {\em Topology Appl.}, 157(14):2271--2288, 2010.

\bibitem{joyal-street-braided}
Andr\'{e} Joyal and Ross Street.
\newblock Braided tensor categories.
\newblock {\em Adv. Math.}, 102(1):20--78, 1993.

\bibitem{kelly-enriched}
G.~M. Kelly.
\newblock Basic concepts of enriched category theory.
\newblock {\em Repr. Theory Appl. Categ.}, (10):vi+137, 2005.
\newblock Reprint of the 1982 original [Cambridge Univ. Press, Cambridge;
  MR0651714].

\bibitem{kochman-dyerlashof}
Stanley~O. Kochman.
\newblock Homology of the classical groups over the {D}yer-{L}ashof algebra.
\newblock {\em Trans. Amer. Math. Soc.}, 185:83--136, 1973.

\bibitem{kriz-towers}
I.~Kriz.
\newblock Towers of {$E_\infty$}-ring spectra with an application to {$BP$}.
\newblock unpublished, 1995.

\bibitem{kuhn-jameshopfhomology}
Nicholas~J. Kuhn.
\newblock The homology of the {J}ames-{H}opf maps.
\newblock {\em Illinois J. Math.}, 27(2):315--333, 1983.

\bibitem{bpobstructions}
T.~{Lawson}.
\newblock {Calculating obstruction groups for E-infinity ring spectra}.
\newblock {\em ArXiv e-prints}, September 2017.

\bibitem{secondary}
Tyler Lawson.
\newblock Secondary power operations and the {B}rown--{P}eterson spectrum at
  the prime {$2$}.
\newblock {\em Ann. of Math. (2)}, 188(2):513--576, 2018.

\bibitem{lazarev-ainfty}
A.~Lazarev.
\newblock Homotopy theory of {$A_\infty$} ring spectra and applications to
  {$M{\rm U}$}-modules.
\newblock {\em $K$-Theory}, 24(3):243--281, 2001.

\bibitem{lewis-may-steinberger}
L.~G. Lewis, Jr., J.~P. May, M.~Steinberger, and J.~E. McClure.
\newblock {\em Equivariant stable homotopy theory}, volume 1213 of {\em Lecture
  Notes in Mathematics}.
\newblock Springer-Verlag, Berlin, 1986.
\newblock With contributions by J. E. McClure.

\bibitem{lin-davis-mahowald-adams}
W.~H. Lin, D.~M. Davis, M.~E. Mahowald, and J.~F. Adams.
\newblock Calculation of {L}in's {E}xt groups.
\newblock {\em Math. Proc. Cambridge Philos. Soc.}, 87(3):459--469, 1980.

\bibitem{lind-bundles}
John~A. Lind.
\newblock Bundles of spectra and algebraic {$K$}-theory.
\newblock {\em Pacific J. Math.}, 285(2):427--452, 2016.

\bibitem{malkiewich-lind-parametrized}
John~A. Lind and Cary Malkiewich.
\newblock The {M}orita equivalence between parametrized spectra and module
  spectra.
\newblock In {\em New directions in homotopy theory}, volume 707 of {\em
  Contemp. Math.}, pages 45--66. Amer. Math. Soc., Providence, RI, 2018.

\bibitem{lurie-higheralgebra}
Jacob Lurie.
\newblock Higher {A}lgebra.
\newblock Draft version available at:
  http://www.math.harvard.edu/\~{}lurie/papers/higheralgebra.pdf, 2017.

\bibitem{madsen-dyerlashof}
Ib~Madsen.
\newblock On the action of the {D}yer-{L}ashof algebra in {$H_{\ast}(G)$}.
\newblock {\em Pacific J. Math.}, 60(1):235--275, 1975.

\bibitem{mahowald-thomcomplexes}
Mark Mahowald.
\newblock Ring spectra which are {T}hom complexes.
\newblock {\em Duke Math. J.}, 46(3):549--559, 1979.

\bibitem{mandell-may-orthogonal}
M.~A. Mandell and J.~P. May.
\newblock Equivariant orthogonal spectra and {$S$}-modules.
\newblock {\em Mem. Amer. Math. Soc.}, 159(755):x+108, 2002.

\bibitem{mandell-einftypadic}
Michael~A. Mandell.
\newblock {$E_\infty$} algebras and {$p$}-adic homotopy theory.
\newblock {\em Topology}, 40(1):43--94, 2001.

\bibitem{mandell-derivedsmash}
Michael~A. Mandell.
\newblock The smash product for derived categories in stable homotopy theory.
\newblock {\em Adv. Math.}, 230(4-6):1531--1556, 2012.

\bibitem{mathew-stojanoska-pictmf}
Akhil Mathew and Vesna Stojanoska.
\newblock The {P}icard group of topological modular forms via descent theory.
\newblock {\em Geom. Topol.}, 20(6):3133--3217, 2016.

\bibitem{may-problems}
J.~P. May.
\newblock Problems in infinite loop space theory.
\newblock In {\em Conference on homotopy theory ({E}vanston, {I}ll., 1974)},
  volume~1 of {\em Notas Mat. Simpos.}, pages 111--125. Soc. Mat. Mexicana,
  M\'exico, 1975.

\bibitem{may-sigurdsson-parametrized}
J.~P. May and J.~Sigurdsson.
\newblock Parameterized homotopy theory.
\newblock Preprint, http://www.math.uiuc.edu/K-theory/0716/.

\bibitem{may-homologyoperations}
J.~Peter May.
\newblock Homology operations on infinite loop spaces.
\newblock pages 171--185, 1971.

\bibitem{may-loopspaces}
J.~Peter May.
\newblock {\em The geometry of iterated loop spaces}.
\newblock Springer-Verlag, Berlin, 1972.
\newblock Lecture Notes in Mathematics, Vol. 271.

\bibitem{may-quinn-ray-ringspectra}
J.~Peter May.
\newblock {\em {$E_{\infty }$} ring spaces and {$E_{\infty }$} ring spectra}.
\newblock Lecture Notes in Mathematics, Vol. 577. Springer-Verlag, Berlin-New
  York, 1977.
\newblock With contributions by Frank Quinn, Nigel Ray, and J{\o}rgen
  Tornehave.

\bibitem{mcclure-tate}
J.~E. McClure.
\newblock {$E_\infty$}-ring structures for {T}ate spectra.
\newblock {\em Proc. Amer. Math. Soc.}, 124(6):1917--1922, 1996.

\bibitem{mcduff-segal-groupcompletion}
D.~McDuff and G.~Segal.
\newblock Homology fibrations and the ``group-completion'' theorem.
\newblock {\em Invent. Math.}, 31(3):279--284, 1975/76.

\bibitem{milgram-spherical}
R.~James Milgram.
\newblock The {${\rm mod}\ 2$} spherical characteristic classes.
\newblock {\em Ann. of Math. (2)}, 92:238--261, 1970.

\bibitem{miller-delooping}
Haynes Miller.
\newblock A spectral sequence for the homology of an infinite delooping.
\newblock {\em Pacific J. Math.}, 79(1):139--155, 1978.

\bibitem{miller-sullivanconjecture}
Haynes Miller.
\newblock The {S}ullivan conjecture on maps from classifying spaces.
\newblock {\em Ann. of Math. (2)}, 120(1):39--87, 1984.

\bibitem{milnor-steenrod}
John Milnor.
\newblock The {S}teenrod algebra and its dual.
\newblock {\em Ann. of Math. (2)}, 67:150--171, 1958.

\bibitem{moerdijk-weiss-dendroidal}
Ieke Moerdijk and Ittay Weiss.
\newblock Dendroidal sets.
\newblock {\em Algebr. Geom. Topol.}, 7:1441--1470, 2007.

\bibitem{nassau-secondary}
Christian Nassau.
\newblock On the secondary {S}teenrod algebra.
\newblock {\em New York J. Math.}, 18:679--705, 2012.

\bibitem{ni-bracket}
Xianglong Ni.
\newblock The bracket in the bar spectral sequence for a finite-fold loop
  space.
\newblock {Online at
  https://math.mit.edu/research/undergraduate/spur/documents/2017Ni.pdf}, 2017.

\bibitem{priddy-dyerlashof}
Stewart Priddy.
\newblock Dyer-{L}ashof operations for the classifying spaces of certain matrix
  groups.
\newblock {\em Quart. J. Math. Oxford Ser. (2)}, 26(102):179--193, 1975.

\bibitem{ravenel-wilson-hopfring}
Douglas~C. Ravenel and W.~Stephen Wilson.
\newblock The {H}opf ring for complex cobordism.
\newblock {\em Bull. Amer. Math. Soc.}, 80:1185--1189, 1974.

\bibitem{rezk-power-operations}
Charles Rezk.
\newblock Lectures on power operations.
\newblock available at: http://www.math.uiuc.edu/\~{}rezk/papers.html.

\bibitem{rezk-wilkerson}
Charles Rezk.
\newblock The congruence criterion for power operations in {M}orava
  {$E$}-theory.
\newblock {\em Homology, Homotopy Appl.}, 11(2):327--379, 2009.

\bibitem{richter-ziegenhagen-spectralsequence}
Birgit Richter and Stephanie Ziegenhagen.
\newblock A spectral sequence for the homology of a finite algebraic delooping.
\newblock {\em J. K-Theory}, 13(3):563--599, 2014.

\bibitem{sagave-schlichtkrull-gradedunits}
Steffen Sagave and Christian Schlichtkrull.
\newblock Diagram spaces and symmetric spectra.
\newblock {\em Adv. Math.}, 231(3-4):2116--2193, 2012.

\bibitem{sagave-schlichtkrull-groupcompletion}
Steffen Sagave and Christian Schlichtkrull.
\newblock Group completion and units in {$I$}-spaces.
\newblock {\em Algebr. Geom. Topol.}, 13(2):625--686, 2013.

\bibitem{schlichtkrull-units}
Christian Schlichtkrull.
\newblock Units of ring spectra and their traces in algebraic {$K$}-theory.
\newblock {\em Geom. Topol.}, 8:645--673, 2004.

\bibitem{schwede-gammaspace}
Stefan Schwede.
\newblock Stable homotopical algebra and {$\Gamma$}-spaces.
\newblock {\em Math. Proc. Cambridge Philos. Soc.}, 126(2):329--356, 1999.

\bibitem{schwede-shipley-stablemodules}
Stefan Schwede and Brooke Shipley.
\newblock Stable model categories are categories of modules.
\newblock {\em Topology}, 42(1):103--153, 2003.

\bibitem{shipley-schwede-algebrasmodules}
Stefan Schwede and Brooke~E. Shipley.
\newblock Algebras and modules in monoidal model categories.
\newblock {\em Proc. London Math. Soc. (3)}, 80(2):491--511, 2000.

\bibitem{segal-categoriescohomology}
Graeme Segal.
\newblock Categories and cohomology theories.
\newblock {\em Topology}, 13:293--312, 1974.

\bibitem{senger-bp}
A.~{Senger}.
\newblock {The Brown-Peterson spectrum is not $E\_{2(p^2+2)}$ at odd primes}.
\newblock {\em ArXiv e-prints}, October 2017.

\bibitem{senger-obstr}
Andrew Senger.
\newblock On the realization of truncated {B}rown--{P}eterson spectra as
  {$E_\infty$} ring spectra.
\newblock unpublished.

\bibitem{stacey-whitehouse-hopfring}
Andrew Stacey and Sarah Whitehouse.
\newblock The hunting of the {H}opf ring.
\newblock {\em Homology Homotopy Appl.}, 11(2):75--132, 2009.

\bibitem{stapleton-character}
Nathaniel {Stapleton}.
\newblock {Lubin-Tate theory, character theory, and power operations}.
\newblock {\em arXiv e-prints}, page arXiv:1810.12339, Oct 2018.

\bibitem{stasheff-associahedra}
James~Dillon Stasheff.
\newblock Homotopy associativity of {$H$}-spaces. {I}, {II}.
\newblock {\em Trans. Amer. Math. Soc. 108 (1963), 275-292; ibid.},
  108:293--312, 1963.

\bibitem{steenrod-epstein}
N.~E. Steenrod.
\newblock {\em Cohomology operations}.
\newblock Lectures by N. E. Steenrod written and revised by D. B. A. Epstein.
  Annals of Mathematics Studies, No. 50. Princeton University Press, Princeton,
  N.J., 1962.

\bibitem{strickland-fpfp}
Neil Strickland.
\newblock Functorial philosophy for formal phenomena.
\newblock Draft available at: http://hopf.math.purdue.edu/Strickland/fpfp.pdf.

\bibitem{tilson-kunneth}
Sean Tilson.
\newblock Power operations in the {K}unneth spectral sequence and commutative
  {$HF_p$}-algebras, 2016.

\bibitem{white-commutativemonoids}
David White.
\newblock Model structures on commutative monoids in general model categories.
\newblock ar{X}iv:1403.6759.

\end{thebibliography}

\end{document}